\documentclass[11pt]{article} 
\usepackage{amsmath,amsfonts,amssymb}
\usepackage{amsthm}
\usepackage{xfrac, enumitem}
\usepackage[pdftitle={Limit trees for free group automorphisms: universality}, pdfauthor={Jean Pierre Mutanguha}]{hyperref}

\usepackage{color}

  \theoremstyle{plain}
\newtheorem{thm}{Theorem}[section]
\newtheorem*{thm*}{Theorem}
\newtheorem*{mthm*}{Main Theorem}
\newtheorem{prop}[thm]{Proposition}
\newtheorem{cor}[thm]{Corollary}
\newtheorem*{mcor*}{Main Corollary}
\newtheorem{lem}[thm]{Lemma}
\newtheorem{claim}[thm]{Claim}

\newtheorem*{cor*}{Corollary}
\newtheorem*{obs*}{Observation}
\newtheorem*{sum*}{Summary}
\newtheorem*{claim*}{Claim}
  \theoremstyle{definition}

\newtheorem*{qn*}{Question}
  \theoremstyle{remark}
\newtheorem*{rmk}{Remark}

\makeatletter
\renewcommand*\l@section[2]{%
  \ifnum \c@tocdepth >\z@
    \addpenalty\@secpenalty
    \addvspace{1.0em \@plus\p@}%
    \setlength\@tempdima{2em}%
    \begingroup
      \parindent \z@ \rightskip \@pnumwidth
      \parfillskip -\@pnumwidth
      \leavevmode \bfseries
      \advance\leftskip\@tempdima
      \hskip -\leftskip
      #1\nobreak\hfil
      \nobreak\hb@xt@\@pnumwidth{\hss #2%
                                 \kern-\p@\kern\p@}\par
    \endgroup
  \fi}
\renewcommand*\l@subsection{\@dottedtocline{2}{2em}{2.5em}}
\makeatother

\newcommand{\fakeenv}{} 

\newenvironment{restate}[2]  
{
  \renewcommand{\fakeenv}{#2} 
  \theoremstyle{plain}
  \newtheorem*{\fakeenv}{#1~\ref{#2}} 
  \begin{\fakeenv}
}
{
  \end{\fakeenv}
}

\newcommand*{\defeq}{\mathrel{\vcenter{\baselineskip0.5ex \lineskiplimit0pt \hbox{\scriptsize.}\hbox{\scriptsize.}}}=}

\usepackage[utf8]{inputenc} 

\usepackage{geometry} 

\usepackage{graphicx} 

\title{Limit trees for free group automorphisms: universality}
\author{Jean Pierre Mutanguha\thanks{\emph{Email:} {\tt \href{mailto:jpmutang@ias.edu}{jpmutang@ias.edu}}, \emph{Web address:} {\tt \url{https://mutanguha.com}} \newline Institute for Advanced Study, Princeton, NJ, USA}}

\begin{document}
\maketitle

\begin{abstract} To any free group automorphism, we associate a universal (cone of) limit tree(s) with three defining properties: 
first, the tree has a minimal isometric action of the free group with trivial arc stabilizers;
second, there is a unique expanding dilation of the tree that represents the free group automorphism;
and finally, the loxodromic elements are exactly the elements that weakly limit to dominating attracting laminations under forward iteration by the automorphism.
So the action on the tree detects the automorphism's dominating exponential dynamics.

As a corollary, our previously constructed limit pretree that detects the exponential dynamics is canonical.
We also characterize all very small trees that admit an expanding homothety representing a given automorphism.
In the appendix, we prove a variation of Feighn--Handel's recognition theorem for atoroidal outer automorphisms.
\end{abstract}

\renewcommand{\thefootnote}{\fnsymbol{footnote}} 
\renewcommand{\thefootnote}{\arabic{footnote}}

\section*{Introduction}

We previously constructed a limit pretree that detects the exponential dynamics for an arbitrary free group automorphism~\cite{Mut24}.
In this sequel, we prove the construction is canonical.
This completes the existence and uniqueness theorem for a free group automorphism's limit pretree.
Recall that if we record all the compact geodesics in an $\mathbb R$-tree but forget their lengths, then the resulting structure is a pretree;
briefly, a pretree is a set with a structure that encodes the notion of closed intervals satisfying certain axioms.
Pretrees are the baseline of our constructions; for instance,
($\mathbb R$-)trees will be defined as pretrees with convex metrics, and pseudotrees as pretrees with a certain hierarchy of convex pseudometrics.

In~\cite{Mut24}, we motivated the existence and uniqueness theorem of a limit pretree by describing it as a free group analogue to the Nielsen--Thurston theory for surface homeomorphisms, which in turn can be seen as the surface analogue to the Jordan canonical form for linear maps.
We now give our own motivation for this result.

\subsection*{Universal representation of an endomorphism}

It feels rather odd to discuss my personal motivation while using the communal ``we''; 
excuse me as I break this convention a bit for this section.
In my doctoral thesis, I extended Brinkmann's hyperbolization theorem to mapping tori of free group endomorphisms. 
This required studying the dynamics of endomorphisms.
Along the way, I proved that injective endomorphisms have canonical representatives.
More precisely, suppose $\phi\colon F \to F$ is an injective endomorphism of a finitely generated free group;
then there is:
\begin{enumerate}
\item a minimal simplicial $F$-action on a simplicial tree~$T$ with trivial edge stabilizers;
\item a {$\phi$-equivariant} expanding embedding $f\colon T \to T$ (unique up to isotopy); and
\item an element in~$F$ is $T$-elliptic if and only if one of its forward $\phi$-iterates is conjugate to an element in a {$[\phi]$-periodic} free factor of~$F$.
\end{enumerate}
Existence of the limit free splitting (i.e.~$T$ with its $F$-action) for the {outer class}~$[\phi]$ was the core of my thesis (see also~\cite[Theorem~3.4.5]{Mut21}).
Universality follows from {bounded cancellation}: any other simplicial tree~$T'$  satisfying these three condition will be uniquely equivariantly isomorphic to~$T$~\cite[Proposition~3.4.6]{Mut21}.

In a way, the limit free splitting detects and filters the ``nonsurjective dynamics'' of the (outer) endomorphism. 
When~$\phi\colon F \to F$ is an automorphism, then~$T$ is a singleton and the free splitting provides no new information.
On the other extreme, the $F$-action on~$T$ can be free;
in this case, let $\Gamma \defeq F \backslash T$ be the quotient graph.
Then the outer endomorphism~$[\phi]$ is represented by a unique expanding {immersion} $[f]\colon \Gamma \to \Gamma$ and $[\phi]$~is \emph{expansive} --- such outer endomorphisms are characterized by the absence of $[\phi]$-periodic (conjugacy classes of) nontrivial free factors~\cite[Corollary~3.4.8]{Mut21}.
The most important thing is that the expanding immersion~$[f]$ has nice dynamics and greatly simplifies the study of expansive outer endomorphisms.

After completing my thesis, I found myself in a paradoxical situation: 
I had a better ``understanding'' of nonsurjective endomorphisms than automorphisms --- the main obstacle to studying the dynamics of nonsurjective endomorphisms was understanding the dynamics of automorphisms. 
The na\"ive expectation (when I started my thesis) had been that nonsurjective endomorphisms have more complicated dynamics as they are not invertible.
The current project was born out of an obligation to correct this imbalance.

\subsection*{Universal representation of an automorphism}

What follows is a direct analogue of the above discussion in the setting of automorphisms.
The main theorem of~\cite{Mut24} produces an action that detects and filters the ``exponential'' dynamics of an automorphism.
Specifically, suppose $\phi\colon F \to F$ is an automorphism of a finitely generated free group.
Then there is:
\begin{enumerate}
\item a minimal {rigid} $F$-action on a {real pretree}~$T$ with trivial arc stabilizers;
\item a $\phi$-equivariant ``expanding'' pretree-automorphism $f \colon T \to T$; and
\item an element in~$F$ is $T$-elliptic if and only if it {grows polynomially with respect to~$[\phi]$}.
\end{enumerate}
The pair of the pretree~$T$ and its rigid $F$-action is called a \emph{(forward) limit pretree} for the outer automorphism~$[\phi]$.
The theorem is stated properly in Chapter~\ref{SecTopDomLimit} as Theorem~\ref{thm-limitpseudotree}.
When~$[\phi]$ is {polynomially growing}, then \emph{the} limit pretree is a singleton (and hence unique) but provides no new information.
We are mainly interested in {exponentially growing}~$[\phi]$ as their limit pretrees are not singletons.
On the other hand, the $F$-action on a limit pretree is free if and only if~$[\phi]$ is \emph{atoroidal}, i.e.~there are no $[\phi]$-periodic (conjugacy classes of) nontrivial elements~\cite[Corollary~III.5]{Mut24}.
As with expanding immersions and expansive outer endomorphisms, the expanding ``homeomorphism''~$[f]$ (on the quotient space~$F \backslash T$) has dynamics that could facilitate the study of atoroidal outer automorphisms.

Unlike the endomorphism case, uniqueness of limit pretrees requires a more involved argument. 
It was remarked in the epilogue of~\cite{Mut24} that the only source of indeterminacy in the existence proof was~\cite[Proposition~III.2]{Mut24};
this proposition is restated in Section~\ref{SubsecGrowthLimit} as Proposition~\ref{prop-limittree} and a proof is sketched in Sections~\ref{SubsecLimittrees} and~\ref{SubsecLimittrees2}.
The main result of this paper is a universal version of the proposition.
It can also be thought of as an existence and uniqueness theorem for an action that detects and filters the ``dominating'' exponential dynamics of an outer automorphism:

\begin{mthm*}[Theorems~\ref{thm-dominatinglimitexist}--\ref{thm-dominatingexpandunique}] {~}

Let $\phi \colon F \to F$ be an automorphism of a finitely generated free group and $\{\mathcal A_j^{dom}[\phi]\}_{j=1}^k$ a (possibly empty) subset of $[\phi]$-orbits of {dominating attracting laminations} for~$[\phi]$.

Then there is:
\begin{enumerate}
\item a minimal {factored $F$-tree} $(Y,\Sigma_{j=1}^k \delta_j)$ with trivial arc stabilizers;
\item a unique $\phi$-equivariant {expanding dilation} $f \colon (Y, \Sigma_{j=1}^k \delta_j) \to (Y, \Sigma_{j=1}^k \delta_j)$; and
\item for $1 \le j \le k$, a nontrivial element in~$F$ is \emph{$\delta_j$-loxodromic} if and only if its forward $\phi$-iterates have axes that {weakly limit} to~$\mathcal A_j^{dom}[\phi]$;
\end{enumerate}
moreover, the factored $F$-tree $(Y,\Sigma_{j=1}^k \delta_j)$ is unique up to a unique equivariant dilation.
\end{mthm*}

\noindent Thus the factored tree (up to rescaling of its factors~$\delta_j$) is a universal construction for outer automorphisms of free groups, and we call it \emph{the complete dominating (resp. topmost) tree} if we consider the whole set of orbits of dominating (resp. topmost) attracting laminations. 
As a corollary, the previously constructed limit pretrees are independent of the choices made in the proof of Theorem~\ref{thm-limitpseudotree}, i.e.~\emph{the} limit pretree is canonical (Corollary~\ref{cor-limitpretreeunique}).
Let us now briefly define the emphasized terms in the theorem's statement.

An \emph{$F$-tree} is an ($\mathbb R$-)tree with an isometric $F$-action.
Informally, an $F$-tree is \emph{factored} if its metric has been equivariantly decomposed as a sum $\sum_{j=1}^k \delta_j$ of pseudometrics.
For a factored $F$-tree $(Y,\Sigma_{j=1}^k \delta_j)$, an element in~$F$ is \emph{$\delta_i$-loxodromic} if it is it is $Y$-loxodromic and its axis has positive $\delta_i$-diameter.
An equivariant homeomorphism $(T, \Sigma_{j=1}^k d_j) \to (Y,\Sigma_{j=1}^k \delta_j)$ of factored $F$-trees is a \emph{dilation} if it is a homothety of each pair of \emph{factors}~$d_j$ and~$\delta_j$;
a dilation is \emph{expanding} if each factor-homothety is expanding.

A \emph{lamination} in~$F$ is a nonempty closed subset in the {space of lines} in~$F$.
A sequence of lines (e.g.~axes) \emph{weakly limits} to a lamination if some subsequence converges to the lamination.
Any~$[\phi]$ has a finite set of \emph{attracting laminations} which is empty if and only if~$[\phi]$ is polynomially growing;
this set is partially ordered by inclusion and has an order-preserving $[\phi]$-action.
The maximal elements of the partial order are called \emph{topmost}.
An attracting lamination~$A$ for~$[\phi]$ has an associated \emph{stretch factor}~$\lambda(A)$;
it is \emph{dominating} if any distinct attracting lamination~$A'$ for~$[\phi]$ containing~$A$ has a strictly smaller stretch factor $\lambda(A') < \lambda(A)$.
Topmost attracting laminations are vacuously dominating;
moreover, the $[\phi]$-action permutes the dominating attracting laminations.

\begin{rmk}
If one considers a subset~$\{\mathcal A_j^{top}[\phi]\}_{j=1}^k$ of $[\phi]$-orbits of topmost attracting laminations, then we prove the {topmost tree} has the additional property that its factor-pseudometrics are pairwise \emph{mutually singular}: for each~$i$, there is an element that is $\delta_i$-loxodromic but $\delta_j$-elliptic for $j \neq i$ (see Section~\ref{SubsecTopmost}).
We highlight this feature by using the notation $(Y, \oplus_{j=1}^k \delta_j)$ for topmost trees.
\end{rmk}

\noindent \textbf{Some applications of universal representations.}
Fix an automorphism $\phi \colon F \to F$;
since~$[\phi]$ has a unique equivariant dilation class $[Y, \Sigma_{j=1}^k \delta_j]$ of complete dominating limit trees, any invariant of the class is automatically an invariant of~$[\phi]$.
For instance, the Gaboriau--Levitt index~$i(Y)$ (as defined in~\cite[Chapter~III]{GL95}) is the \emph{dominating forward index for~$[\phi]$}.
In fact, since the limit pretree~$T$ for~$[\phi]$ is canonical, its index~$i(T)$ (defined in~\cite[Appendix~A]{Mut24}) is the \emph{exponential (forward) index for~$[\phi]$};
when~$[\phi]$ is atoroidal, the index~$i(T)$ is closely related to the Gaboriau--Jaeger--Levitt--Lustig index for~$[\phi]$ defined in~\cite[Section~6]{GJLL98}.
Each factor~$\delta_j$ has an associated $F$-tree $(Y_j^{dom}, \delta_j)$;
the pairing of~$\delta_j$ with the orbit of dominating attracting lamination~$\mathcal A_j^{dom}[\phi]$ means~$i(Y_j^{dom})$ is an {index for~$\mathcal A_j^{dom}[\phi]$} respectively.

Our main application is a characterization of minimal $F$-trees with $\phi$-equivariant expanding homotheties:

\begin{mcor*}[Theorem~\ref{thm-unique}]  {~}

Let $\phi \colon F \to F$ be an automorphism and $(Y, \delta)$ a minimal \emph{very small} $F$-tree.
The $F$-tree $(Y, \delta)$ admits a $\phi$-equivariant expanding homothety if and only if it is equivariantly isometric to the dominating tree for~$[\phi]$ with respect to a subset of $[\phi]$-orbits of dominating attracting laminations with the same stretch factor.
\end{mcor*}

In the appendix, we prove a variation of Feighn--Handel's recognition theorem for atoroidal outer automorphisms.


\subsection*{Some historical context}
This paper continues Gaboriau--Levitt--Lustig's philosophy of prioritizing limit trees in their alternative proof of the Scott conjecture~\cite{GLL98}. In particular, our paper relies only on the existence of irreducible train tracks~\cite[Section~1]{BH92} but none of the typical {splitting paths} analysis of relative train tracks \cite{BFH00, FH11}.
Zlil Sela gave another dendrogical proof the conjecture (now Bestvina--Handel's theorem) that used {Rips' theorem} in place of train track technology~\cite{Sel96}.
Fr\'ed\'eric Paulin gave yet another dendrological proof that avoids both train tracks and Rips' theorem~\cite{Pau97}.

About the same time, Bestvina--Fieghn--Handel used train tracks and trees to prove fully irreducible (outer) automorphisms have universal limit trees~\cite{BFH97}.
They used this to give a short dendrological proof of a special case of the Tits Alternative for~$\operatorname{Out}(F)$;
their later proof of the general case was much more involved due to the lack of such a universal limit construction~\cite{BFH00}.
Universal limit trees have been indispensable for studying fully irreducible automorphisms.
In principle, a universal construction of limit trees for all automorphisms would lead to a dendrological proof of the Tits alternative and extend much of the theory for  fully irreducible automorphisms to arbitrary automorphisms.
Speaking of dendrological proofs of the Tits alternative, we mention that Camille Horbez gave such a proof with a very different approach~\cite{Hor14}.

Continuing the work started in~\cite{BFH00}, Feighn--Handel defined and proved the existence of {completely split relative train tracks (CTs)} in~\cite[Section 4]{FH11};
they use CTs to characterize abelian subgroup of $\operatorname{Out}(F)$~\cite{FH09}.
The main obstacle when working with topological representatives is that they are not canonical, which can make defining invariants of the outer automorphism quite technical.
This is the difficulty that we had to deal with in this paper;
however, now that it is done, we can use our new universal representatives to define other invariants rather easily.
A minor inconvenience when working with CTs is that they are only proven to exist for some (uniform) iterate of the outer automorphism;
we were very careful (perhaps to a fault) in this paper to ensure our universal representatives exist for all outer automorphisms. 
Finally, a subtle advantage to our approach is that we find universal representatives for automorphisms and not just outer automorphisms!

\smallskip
In a sequel to~\cite{Sel96}, Sela used limit trees and Rips' theorem to give a canonical hierarchical decomposition of the free group~$F$ that is invariant under a given atoroidal automorphism~\cite{Sel95}.
This second paper was never published and a third announced paper that extends the canonical decomposition to arbitrary automorphisms was never released even as a preprint (as far as we know).
We remark that the limit trees used in that paper were not (or rather, were never proven to be) canonical/universal.
Perhaps, one could combine Sela's canonical decomposition with Bestvina--Feighn--Handel's work to give a universal construction of limit trees for atoroidal automorphisms --- our approach is independent of Sela's work and applies more generally to exponentially growing automorphisms.
Conversely, we suspect that a careful study of the structure of our topmost trees might recover Sela's canonical hierarchical decomposition.

Morgan--Shalen introduced the term ``{$\mathbb R$-trees}'' in~\cite{MS84}.
They also defined ``{$\Lambda$-trees}'' for an ordered abelian group~$\Lambda$.
At first glance, the {hierarchy of pseudometrics} on a real pretree (defined in Section~\ref{SubsecPretrees}) looks like a $\Lambda$-tree.
But paths in our constructed hierarchies ``exit'' {infinitesimal} trees through metric completion points;
whereas paths in a $\Lambda$-tree exit at infinity.
Hierarchies appear to be a new construction to the best of our knowledge.

\subsection*{Proof outline for existence of topmost tree (Theorem~\ref{thm-topmostlimitexist})}

One method for constructing limit trees is iterating \emph{expanding irreducible train tracks}.
This is carried out in Section~\ref{SubsecLimittrees} but it has two drawbacks: exponentially growing automorphisms do not always have expanding irreducible train tracks; and even when they do, the point stabilizers of the corresponding limit tree are not canonical as they can change with the choice of train tracks.
We handle the first obstacle in Section~\ref{SubsecLimittrees2} by constructing a limit tree $(Y_1, \delta_1)$ using a descending sequence of irreducible train tracks, where only the last train track is expanding. 
Such descending sequences always exist for exponentially growing automorphisms.

Next, we construct in Section~\ref{SubsecAssembHier} a pretree with an $F$-action whose point stabilizers are canonical.
Set $G_1 \defeq F$, and let $\mathcal G_2$ be the $[\phi]$-invariant subgroup system determined by the point stabilizers of~$G_1$ acting on~$Y_1$.
By restricting~$[\phi]$ to~$\mathcal G_2$ and inductively repeating the construction, we get a descending sequence of limit forests $(\mathcal Y_i, \delta_i)_{i=1}^n$.
Each limit forest $(\mathcal Y_i, \delta_i)$ has ($[\phi]$-orbits of) attracting laminations~$\mathcal A_i[\phi]$ for~$[\phi]$ that are forward limits of $\mathcal Y_i$-loxodromic elements in~$\mathcal G_i$.
Starting with $X^{(1)} = Y_1$, equivariantly replace the points in~$X^{i}$ fixed by~$\mathcal G_{i+1}$ with the pretrees~$\mathcal Y_{i+1}$ to produce~$X^{(i+1)}$ for $i < n$.
The limit pretree~$T = X^{(n)}$ has canonical point stabilizers: the maximal polynomially growing subgroups.

Everything we have mentioned so far is a rehash of~\cite{Mut24}.
From the blow-up construction, the limit pretree~$T$ inherits an $F$-invariant hierarchy $(\delta_i)_{i=1}^n$ of convex pseudometrics --- the pseudometric~$\delta_i$ is defined on maximal $\mathcal G_i$-invariant convex subsets of~$T$ of $\delta_{i-1}$-diameter~$0$.
The theorem is finally proven in Section~\ref{SubsecTopmost}.
The new insight for this proof: if attracting laminations~$\mathcal A_i[\phi]$ are topmost, then the $\mathcal G_i$-invariant pseudometric~$\delta_i$ can be extended to an $F$-invariant convex pseudometric, still denoted~$\delta_i$, on~$T$.
Let $\{\mathcal A_{\iota(j)}[\phi]\}_{j=1}^k$ be a subset of topmost attracting laminations.
The sum of the corresponding $F$-invariant pseudometrics on~$T$, denoted $\oplus_{j=1}^k \delta_{\iota(j)}$, is an $F$-invariant convex pseudometric on~$T$.
Let~$Y$ be the partition of~$T$ into its maximal subsets of $\oplus_{j=1}^k \delta_{\iota(j)}$-diameter~$0$;
as these subsets are convex, $Y$ inherits a pretree structure from~$T$.
The pseudometric $\oplus_{j=1}^k \delta_{\iota(j)}$ on~$T$ induces a convex metric, also denoted $\oplus_{j=1}^k \delta_{\iota(j)}$, on~$Y$.
The metric space $(Y, \oplus_{j=1}^k \delta_{\iota(j)})$ is our topmost tree. 
This concludes the outline.

\smallskip
At the end of Section~\ref{SubsecDominating}, we prove universality.
The proof relies on Chapter~\ref{SecConverge}: variations of Bestvina--Feighn--Handel's convergence criterion~\cite {BFH97}; 
it boils down to bounded cancellation and Perron--Frobenius theory.

\smallskip
\noindent We use the results of~\cite{Mut24} as black boxes and the two papers can be read in any order.

\bigskip
\noindent \textbf{Acknowledgments.} 
I thank Gilbert Levitt for encouraging me to include the characterization of expanding forests and the referee for improving my exposition.
This material is based upon work supported by the National Science Foundation under Grant No. DMS-1926686 at the Institute for Advanced Study. 

\tableofcontents

\newpage
\section{Preliminaries}\label{SecPrelims}

In this paper,~$F$ denotes a free group with $2 \le \operatorname{rank}(F) < \infty$. 
Subscripts never indicate the rank but instead are used as indices.
For inductive arguments, we also work with a \underline{free group system of finite type}: disjoint union $\bigsqcup_{j \in J} F_j$ of nontrivial finitely generated free groups~$F_j$ indexed by a possibly empty finite set~$J$.
In this paper,~$\mathcal F$ is always a free group system of finite type with some component~$F_j$ that is not cyclic.

\subsection{Group systems and actions}\label{SubsecSystems}

Nearly all statements and results about groups and connected spaces that we are interested in still hold when ``connectivity'' is relaxed and we work with ``systems'' componentwise.
In general (almost categorical) terms, a \underline{system of [?-objects]} is a disjoint union~$\mathcal O = \bigsqcup_{j \in J} O_j$ of [?-objects]~$O_j$ indexed by some set~$J$.
An \underline{[?-isomorphism] of systems}~$\psi \colon \mathcal O \to \mathcal O'$ is a bijection~$\sigma \colon J \to J'$ of the corresponding indexing sets and a union of [?-isomorphisms] $\psi_j \colon O_j \to O_{\sigma \cdot j}'$.
The calligraphic font is reserved for systems.

In more concrete terms, here are some basic concepts that will show up in the paper:
\begin{enumerate}
\item an \underline{isomorphism of group systems} $\psi\colon \mathcal G \to \mathcal G'$ is a bijection whose restriction to any component~$G_j \subset \mathcal G$ is a group isomorphism of components;
for group systems, we always assume (for convenience) components are nontrivial if the system is nonempty.
\item two isomorphisms of group systems $\psi,\psi'\colon \mathcal G \to \mathcal G'$ are in the same \underline{outer class}~$[\psi]$ if the component isomorphisms $\psi_j,\psi_j' \colon G_j \to G_{\sigma \cdot j}'$ differ only by post-composition with an inner automorphism of~$G_{\sigma \cdot j}'$ for all $j \in J$.
\item a \underline{metric on a set system}~$\mathcal X$ is a disjoint union of metrics~$d_j\colon X_j \times X_j \to \mathbb R_{\ge 0}$ on the components~$X_j \subset \mathcal X$.
\item for a group system~$\mathcal G$ indexed by~$J$ and object system~$\mathcal O$ indexed by~$J'$, a \underline{$\mathcal G$-action} on~$\mathcal O$ (or \underline{$\mathcal G$-object system~$\mathcal O$}) is a union of component $G_j$-actions on~$O_{\beta \cdot j}$ for some bijection $\beta\colon J \to J'$.

\item for an automorphism of a group system $\psi \colon \mathcal G \to \mathcal G$ and a $\mathcal G$-object system~$\mathcal O$, the \underline{$\psi$-twisted} $\mathcal G$-object system~$\mathcal O \psi$ is given by precomposing the component $G_{\sigma \cdot j}$-action on~$O_{\beta\sigma \cdot j}$ with the component isomorphism $\psi_j\colon G_j \to G_{\sigma \cdot j}$ to get a $G_j$-object $O_{\beta\sigma \cdot j}$.
\end{enumerate}


\subsection{Pretrees, trees, and hierarchies}\label{SubsecPretrees}

Pretrees are what arises when one wants to discuss ``treelike'' objects without reference to a metric or topology.
In this paper, the pretrees are the ``primitive'' objects and metrics/topologies are additional structures on the pretree --- think of it the same way a Riemannian metric is a compatible addition to a manifold's smooth structure.

Fix a set~$T$;
an \underline{interval function on~$T$} is a function~$[\cdot, \cdot]\colon T \times T \to \mathcal P(T)$, where~$\mathcal P(T)$ is the power set of~$T$, that satisfies the following axioms:
for all $p,q,r \in T$,
\begin{enumerate}
\item \label{axiom-pretree-symmetry} (symmetric) $[p,q] = [q,p]$ contains $\{p, q\}$;
\item \label{axiom-pretree-thin} (thin) $[p,r] \subset [p,q] \cup [q,r]$; and
\item \label{axiom-pretree-linear} (linear) if $r \in [p,q]$ and $q \in [p,r]$, then $q = r$.
\end{enumerate}
\noindent A \underline{pretree} is a pair $(T, [\cdot, \cdot])$ of a nonempty set~$T$ and an interval function $[\cdot, \cdot]$ on~$T$.

The subsets $[p,q] \subset T$ are called \emph{closed intervals} and they should be thought of as the points between~$p$ and~$q$ (inclusive).
We can similarly define \emph{open} (resp. \emph{half-open}) intervals by excluding both (resp. exactly one) of~$\{p,q\}$.
Generally, ``interval'' (with no qualifier) refers to any of the three types of intervals we  have defined.
An interval is \underline{degenerate} if it is empty or a singleton.
We usually omit the interval function and denote a pretree by~$T$.
Note that the real line~$\mathbb R$ is a pretree.

Any subset $S \subset T$ of a pretree inherits an interval function: $[u,v]_S \defeq [u,v] \cap S$ for all $u,v \in S$.
A subset $C \subset T$ is \underline{convex} if~$[p,q] \subset C$ for all $p,q \in C$;
or equivalently, $[\cdot, \cdot]_C$ is the restriction of $[\cdot, \cdot]$ to $C \times C \subset T \times T$.
A \underline{system of pretrees} is a set system~$\mathcal T = \bigsqcup_{j \in J} T_j$ and a disjoint union of interval functions on~$T_j$; 
we call these systems \emph{pretrees} for short.

Let $(T, [\cdot, \cdot])$ and $(T', [\cdot, \cdot]')$ be pretrees.
A \underline{pretree-isomorphism} is a bijection $f\colon T \to T'$ satisfying $f([p,q]) = [f(p), f(q)]'$ for all $p,q \in T$.
Similarly, a \underline{pretree-automorphism} of $(T, [\cdot, \cdot])$ is a pretree-isomorphism $g \colon (T, [\cdot, \cdot]) \to (T, [\cdot, \cdot])$.
A pretree is \underline{real} if its closed intervals are pretree-isomorphic to closed intervals of~$\mathbb R$.
By definition, the real line~$\mathbb R$ is a real pretree.
Note that being real is a property of a pretree, not an added structure like a metric!
An \underline{arc} of a real pretree~$T$ is a nonempty union of an ascending chain of nondegenerate intervals.
A real pretree is \underline{degenerate} if it is a singleton; and a system of real pretrees is \underline{degenerate} if {all} components are degenerate.

\smallskip
Fix a real pretree~$T$; a \underline{convex pseudometric} on~$T$ is a function $d\colon T \times T \to\mathbb R_{\ge 0}$ satisfying the following axioms:
for all $p,q,r \in T$,
\begin{enumerate}
\item \label{axiom-pmetric-symmetry} (symmetric) $d(p,q) = d(q,p)$;
\item \label{axiom-pmetric-convex} (convex) $d(p,r) = d(p,q) + d(q,r)$ if $q \in [p,r]$; and
\item \label{axiom-pmetric-cont} (continuous) $d(p,q) = 2\, d(p,q')$ for some $q' \in [p,q]$.
\end{enumerate}

For any given convex pseudometric~$d$ on~$T$, the preimage~$d^{-1}(0) \subset T \times T$ is an equivalence relation on the real pretree~$T$ such that each equivalence class is convex
 and the set~$T_d$ of equivalence classes inherits a real pretree structure.
A \underline{convex metric} on~$T$ is a convex pseudometric whose equivalence relation~$d^{-1}(0)$ is the equality relation on~$T$.
A \underline{(metric) tree} (or $\mathbb R$-tree) is a real pretree with a convex metric; a \underline{forest} is a system of trees.
For example, the real line~$\mathbb R$ is a tree with the {standard metric} $d_{\mathrm{std}}(p,q) \defeq |p-q|$.
Note that a convex pseudometric~$d$ on a real pretree~$T$ induces a convex metric, still denoted~$d$, on the real pretree~$T_d$;
we refer to the tree $(T_d, d)$ as the \underline{associated tree}.

A \underline{$\lambda$-homothety} of trees $h \colon (T, d) \to (Y, \delta)$ is a pretree-isomorphism $h \colon T \to Y$ that uniformly scales the metric~$d$ by~$\lambda$: 
\[ \delta(h(p),h(q)) = \lambda \, d(p,q) ~ \text{for all} ~ (p,q) \in \operatorname{dom}(d)= T \times T; \]
equivalently, $h^*\delta = \lambda d$, where $h^*\delta$ is the \emph{pullback of~$\delta$ via~$h$}.
A \underline{homothety} is a $\lambda$-homothety for some $\lambda > 0$;
it is \underline{expanding} (resp. an \underline{isometry}) if $\lambda > 1$ (resp. $\lambda = 1$).
An isometry $\iota \colon (T,d) \to (T, d)$ is \underline{elliptic} if it fixes a point of~$T$; 
otherwise, it is \underline{loxodromic} and acts by a nontrivial translation on its \emph{axis}, the unique $\iota$-invariant arc of $(T, d)$ isometric to $(\mathbb R,d_{\mathrm{std}})$;
the \underline{translation distance}~$\|\iota\|_d \in \mathbb R_{\ge 0}$ is~$0$ if~$\iota$ is elliptic and equal to the displacement of points in $\iota$'s axis if~$\iota$ is loxodromic.
These definitions extend componentwise to forests.

\smallskip
Let~$d_1$ be a nonconstant convex pseudometric on~$T$ and $d_{i+1} \colon d_i^{-1}(0) \to \mathbb R_{\ge 0}$ a nonconstant disjoint union of convex pseudometrics for $1 \le i < n$. 
The sequence~$(d_i)_{i=1}^n$ will be known as an \underline{$n$-level hierarchy} of convex pseudometrics on~$T$; 
We will say just \emph{hierarchies} for short.
A hierarchy~$(d_i)_{i=1}^n$ has \underline{full support} if~$d_n$ is a disjoint union of convex metrics.
A \underline{pseudotree} is a pair $(T, (d_i)_{i=1}^n)$ of a real pretree and a hierarchy with full support; a \underline{pseudoforest} is a system of pseudotrees.
A \underline{$(\lambda_i)_{i=1}^n$-homothety} of $n$-level pseudoforests $h \colon (\mathcal T, (d_i)_{i=1}^n) \to (\mathcal Y, (\delta_i)_{i=1}^n)$ is a system of pretree-isomorphisms $h \colon \mathcal T \to \mathcal Y$ that scales each pseudometric~$d_i$ by~$\lambda_i$: 
\[ \delta_i(h(p),h(q)) = \lambda_i \, d_i(p,q) ~ \text{for all} ~ i \ge 1 ~ \text{and} ~ (p,q) \in \operatorname{dom}(d_i);\]
a \underline{homothety} is a $(\lambda_i)_{i=1}^n$-homothety for some $\lambda_i > 0$; it is \underline{expanding} (resp. \underline{isometry}) if each $\lambda_i > 1$ (resp. each $\lambda_i = 1$).
As with trees, an isometry of a pseudotree is either \underline{elliptic} (fixes a point) or \underline{loxodromic} (translates a ``pseudoaxes'').
Hierarchies and pseudoforests are the fundamental (perhaps novel) tool in this paper.
They are first used in Chapter~\ref{SecTopDomLimit}.

\subsection{Simplicial actions and train tracks}\label{SubsecTrains}

For a pretree~$T$, a \underline{direction} at $p \in T$ is a maximal subset~$D_p \subset T \setminus \{ p \}$ not separated by~$p$, i.e.~$p \notin [q,r]$ for all $q,r \in D_p$.
A \underline{branch point} is a point with at least three directions, and a \underline{branch} is a direction at a branch point.
An \underline{endpoint} is a point with at most one direction.
A \emph{simple pretree} is a pretree whose closed intervals are finite subsets.
A pretree~$T$ is \underline{simplicial} if it is real, its subset~$V$ of branch points and endpoints is a simple pretree, and no convex proper subset contains~$V$;
a \underline{vertex} is a point in~$V$.
An \underline{(open) edge} in a simplicial pretree~$T$ is a maximal convex subset $e \subset T$ that contains no vertex.
By construction, edges are open intervals; the corresponding closed intervals in~$T$ are called \underline{closed edges}.

\begin{rmk} Being simplicial is a property of a pretree, not an added structure!
Besides that, our definition of a simplicial pretree is more general (with one exception) than the standard definition of a \emph{simplicial tree} and has the advantage that it is independent of any choice of metric/topology. 
See~\cite[Interlude]{Mut24} for a discussion on this distinction. 
The one exception: \underline{the real line~$\mathbb R$ is not a simplicial pretree!} 
\end{rmk}

An \emph{$F$-pretree} is a pretree with an $F$-action by pretree-automorphisms.
An \emph{$F$-pseudotree} is a pair of a real $F$-pretree and an {$F$-invariant} hierarchy with full support;
equivalently, an $F$-pseudotree (resp. $F$-tree) is a pseudotree with an isometric $F$-action.
An $F$-pseudotree or $F$-tree is \underline{minimal} if the underlying $F$-pretree has no proper nonempty $F$-invariant convex subset;
in this case, the underlying $F$-pretree has no endpoints.
We mostly consider minimal $F$-pseudotrees with trivial arc (pointwise) stabilizers.

Suppose an $F$-pseudotree $(T, (d_i)_{i=1}^n)$ has trivial arc stabilizers.
For any nontrivial subgroup $G \le F$, the \underline{characteristic convex subset} (of~$T$) for~$G$ is the unique minimal nonempty $G$-invariant convex subset $T(G) \subset T$.
In an $F$-tree $(T, d)$ with trivial arc stabilizers, the restriction of~$d$ to~$T(G)$ is a $G$-invariant convex metric, still denoted~$d$;
the minimal $G$-tree $(T(G), d)$ is the \underline{characteristic subtree} (of $(T, d)$) for~$G$. 

\begin{rmk}We do not really need an isometric action to define characteristic convex subsets and minimality.
All we need is the $F$-action on the real pretree~$T$ to be \emph{rigid/non-nesting}: no closed interval is sent properly into itself by the $F$-action~\cite[Section~II.2]{Mut24}.
While rigid actions are central to~\cite{Mut24}, they are superseded by isometric actions in this paper.
\end{rmk}

An $F$-pretree~$T$ is \underline{simplicial} if~$T$ is simplicial and admits an $F$-invariant convex metric~$d$;
equivalently, a simplicial $F$-pretree is a simplicial pretree with a rigid $F$-action.
Any simplicial $F$-pretree has an open cone (over a finite dimensional open simplex) worth of $F$-invariant convex metrics (up to an equivariant isometry {isotopic} to the identity map).
The definitions given so far extend componentwise to systems.

\smallskip
Let~$\mathcal T$ and~$\mathcal T'$ be simplicial pretrees and $f\colon\mathcal T \to \mathcal T'$ a \emph{tight cellular map}, i.e.~a function that maps vertices to vertices and the restriction to any closed edge is a \emph{pretree-embedding}, i.e.~a pretree-isomorphism onto its image.
For any choice of convex metrics~$d, d'$ on~$\mathcal T, \mathcal T'$ respectively, there is a unique map $(\mathcal T, d) \to (\mathcal T', d')$ that is {linear} on edges and {isotopic} to~$f$;
whenever a choice of convex metrics is made, we implicitly replace~$f$ with this map.

Let~$\mathcal T$ be a \underline{free splitting} of~$\mathcal F$, i.e.~minimal simplicial $\mathcal F$-pretrees with trivial edge stabilizers,
and suppose $\psi \colon \mathcal F \to \mathcal F$ is an automorphism of a free group system.
The \emph{$\psi$-twisted} free splitting~$\mathcal T\psi$ is the same real pretrees~$\mathcal T$ but the original simplicial $\mathcal F$-action is precomposed with~$\psi$.
A \underline{(relative) topological representative} for~$\psi$ is a \emph{$\psi$-equivariant} tight cellular map~$f\colon \mathcal T \to \mathcal T$ on a nondegenerate free splitting~$\mathcal T$ of~$\mathcal F$: $\psi$-equivariance means $f(x \cdot p) = \psi(x) \cdot f(p)$ for all $x \in \mathcal F$ and $p \in \mathcal T$,
or equivalently, $f\colon \mathcal T \to \mathcal T \psi$ is equivariant.
Given a topological representative $f \colon \mathcal T \to \mathcal T$ for~$\psi$, we let~$[f]$ denote the induced map on the quotient $\mathcal F \backslash \mathcal T$;
we say $[f]$ is a \emph{topological representative} for the outer class~$[\psi]$.
A \underline{(relative) train track} for~$\psi$ is a topological representative~$\tau\colon \mathcal T \to \mathcal T$ for~$\psi$ whose iterates~$\tau^m~(m \ge 1)$ are topological representatives for~$\phi^m$; or equivalently, whose iterates~$\tau^m$ restrict to pretree-embeddings on closed edges.

For any free splitting~$\mathcal T$ of~$\mathcal F$, Bass-Serre theory gives a uniform bound on the number of $\mathcal F$-orbits of edges (linear in $\operatorname{rank}(\mathcal F)$) and relates the vertices with nontrivial stabilizers in a (componentwise) connected fundamental domain to a (possibly empty) \emph{free factor system}~$\mathcal F[\mathcal T]$ of~$\mathcal F$ --- take this as the working definition of {free factor systems}. 
The theory also gives a uniform bound on the {complexity} (e.g.~ranks) of free factor systems.
A free factor system~$\mathcal F[\mathcal T]$ \emph{proper} if~$\mathcal F[\mathcal T] \neq \mathcal F$;
equivalently, $\mathcal F[\mathcal T]$ is proper if and only if~$\mathcal T$ is not degenerate.
Any proper free factor system of~$\mathcal F$ has strictly lower complexity than~$\mathcal F$.
The \underline{trivial} free factor system of~$\mathcal F$ is the (possibly empty) free factor system consisting of the cyclic $\mathcal F$-components;
it is always proper since we assume~$\mathcal F$ has a noncyclic component.

\begin{rmk} We will abuse notation and write \underline{$\mathcal F[\mathcal T] = \mathcal F[\mathcal T']$} for two free splittings~$\mathcal T, \mathcal T'$ of~$\mathcal F$ when we mean: an element of~$\mathcal F$ is $\mathcal T$-elliptic if and only if it is $\mathcal T'$-elliptic.
\end{rmk}

Fix an automorphism $\psi\colon \mathcal F \to \mathcal F$ and a topological representative $f\colon \mathcal T \to \mathcal T$ for~$\psi$.
By $\psi$-equivariance of~$f$, the proper free factor system~$\mathcal F[\mathcal T]$ is \emph{$[\psi]$-invariant} --- again, we can take this as the definition of $[\psi]$-invariance for proper free factor systems. Form a nonnegative integer square matrix~$A[f]$ whose rows and columns are indexed by the $\mathcal F$-orbits of edges in~$\mathcal T$; and the entry at row-$[e]$ and column-$[e']$ is given by the number of $e$-translates in the interval~$f(e')$, where $e, e'$ are edges in~$T$.
The topological representative~$f$ is \underline{irreducible} if the matrix~$A[f]$ is {irreducible}; 
or equivalently, if, for any pair of edges $e, e'$ in~$\mathcal T$, a translate of~$e$ is contained~$f^m(e')$ for some~$m = m(e,e') \ge 1$.
It is a foundational theorem of Bestvina--Handel that automorphisms have irreducible train tracks.

\begin{thm}[cf.~{\cite[Theorem~1.7]{BH92}}]\label{thm-irredtt} Let $\psi\colon \mathcal F \to \mathcal F$ be an automorphism of a free group system and $\mathcal Z$ a $[\psi]$-invariant proper free factor system of~$\mathcal F$. 
Then there is an irreducible train track~$\tau\colon \mathcal T \to \mathcal T$ for~$\psi$, where the components of~$\mathcal Z$ are $\mathcal T$-elliptic.
\end{thm}

\noindent The proof outline of~\cite[Theorem~I.1]{Mut24} explains how to deduce the theorem as currently stated from the cited theorem.

\smallskip
Suppose $\psi\colon \mathcal F \to \mathcal F$ is an automorphism with an irreducible topological representative  $f\colon \mathcal T \to \mathcal T$.
Perron--Frobenius theory implies the matrix~$A[f]$ has a unique real eigenvalue~$\lambda = \lambda[f] \ge 1$ with a unique positive left eigenvector~$\nu[f]$ whose entries sum to~$1$.
From the eigenvector~$\nu[f]$, we get an $\mathcal F$-invariant convex metric~$d_f$ on~$\mathcal T$ (well-defined up to an equivariant isometry isotopic to the identity map).
The restriction of~$f$ to any edge is a {$\lambda$-homothetic embedding} with respect to~$d_f$;
the metric~$d_f$ is the \underline{eigenmetric} (on~$\mathcal T$) for~$[f]$.
If $\lambda =  1$, then~$f$ is a $\psi$-equivariant {simplicial automorphism} of~$\mathcal T$.

\subsection{Growth types and limit trees}\label{SubsecGrowthLimit}

Since the introduction of train tracks, it has been standard to construct {limit forests} by iterating an expanding irreducible train track (Section~\ref{SubsecLimittrees}).
Unfortunately, such a construction is not canonical as it can depend on the initial train track.
The main idea of the paper: patch together a ``descending'' sequence of limit trees to get a limit pseudoforest and inductively ``normalize'' its hierarchy into a canonical limit pseudoforest.

Fix a free group system~$\mathcal G$ of finite type (unlike~$\mathcal F$, all components of $\mathcal G$ can be cyclic), an automorphism $\psi \colon \mathcal G \to \mathcal G$, and a metric free splitting $(\mathcal T,d)$ of~$\mathcal G$ whose free factor system $\mathcal Z \defeq \mathcal F[\mathcal T]$ is $[\psi]$-invariant.
An element $x \in \mathcal G$ \underline{$[\psi]$-grows exponentially} rel.~$d$ with \underline{rate $\lambda_x$} if it is $\mathcal T$-loxodromic and the limit inferior of the sequence $\left(m^{-1} \log \| \psi^m(x) \|_d\right)_{m \ge 0}$ is $\log \lambda_x > 0$.
If an element $[\psi]$-grows exponentially rel.~$d$, then it $[\psi]$-grows exponentially rel.~$d'$ with the same rate for any metric free splitting~$(\mathcal T',d')$ of~$\mathcal G$ with $\mathcal F[\mathcal T'] = \mathcal Z$;
say the element \underline{$[\psi]$-grows exponentially} rel.~$\mathcal Z$.
An element $x\in \mathcal G$ \underline{$[\psi]$-grows polynomially} rel.~$\mathcal Z$ with \underline{degree~$< n$} if the sequence $\left(m^{-n} \| \psi^m(x)\|_d\right)_{n \ge 0}$ converges to~$0$.
Any element of~$\mathcal G$ $[\psi]$-grows either exponentially or polynomially rel.~$\mathcal Z$ \cite[Corollary~III.4]{Mut24}.
The {growth type} of an element is preserved when passing to invariant subgroup systems of finite type.

The automorphism~$\psi$ is \underline{exponentially growing} rel.~$\mathcal Z$ if some element $[\psi]$-grows exponentially rel.~$\mathcal Z$;
otherwise,~$\psi$ is \underline{polynomially growing} rel.~$\mathcal Z$.
The growth type of an outer class~$[\psi]$ is also well-defined.
The ``rel.~$\mathcal Z$'' in our terminology may be omitted when $\mathcal Z$ is trivial.
The next proposition deals with the first obstacle:

\begin{prop}[cf.~{\cite[Proposition~III.2]{Mut24}}]\label{prop-limittree} Let $\psi \colon \mathcal F \to \mathcal F$ be an automorphism of a free group system and $\mathcal Z$ a $[\psi]$-invariant proper free factor system.
Then there is a:
\begin{enumerate}
\item a minimal $\mathcal F$-forest $(\mathcal Y, \delta)$ with trivial arc stabilizers for which~$\mathcal Z$ is elliptic; and
\item a unique $\psi$-equivariant expanding homothety $h \colon (\mathcal Y, \delta) \to (\mathcal Y, \delta)$.
\end{enumerate}
The forest $(\mathcal Y, \delta)$ is degenerate if and only if~$[\psi]$ is polynomially growing rel.~$\mathcal Z$.
\end{prop}

\noindent 
The constructed $\mathcal F$-forest $(\mathcal Y, \delta)$ is the \emph{limit forest for~$[\psi]$ rel.~$\mathcal Z'$}, for some $[\psi]$-invariant proper free factor system~$\mathcal Z'$ that supports~$\mathcal Z$ (see Sections~\ref{SubsecLimittrees} and~\ref{SubsecLimittrees2}).
Unfortunately, these limit forests depend on the choice of~$\mathcal Z'$;
our goal is to give a canonical contruction.

Given the central tool (hierarchies) and objective (universal limit trees), we outline again how these two fit together.
Gaboriau--Levitt's index theory~\cite{GL95} gives a uniform bound on the complexity of the {point stabilizers system}~$\mathcal G[\mathcal Y]$ for a minimal $\mathcal F$-forest $(\mathcal Y,\delta)$ with trivial arc stabilizers --- this is a partial generalization of Bass--Serre theory.
When~$\mathcal Y$ is not degenerate, the subgroup system~$\mathcal G[\mathcal Y]$ has strictly lower {complexity} than~$\mathcal F$.
This allows us to induct on complexity (see Chapter~\ref{SecTopDomLimit}).

\smallskip
Suppose the automorphism $\psi \colon \mathcal F \to \mathcal F$ has a nondegenerate limit forest $(\mathcal Y_1,\delta_1)$ with nontrivial point stabilizers;
the system of stabilizers $\mathcal G \defeq \mathcal G[\mathcal Y]$ has strictly smaller complexity than~$\mathcal F$.
By $\psi$-equivariance of $\lambda_1$-homothety $h_1\colon (\mathcal Y_1, \delta_1) \to (\mathcal Y_1, \delta_1)$, the $\mathcal F$-orbits of points with nontrivial stabilizers are permuted by~$[h_1]$, the subgroup system~$\mathcal G$ is {$[\psi]$-invariant}, and the restriction of~$\psi$ to~$\mathcal G$ determines a unique outer automorphism~$[\varphi]$ of~$\mathcal G$.

Suppose $\varphi \colon \mathcal G \to \mathcal G$ has a nondegenerate limit forest $(\mathcal Y_2, \delta_2)$ with stretch factor~$\lambda_2$. 
Using the blow-up construction from~\cite{Mut24}, we equivariantly blow-up~$\mathcal Y_1$ with respect to $h_i \colon \mathcal Y_i \to \mathcal Y_i ~ (i = 1,2)$ to get real pretrees~$\mathcal T$ with a minimal rigid $\mathcal F$-action and a $\psi$-equivariant ``$\mathcal F$-expanding'' pretree-isomorphism~$f\colon \mathcal T \to \mathcal T$ induced by~$h_1$ and~$h_2$.
In fact, the blow-up construction implies the $\mathcal F$-pretrees~$\mathcal T$ inherit an $\mathcal F$-invariant $2$-level hierarchy $(\delta_1, \delta_2)$ with full support and~$f$ is an expanding homothety with respect to this hierarchy.
So we have a \emph{limit pseudoforest} $(\mathcal T, (\delta_1,\delta_2))$ for~$[\psi]$ (see Section~\ref{SubsecAssembHier}).
Under what conditions can we construct an $\mathcal F$-invariant convex metric on~$\mathcal T$ from $(\delta_1, \delta_2)$?
The heart of the paper is the following observation: the two limit forests $(\mathcal Y_i, \delta_i)$ are paired with {attracting laminations}~$\mathcal L_i^+[\psi]$ partially ordered by inclusion;
an $\mathcal F$-invariant convex metric on~$\mathcal T$ can be naturally constructed from $(\delta_1, \delta_2)$ if $\mathcal L_2^+[\psi]$ is not in $\mathcal L_1^+[\psi]$ (see Section~\ref{SubsecTopmost}) or $\lambda_1 < \lambda_2$ (see Section~\ref{SubsecDominating})!

\subsection{Bounded cancellation and laminations}\label{SubsecBCandLines}

Suppose a minimal $\mathcal F$-forest $(\mathcal Y, \delta)$ is \underline{very small}, i.e.~nontrivial arc stabilizers are maximal cyclic subgroups and the fixed point subset for a nontrivial elliptic element is an arc.
Let~$(\mathcal T, d)$ be a metric free splitting of~$\mathcal F$ and $[\cdot,\cdot]_T$ (resp. $[\cdot, \cdot]_Y$)  denote the interval function for $\mathcal T$ (resp. $\mathcal Y$).
A map $f \colon (\mathcal T,d) \to (\mathcal Y, \delta)$ is \underline{piecewise-linear (PL)} if the restriction to any closed edge is a linear map; 
an equivariant PL-map exists if and only if $\mathcal T$-elliptic elements in~$\mathcal F$ are $\mathcal Y$-elliptic.
Equivariant PL-maps $(\mathcal T,d) \to (\mathcal Y, \delta)$ are surjective and Lipschitz since the isometric $\mathcal F$-action on $(\mathcal Y, \delta)$ is minimal and there are only finitely many $\mathcal F$-orbits of edges in~$\mathcal T$;
$1$-Lipschitz maps are also known as \underline{metric maps}.
Generally, if~$\mathcal T, \mathcal Y$ are free splittings of~$\mathcal F$, then an equivariant function $f \colon \mathcal T \to \mathcal Y$ is a \underline{(simplicial) PL-map} if its restrictions to any closed edge is isotopic to a linear map with respect to some/any $\mathcal F$-invariant convex metrics~$d, \delta$ on~$\mathcal T, \mathcal Y$ respectively.

\begin{lem}[bounded cancellation]\label{lem-bcl}
Let $f\colon (\mathcal T,d) \to (\mathcal Y, \delta)$ be an equivariant PL-map.
For some constant $C[f] \ge 0$ and all points $p, q \in \mathcal T$, the image $f([p,q]_T)$ is in the $C[f]$-neighbourhood of the interval $[f(p), f(q)]_{Y}$. 
\end{lem}
\noindent 
Such a~$C[f]$ is a \underline{cancellation constant} for~$f$.
This proof is due to Bestvina--Feighn--Handel.
\begin{proof}[Sketch of proof~{\cite[Lemma~3.1]{BFH97}}]
Let $\operatorname{Lip}(f)$ be the Lipschitz constant and $\operatorname{vol}(\mathcal T, d)$ the {volume$\pmod{\mathcal F}$}.
Then $f = g \circ h$ for some equivariant $\operatorname{Lip}(f)$-homothety~$h$ and metric PL-map~$g$.
Suppose~$f$ is \emph{simple}: its target is a metric free splitting with free factor system $\mathcal F[\mathcal T]$.
Then~$g$ factors as finitely many equivariant edge collapses and Stallings folds followed by an equivariant metric homeomorphism.
The homeomorphism and each edge collapse have cancellation constants 0.
A fold has a cancellation constant given by the length of folded segment.
Finally, the metric PL-map~$g$ has a cancellation constant since cancellation constants are (sub)additive over compositions of metric maps.
As cancellation constants are preserved by precomposition with homeomorphisms, the PL-map~$f = g \circ h$ has a cancellation constant $C[f] < \operatorname{Lip}(f)\operatorname{vol}(\mathcal T, d)$.

Otherwise, the PL-map~$f$ is not simple.
For a contradiction, suppose the image $f([p,q]_T)$ is not in the $\operatorname{Lip}(f)\operatorname{vol}(\mathcal T, d)$-neighbourhood of $[f(p), f(q)]_Y$ for some $p, q \in \mathcal T$.
Let $\delta(f(r_0), [f(p), f(q)]_Y) > \operatorname{Lip}(f)\operatorname{vol}(\mathcal T, d) + \epsilon_0$ for some $\epsilon_0 > 0$ and point $r_0 \in [p,q]_T$.
For any $\epsilon > 0$, the PL-map~$f$ is {approximated} by an equivariant simple PL-map~$f_\epsilon$ with $\operatorname{Lip}(f_\epsilon) < \operatorname{Lip}(f) + \epsilon$ and $C[f_\epsilon] \ge \operatorname{Lip}(f)\operatorname{vol}(\mathcal T, d) + \epsilon_0$ (see~\cite[Theorem~6.1]{Hor17}).
By the previous paragraph, $C[f_\epsilon] < \operatorname{Lip}(f_\epsilon)\operatorname{vol}(\mathcal T, d)$ for $\epsilon > 0$.
So $C[f_\epsilon] < \operatorname{Lip}(f)\operatorname{vol}(\mathcal T, d) +\epsilon_0$ for small enough $\epsilon > 0$ --- a contradiction.
\end{proof}
\begin{rmk}The results in this section apply to $\psi$-equivariant PL-maps $g \colon (\mathcal T, d) \to (\mathcal T, d)$ for any automorphism $\psi \colon \mathcal F \to \mathcal F$: view~$g$ as an equivariant PL-map $(\mathcal T, d) \to (\mathcal T\psi, d)$ instead.
\end{rmk}

A \underline{line} in a forest is an arc that is isometric to~$(\mathbb R, d_{\mathrm{std}})$;
a \underline{ray} in a forest is an arc that is isometric to $(\mathbb R_{\ge 0}, d_{\mathrm{std}})$ and its \underline{origin} is the point corresponding to~$0$ under the isometry.
Two rays are \emph{end-equivalent} if their intersection is a ray;
an \underline{end} of a forest is an end-equivalence class of rays in the forest.
Note that there is a natural bijection between the set of lines in a forest and set of unordered pairs of distinct ends in the same component of the forest.
For simplicial $\mathcal F$-pretrees~$\mathcal T$, the notions of line/ray/end are well-defined for the cone of $\mathcal F$-invariant convex metrics on~$\mathcal T$.

\begin{cor}[{cf.~\cite[Lemma~3.4]{GJLL98}}]\label{cor-tightenings} {~}

Let $f\colon (\mathcal T,d) \to (\mathcal Y, \delta)$ be an equivariant PL-map.
\begin{enumerate}
\item\label{cor-tightenings-dichotomy} For any ray~$\rho$ in $(\mathcal T, d)$ with origin~$p_0$, the image~$f(\rho)$ is either bounded or in the $C[f]$-neighbourhood of a unique ray $f_*(\rho) \subset f(\rho)$ with origin~$f(p_0)$;
moverover, if $\rho, \rho'$ represent the same end~$e$ and $f(\rho)$ is unbounded, then so is $f(\rho')$ and $f_*(\rho), f_*(\rho')$ are end-equivalent --- denote their end-equivalence class by~$f_*(e)$.
\item\label{cor-tightenings-lineproj} For any line~$l$ in $(\mathcal T, d)$,~$f(l)$ is in a $C[f]$-neighbourhood of a unique line~$f_*(l) \subset f(l)$ if both ends of~$l$ have unbounded $f$-images.
\item\label{cor-tightenings-preimage} For any end~$\epsilon$ of $(\mathcal Y, \delta)$, there is a unique end~$f^*(\epsilon)$ of $(\mathcal T, d)$ with $\epsilon = f_*(f^*(\epsilon))$.
\end{enumerate}
\end{cor}

\begin{proof}[Sketch of proof] {~}

\noindent (\ref{cor-tightenings-dichotomy}): Let $\rho$ be a ray in $(\mathcal T,d)$, $p_{0} \in \rho$ its origin, $f(\rho)$ unbounded, and $s_0 = f(p_0)$.
Use Figure~\ref{fig - ray} for reference.
By bounded cancellation and the Lipschitz property,~$f(\rho)$ has at most one end of $(\mathcal Y, \delta)$.
For some $n \ge 0$, assume $s_{n} \in [s_0, f(p)]_Y$ for all $p \in \rho \setminus [p_0, p_{n}]_T$.
Set $C \defeq C[f]$.
Since $f(\rho)$ is unbounded, \[ \delta(s_0, f(p_{n+1})) > 2\, \delta(s_0, s_{n}) + C\] for some $p_{n+1} \in \rho \setminus [p_{0},p_{n}]_T$.
Pick $s_{n+1} \in [s_0, f(p_{n+1})]_Y$ satisfying $\delta(s_0, s_{n+1}) > 2\,\delta(s_0, s_{n})$ and $\delta(s_{n+1},f(p_{n+1})) > C$;
so $s_{n} \in [s_0, s_{n+1}]_Y$.
By bounded cancellation, the interval $[s_0,s_{n+1}]_Y \subset f([p_0, p_{n+1}]_T)$ is disjoint from $f(\rho \setminus [p_0,p_{n+1}]_T)$.
So the union $\bigcup_{n \ge 0} [s_0, s_n]_Y$ is a ray~$f_*(\rho)$ in~$f(\rho)$ with origin~$s_0$.
By construction,~$f(\rho)$ is in the $C$-neighbourhood of~$f_*(\rho)$.
Any bounded neighbourhood of a ray contains a unique ray, up to end-equivalence.

\begin{figure}[ht]
 \centering 
 \includegraphics{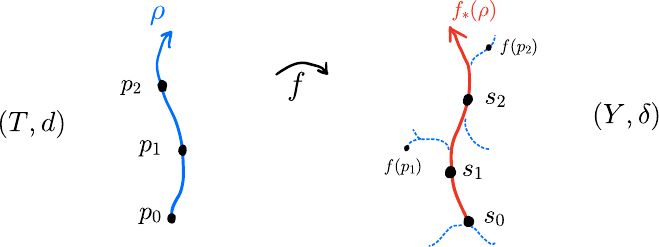}
 \caption{The ray $f_*(\rho)$ with origin $s_0 = f(p_0)$ is built inductively in the image $f(\rho)$.}
 \label{fig - ray}
\end{figure}

\smallskip
\noindent (\ref{cor-tightenings-lineproj}): Represent both ends of~$l$ with rays $\rho^\pm \subset l$ sharing the same origin.
By Part~\ref{cor-tightenings-dichotomy} and bounded cancellation, we have rays~$f_*(\rho^\pm)$ representing unique distinct ends~$\epsilon^\pm$ of $(\mathcal Y, \delta)$;
moreover, $f(l)=f(\rho^-)\cup f(\rho^+)$ is in the $C$-neighbourhood of $f_*(\rho^-) \cup f_*(\rho^+) \subset f(l)$.
Let $f_*(l) \subset f_*(\rho^-) \cup f_*(\rho^+)$ be the line determined by the ends~$\epsilon^\pm$.
Then~$f(l)$ is in the $C$-neighbourhood of~$f_*(l)$.
Any bounded neighbourhood of a line contains a unique line.

\smallskip
\noindent (\ref{cor-tightenings-preimage}): The argument is almost the same.
Let~$\rho'$ be a ray representing~$\epsilon$ and $s_0 = q_0$ its origin.
Pick points $q_{n+1}, s_{n+1} \in \rho'$ with $\delta(s_0, s_{n+1}) > 2\,\delta(s_0, s_{n})$, $\delta(s_0, q_{n+1}) > 2\, \delta(s_0, s_n) + C$, and $\delta(s_{n+1},q_{n+1}) > C$.
Since $f\colon \mathcal T \to \mathcal Y$ is surjective, we can lift~$q_n$ to $p_n \in T$.
By bounded cancellation and $K$-Lipschitz property, the distance $d(p_0, [p_n, p_{n+m}]_T) > \frac{1}{K} \delta(s_0,s_n)$.
Thus $(p_n)_{n\ge 0}$ determines an end~$e$ of~$(\mathcal T,d)$ with unbounded $f$-image.
Let~$\rho$ be a ray representing~$e$ with origin~$p_0$.
As $\rho' \subset f(\rho)$ by construction, we get $f_*(\rho) = \rho'$ by Part~\ref{cor-tightenings-dichotomy}. By Part~\ref{cor-tightenings-lineproj}, the end~$e$ is the unique end with $f_*(e) = \epsilon$, and we denote it by~$f^*(\epsilon)$.
\end{proof}

The corollary defines the equivariant \underline{lifting} (resp. \underline{projecting}) function~$f^*$ (resp.~$f_*$), where the domain~$\operatorname{dom}(f^*)$ of~$f^*$ is the set of lines in $(\mathcal Y, \delta)$ and the domain~$\operatorname{dom}(f_*)$ of~$f_*$ is the set of lines in $(\mathcal T, d)$ whose ends both have unbounded $f$-images.
Note that the image~$\operatorname{im}(f^*)$ is~$\operatorname{dom}(f_*)$;
moreover,~$f^*$ and~$f_*$ are inverses of each other.
Both lifting and projecting functions are \emph{canonical}:
$f^* = g^*$ and $f_* = g_*$ for any equivariant maps $f,g \colon (\mathcal T, d) \to (\mathcal Y, \delta)$ since $f,g$ will be a bounded $\delta$-distance from some equivariant PL-map;
for lack of better notation, we still denote the functions by~$f^*, f_*$ despite this independence.

Alternatively, we view~$f^*$ and~$f_*$ as functions on the sets of $\mathcal F$-orbits of lines.
We can equip these sets with a natural topology.
The set $\mathbb R(\mathcal Y, \delta)$ of $\mathcal F$-orbits of lines in $(\mathcal Y, \delta)$ has the following topology: for any $p,q \in \mathcal Y$, let $U[p,q]$ be the $\mathcal F$-orbit of lines that contain a translate of $[p,q]$;
the collection $\{ U[p,q]:p,q \in \mathcal Y \}$ is a basis for the \underline{space of ($\mathcal F$-orbits of) lines}.
This space is well-defined for the equivariant homothetic class of $(\mathcal Y, \delta)$.
The space of lines is also well-defined for the free splitting~$\mathcal T$ and denoted $\mathbb R(\mathcal T)$.

\begin{claim}\label{claim-liftembed}
The canonical lifting function $f^* \colon \mathbb R(\mathcal Y, \delta) \to \mathbb R(\mathcal T)$ is a topological embedding.
\end{claim}

\noindent Henceforth, we identify $\mathbb R(\mathcal Y, \delta)$ with a subspace of $\mathbb R(\mathcal T)$ using the canonical embedding~$f^*$.

\begin{proof}[Sketch of proof]
We first prove the injection~$f^*$ is continuous.
Let $\Lambda \subset \mathbb R(\mathcal T)$ be a closed subset and $\Lambda_f \defeq \Lambda \cap \operatorname{im}(f^*)$.
Suppose~$[\gamma]$ is in the closure of~$f_*(\Lambda_f)$ in $\mathbb R(\mathcal Y,\delta)$.
For continuity, it is enough to show $f^*[\gamma] \in \Lambda$.
Fix a long interval $I_\gamma \subset \gamma$;
then $I_\gamma \subset [f(p), f(q)]_Y$ for some $p,q \in f^*(\gamma)$.
As~$[\gamma]$ is in the closure of~$f_*(\Lambda_f)$, the interval $I_\gamma \subset \gamma$ is in the line~$f_*(l)$ for some $[l] \in \Lambda_f$. 
By bounded cancellation, the $f$-image of the intersection $I_l \defeq f^*(\gamma) \cap l$ contains a long interval in~$I_\gamma$. 
As the interval~$I_\gamma$ will exhaust~$\gamma$, the interval~$I_l$ exhausts~$f^*(\gamma)$;
in particular, any interval in~$f^*(\gamma)$ is contained in~$l$ for some $[l] \in \Lambda$.
So $f^*[\gamma]$ is in the closed subset~$\Lambda$.

We finally prove $f^*\colon \mathbb R(\mathcal Y, \delta) \to \operatorname{im}(f^*)$ is an open map, where the image $\operatorname{im}(f^*) \subset \mathbb R(\mathcal T)$ has the subspace topology.
Suppose $p,q\in \mathcal Y$  and $[\gamma] \in U[p,q]$, i.e.~a line~$\gamma$ in $(\mathcal Y, \delta)$ contains an interval~$[p,q]_Y$.
There is an interval~$[u,v]_T \subset f^*(\gamma)$ whose $f$-image covers a long neighbourhood of~$[p,q]_Y$.
By bounded cancellation, any line~$f^*(\gamma')$ containing~$[u,v]_T$ will have an $f_*$-image~$\gamma'$ containing~$[p,q]_Y$.
So $f^*[\gamma] \in U[u,v] \cap \operatorname{im}(f^*) \subset f^*(U[p,q])$.
As $[\gamma] \in U[p,q]$ was arbitrary, the image $f^*(U[p,q])$ is open in~$\operatorname{im}(f^*)$.
\end{proof}

Now assume~$\mathcal T'$ is a free splitting of~$\mathcal F$ with $\mathcal F[\mathcal T] = \mathcal F[\mathcal T']$ and let $f\colon \mathcal T \to \mathcal T'$ be an equivariant PL-map. 
The folds in the factorization of~$f$ never identify points in the same $\mathcal F$-orbit.
For~$[l] \in \mathbb R(\mathcal T)$, each end of~$l$ has unbounded $f$-image, i.e.~$\operatorname{dom}(f_*) = \mathbb R(\mathcal T)$;
so $f_* \colon \mathbb R(\mathcal T) \to \mathbb R(\mathcal T')$ is a canonical homeomorphism (with inverse~$f^*$).
Similarly, if $g \colon \mathcal T \to \mathcal T$ is a $\psi$-equivariant PL-map for some automorphism $\psi \colon \mathcal F \to \mathcal F$, then $g_*, g^* \colon \mathbb R(\mathcal T)  \to \mathbb R(\mathcal T)$ are canonical homeomorphisms for~$[\psi]$.

\begin{rmk}
We use ambiguous terminology and say ``line'' when we mean a line or an $\mathcal F$-orbit of a line;
our notation remains distinct: ``$l$'' is always a line, while ``$[l]$'' is its $\mathcal F$-orbit.
\end{rmk}

A \underline{lamination} in $(\mathcal Y, \delta)$ (resp. $\mathcal T$) is a nonempty closed subset of $\mathbb R(\mathcal Y, \delta)$ (resp. $\mathbb R(\mathcal T)$); 
when the $\mathcal F$-forest in question is clear, we say \emph{lamination} with no qualifier.
An element of a lamination is called a \underline{leaf};
a \underline{leaf segment} of a lamination~$\Lambda$ is a nondegenerate closed interval in a line representing a leaf of~$\Lambda$.
A lamination is \underline{minimal} if each leaf is dense in the lamination;
a lamination is \underline{perfect} if it has no isolated leaves.

Let $[l]$ be a line and $\Lambda$ a lamination in $\mathbb R(\mathcal Y, \delta)$ (or~$\mathbb R(\mathcal T)$).
A sequence~$[l_m]_{m \ge 0}$ in the space of lines \underline{weakly limits} to~$[l]$ if some subsequence converges to~$[l]$; we say~$[l]$ is a \emph{weak limit} of the sequence.
The sequence~$[l_m]_{m \ge 0}$ \underline{weakly limits} to~$\Lambda$ if it weakly limits to every leaf of~$\Lambda$.
The ``weak'' terminology is used to highlight that the space of lines is not Hausdorff --- a convergent sequence may have multiple limits!

More generally, a sequence $[p_m, q_m]_{m \ge 0}$ of intervals \emph{converges} to~$[l]$ if, for any closed interval $[a,b] \subset l$, $[p_m, q_m]$ contains a translate of~$[a,b]$ for $m \gg 1$ (i.e.~for large enough~$m$) --- precisely, there is an $M[a,b] \ge 1$ such that $U[a, b]$ contains $U[p_m,q_m]$ for $m \ge M[a,b]$.
Again, a sequence of intervals \underline{weakly limits} to~$[l]$ if some subsequence converges to~$[l]$ and it \underline{weakly limits} to~$\Lambda$ if it weakly limits to every leaf of~$\Lambda$.

\newpage
\section{Limit forests}\label{SecLimits}

In this chapter, we sketch the proof of Proposition~\ref{prop-limittree} (existence of limit forests) and, in the process, introduce stable laminations.
The first half deals with limit forests for expanding irreducible train tracks; then, in the second half, we extend the results to all limit forests.

\subsection{Constructing limit forests (1)}\label{SubsecLimittrees}

This is a summary of~\cite[Appendix]{GLL98};
the reader is invited to read that appendix.

Fix an automorphism $\psi \colon \mathcal F \to \mathcal F$ with an expanding irreducible train track $\tau \colon \mathcal T \to \mathcal T$.
Set~$\lambda \defeq \lambda[\tau] > 1$ and let~$d_\tau$ be the eigenmetric on~$\mathcal T$ for~$[\tau]$.
For $m \ge 0$, let~$d_m$ be the pullback of~$\lambda^{-m}d_\tau$ via~$\tau^m$: 
\[ d_m(p,q) \defeq \lambda^{-m} d_\tau(\tau^m(p), \tau^m(q)) \le d_\tau(p,q) \quad \text{for } p,q \text{ in a component of } T. \]
By definition, the pullback~$d_m$ is an $\mathcal F$-invariant (not necessarily convex) pseudometric on~$\mathcal T$ whose quotient metric space is equivariantly isometric to $(\mathcal T\psi^m, \lambda^{-m}d_\tau)$.
The $\lambda$-Lipschitz property for~$\tau$ with respect to~$d_\tau$ implies the sequence of pseudometrics~$d_m$ is (pointwise) monotone decreasing.
The limit~$d_\infty$ is an $\mathcal F$-invariant pseudometric on~$\mathcal T$, the quotient metric space $(\mathcal T_\infty, d_\infty)$ is an $\mathcal F$-forest, and the $\psi$-equivariant $\lambda$-Lipschitz train track~$\tau$ induces a $\psi$-equivariant $\lambda$-homothety~$h \colon (\mathcal T_\infty, d_\infty) \to (\mathcal T_\infty, d_\infty)$;
in particular, the equivariant metric surjection $\pi \colon (\mathcal T, d_\tau) \to (\mathcal T_\infty, d_\infty)$  {semiconjugates}~$\tau$ to~$h$:~$\pi \circ \tau = h \circ \pi$.

As~$\tau$ is a train track, the restriction of~$\pi$ to any edge of~$\mathcal T$ is an isometric embedding.
So~$\mathcal T_\infty$ is not degenerate.
In fact, the $\pi$-image of any edge of~$\mathcal T$ can be extended to an axis for a $\mathcal T_\infty$-loxodromic element in~$\mathcal F$. 
Thus the $\mathcal F$-forest $(\mathcal T_\infty, d_\infty)$ is minimal, and the uniqueness of~$h$ follows from~\cite[Theorem~3.7]{CM87}.
Finally, the minimal $\mathcal F$-forest $(\mathcal T_\infty, d_\infty)$ has trivial arc stabilizers.
This sketches the first case of Proposition~\ref{prop-limittree}.
The $\mathcal F$-forest $(\mathcal Y_\tau, d_\infty) \defeq (\mathcal T_\infty, d_\infty)$ is the \underline{(forward) limit forest} for~$[\tau]$.

\subsection{Stable laminations (1)}\label{SubsecLaminations}
The first part of this section is mostly adapted from Section~1 of~\cite{BFH97}.
The following definition of \emph{stable} laminations is from~\cite[Definition~1.3]{BFH97}.

Fix an automorphism $\psi \colon \mathcal F \to \mathcal F$ with an expanding irreducible train track $\tau \colon \mathcal T \to \mathcal T$.
Set $\lambda \defeq \lambda[\tau] > 1$, let~$d_\tau$ be the eigenmetric on~$\mathcal T$ for~$[\tau]$, and pick an edge $e \subset \mathcal T$.
Expanding irreducibility implies the interval $\tau^k(e)$ contains at least three translates of~$e$ for some $k \ge 1$.
By the intermediate value theorem,~$\tau^k(p) = x \cdot p$ for some $x \in \mathcal F$ and $p \in e$.
Recall that edges are open intervals;
since the restriction of $x^{-1} \cdot \tau^k$ to the edge~$e$ is an expanding $\lambda^k$-homothetic embedding $e \to \mathcal T$ (with respect to~$d_\tau$) that fixes~$p$ and has~$e$ in its image, we can extend~$e$ to a line~$l_p \subset \mathcal T$ by iterating $x^{-1} \cdot \tau^k$.
By construction, the restriction of $x^{-1} \cdot \tau^k$ to~$l_p$ is a $\lambda^k$-homothety $l_p \to l_p$ with respect to the eigenmetric~$d_\tau$ for~$[\tau]$;
the $\mathcal F$-orbit~$[l_p]$ is an \underline{eigenline} of~$[\tau^k]$ based at~$[p]$ (in $\mathcal F \backslash \mathcal T$).
A \underline{stable $\mathcal T$-lamination}~$\Lambda^+$ for~$[\tau]$ is the closure of an eigenline of~$[\tau^k]$ for some $k \ge 1$. 
By $\phi$-equivariance of~$\tau$, the restriction of~$\tau$ to~$l$ representing a leaf of a stable lamination~$\Lambda^+$ is a $\lambda$-homothetic embedding.
In fact,~$[\tau]$ maps eigenlines to eigenlines, and the image $\tau_*(\Lambda^+) \defeq \{ [\tau(l)] : [l] \in \Lambda^+\}$ is also a stable lamination for~$[\tau]$.

As the {transition matrix}~$A[\tau]$ is irreducible, it is a block transitive permutation matrix, and the ``first return'' matrix for each block is \emph{primitive}, i.e.~has a positive power.
There is a bijective correspondence between the stable laminations for~$[\tau]$ and the blocks of~$A[\tau]$.
In particular, there are finitely many stable laminations for~$[\tau]$, these laminations are pairwise disjoint, and~$\tau_*$ transitively permutes them~\cite[Lemma~1.2]{BFH97}.
By finiteness, their union~$\mathcal L^+[\tau]$ is a lamination and is called the \emph{system of stable laminations} for~$[\tau]$.

\subsubsection {Quasiperiodic lines}\label{Subsubsec-quasiproj}

A line~$[l]$ in an $\mathcal F$-forest is \emph{periodic} if it is the axis for the conjugacy class of some loxodromic element of~$\mathcal F$.
A line~$[l]$ is \underline{quasiperiodic} in an $\mathcal F$-forest if any closed interval~$I$ in~$l$ has an assigned number $L(I) \ge 0$ such that any interval in~$l$ of length~$L(I)$ contains a translate of~$I$;
periodic lines are quasiperiodic.
If~$[l]$ is a quasiperiodic line, then any leaf of its closure~$\Lambda$ is quasiperiodic and hence dense in~$\Lambda$ (exercise), i.e.~$\Lambda$ is minimal.
If~$[l]$ is quasiperiodic but not periodic, then no leaf of its closure~$\Lambda$ is isolated (exercise), i.e.~$\Lambda$ is also perfect.

\begin{rmk}
When the $F$-action on a free splitting~$T$ is free, then our definition of quasiperiodicity is equivalent to~\cite[Definition~1.7]{BFH97}; however, our definition is weaker when the action is not free.
\end{rmk}

\begin{lem}[cf.~{\cite[Proposition~1.8]{BFH97}}]\label{lem-stableminperf} 
The eigenlines of~$[\tau^k]$ are quasiperiodic but not periodic for $k \ge 1$.
Thus the stable laminations for~$[\tau]$ are minimal and perfect.
\end{lem}

\begin{proof}
There is a length~$L_0$ such that any interval in~$\mathcal T$ of length~$L_0$ contains an edge.
Fix an $\mathcal F$-orbit~$[I]$ of intervals in an eigenline~$[l]$ of~$[\tau^k]$.
By construction,~$I$ is contained in~$\tau^{km}(e)$ for some edge~$e$ in~$\mathcal T$ and integer $m \ge 0$.
As the blocks in~$A[\tau^k]$ are primitive, there is an integer $m' \ge 1$ such that $\tau^{km'}(e')$ contains a translate of~$e$ for any edge~$e'$ in~$l$.
Altogether, any interval in~$l$ of length $\lambda[\tau]^{k(m+m')} L_0$ contains a translate of~$I$.
This proves quasiperiodicity.

Now assume, for a contradiction, that the eigenline~$[l]$ were periodic, i.e.~$l$ is an axis for a $\mathcal T$-loxodromic element~$x \in \mathcal F$.
By construction, the $\mathcal F$-orbit~$[l]$ is $\tau^k$-invariant and hence the cyclic subgroup~$\langle x \rangle$ is $[\psi^k]$-invariant.
So~$x$ is $[\psi]$-periodic as~$\psi$ is an automorphism;
yet~$x$ must $[\psi]$-grow exponentially since its axis is an eigenline and~$\tau$ is expanding.
\end{proof}

Fix an equivariant PL-map $f \colon (\mathcal T, d) \to (\mathcal Y, \delta)$ and canonically embed $\mathbb R(\mathcal Y, \delta)$ into $\mathbb R(\mathcal T)$ via~$f^*$ (Claim~\ref{claim-liftembed}).
If a quasiperiodic line $[l] \in \mathbb R(\mathcal T)$ is in the subspace $\mathbb R(\mathcal Y, \delta) = \operatorname{im}(f^*)$, then so its closure~$\Lambda$ in $\mathbb R(\mathcal T)$ (exercise).
Returning to limit forests, the equivariant metric PL-map $\pi \colon (\mathcal T, d_\tau) \to (\mathcal Y_\tau, d_\infty)$ restricts to an isometric embedding on the leaves of~$\mathcal L^+[\tau]$;
therefore, the stable lamination~$\mathcal L^+[\tau]$ is in $\mathbb R(\mathcal Y_\tau, d_\infty) \subset \mathbb R(\mathcal T)$.

\subsubsection {Characterizing loxodromics}

Let $(\mathcal Y_\tau, d_\infty)$ be the limit forest for~$[\tau]$, $h: (\mathcal Y_\tau, d_\infty) \to (\mathcal Y_\tau, d_\infty)$ the unique $\psi$-equivariant $\lambda$-homothety, and $\pi\colon (\mathcal T, d_\tau) \to (\mathcal Y_\tau, d_\infty)$ the constructed equivariant metric PL-map.
By Lemma~\ref{lem-bcl}, the map $\tau \colon (\mathcal T, d_\tau) \to (\mathcal T, d_\tau)$ has a cancellation constant $C \defeq C[\tau]$.
Set $C' \defeq \frac{C}{\lambda-1}$ and denote the interval functions for~$\mathcal T$  by $[\cdot, \cdot]$.
The sequence of equivariant metric maps $\tau^m\colon (\mathcal T, d_\tau) \to (\mathcal T\psi^m, \lambda^{-m}d_\tau)$ have cancellation constants $\sum_{i=1}^m \lambda^{-i}C \le C'$;
so their {limit}~$\pi$ has a cancellation constant $C[\pi] \defeq C'$.

Let~$P \subset \mathcal Y_\tau$ be $\mathcal F$-orbit representatives of points with nontrivial stabilizers.
Define the subgroup system $\mathcal G[\mathcal Y_\tau] \defeq \bigsqcup_{p \in P} G_p$, where~$G_p \defeq \operatorname{Stab}_{\mathcal F}(p)$ is the stabilizer in~$\mathcal F$ of~$p \in P$.
As the action on~$\mathcal Y_\tau$ has trivial arc stabilizers, the system~$\mathcal G[\mathcal Y_\tau]$ is \underline{malnormal}: each component is malnormal (as a subgroup of the appropriate component of~$\mathcal F$) and conjugates of distinct components (in the same component of~$\mathcal F$) have trivial intersections.
The $\psi$-equivariance of homothety~$h$ implies~$\mathcal G[\mathcal Y_\tau]$ is $[\psi]$-invariant.
By Gaboriau--Levitt's index theory, the complexity of~$\mathcal G[\mathcal Y_\tau]$ is strictly less than that of~$\mathcal F$~\cite[Theorem~III.2]{GL95}.
In particular,~$\mathcal G[\mathcal Y_\tau]$ has finite type:~$P$ is finite, and each component~$G_p$ is finitely generated.
The restriction of~$\psi$ to~$\mathcal G[\mathcal Y_\tau]$ determines a unique outer automorphism of the system.

We now characterize the elliptic/loxodromic elements in~$\mathcal F$ in the limit forest~$(\mathcal Y_\tau, d_\infty)$:

\begin{prop}[cf.~{\cite[Proposition~1.6]{BFH97}}]\label{prop-limitloxodromics}
Let~$\psi\colon \mathcal F \to \mathcal F$ be an automorphism, $\tau \colon \mathcal T \to \mathcal T$ an expanding irreducible train track for~$\psi$, and $(\mathcal Y_\tau, d_\infty)$ the limit forest for~$[\tau]$.

If~$x \in \mathcal F$ is a $\mathcal T$-loxodromic element, then the following statements are equivalent:
\begin{enumerate}
\item the element~$x$ is $\mathcal Y_\tau$-loxodromic;
\item the element~$x$ $[\psi]$-grows exponentially rel.~$\mathcal T$: $\underset{m \to \infty} \lim \, \frac{1}{m} \log \| \psi^m(x)\|_{\mathcal T} = \log \lambda[\tau]$; and
\item the $\mathcal T$-axis for~$\psi^m(x)$ has an arbitrarily long leaf segment of~$\mathcal L^+[\tau]$ for $m \gg 1$.
\end{enumerate}
The restriction of~$\psi$ to the $[\psi]$-invariant subgroup system~$\mathcal G[\mathcal Y_\tau]$ of $\mathcal Y_\tau$-point stabilizers has constant growth rel.~$\mathcal T$:
$\{ \| \psi^m(x)\|_{\mathcal T} : m \ge 0 \}$ is bounded for all $x \in \mathcal G[\mathcal Y_\tau]$.
\end{prop}

\begin{proof}
Let $\lambda \defeq \lambda[\tau] > 1$, $C \defeq C[\tau]$ a cancellation constant for $\tau\colon(\mathcal T,d_\tau) \to (\mathcal T,d_\tau)$,
and $C' \defeq \frac{C}{\lambda - 1}$ a cancellation constant for $\pi\colon (\mathcal T, d_\tau) \to (\mathcal Y_\tau, d_\infty)$.
Fix a line~$l$ in~$\mathcal T$, and let $\pi\colon (\mathcal T, d_\tau) \to (\mathcal Y_\tau, d_\infty)$ be the constructed equivariant metric PL-map.

\smallskip
\noindent \underline{Case 1:} $d_\infty(\pi(p), \pi(q)) > 2C' + L$ for some $k \ge 0$, $p,q \in \tau_*^k(l)$, and $L > 0$.
By definition of~$d_\infty$ (construction of~$\pi$), $d_\tau(\tau^m(p),\tau^m(q)) > \lambda^m(2C' + L)$ for $m \gg 1$.
Pick $m \gg 1$ and $r_m, s_m \in [\tau^m(p), \tau^m(q)]$ so that $d_\tau(\tau^m(p), r_m), d_\tau(s_m, \tau^m(q)) > \lambda^m C'$ and $d_\tau(r_m, s_m) > \lambda^m L$.
By bounded cancellation (for~$\tau^m$), the interval $I_m \defeq [r_m, s_m]$ is disjoint from $\tau^m(\tau_*^k(l) \setminus[p,q])$ in $(T, d_\tau)$.
So~$I_m$ is an interval in~$\tau_*^{m+k}(l)$.

Let $N \defeq N(p,q)$ be the number of vertices in the interval~$(p,q)$. 
Then~$I_m$ is covered by~$N+1$ leaf segments (of~$\mathcal L^+[\tau]$) as~$\tau$ is a train track.
By the pigeonhole principle,~$I_m$ (and hence~$\tau_*^{m+k}(l)$) contains a leaf segment with $d_\tau$-length $> \frac{\lambda^m L}{N+1}$;
therefore, the line~$\tau_*^n(l)$ in~$\mathcal T$ contains arbitrarily long leaf segments for $m \gg 1$.

\smallskip
\noindent \underline{Case 2:} $\pi(\tau_*^m(l))$ has diameter $\le 2C'$ for all $m \ge 0$.
We claim that any leaf segment in the line~$\tau_*^m(l)~(m\ge 0)$ has $d_\tau$-length $ \le 2C'$.
For the contrapositive, suppose some~$\tau_*^m(l)$ has a leaf segment with $d_\tau$-length $L > 2C'$.
By the train track property and bounded cancellation,~$\tau_*^{m+1}(l)$ has a leaf segment with $d_\tau$-length $\ge \lambda L - 2C > L$.
By induction, $\pi(\tau_*^{m+m'}(l))$ has diameter $\ge \lambda^{m'} (L - 2C')$ for $m' \ge 0$ and $\lambda^{m'} (L - 2C') > 2C'$ for $m' \gg 1$.

\smallskip
We finally return to the proof of the proposition.
Fix a $\mathcal T$-loxodromic element $x \in \mathcal F$
and let~$l \subset \mathcal T$ be its axis;
in particular,~$\pi(l)$ is $x$-invariant by equivariance of~$\pi$.
As~$\tau$ is $\lambda$-Lipschitz with respect to~$d_\tau$, $\underset{m \to \infty} \limsup \, \frac{1}{m} \log \| \psi^m(x)\|_{d_\tau} \le \log \lambda$.

\smallskip
\noindent \emph{Case--i:}
$d_\infty(\pi(p), \pi(q)) > 2C'$ for some $k \ge 0$ and $p,q \in \tau_*^k(l)$.
The line~$\tau_*^m(l)$, the axis for~$\phi^m(x)$ in~$\mathcal T$, contains an arbitrarily long leaf segment for $m \gg 1$ by the Case~1 analysis.
By bounded cancellation (for~$\pi$), some nondegenerate interval~$I$ in $[\pi(p), \pi(q)]_{\infty}$ is disjoint from $\pi(\tau_*^k(l) \setminus [p,q])$.
Since~$\tau_*^k(l)$ is the axis for~$\psi^k(x)$, it contains disjoint translates $[p,q]$, $\psi^k(x^{-n}) \cdot [p,q]$, $\psi^k(x^{n}) \cdot [p,q]$ for some~$n \gg 1$.
Then~$\psi^k(x^{-n}) \cdot I $ and $\psi^k(x^{n}) \cdot I$ are in distinct components of $\mathcal Y_\tau \setminus I$ and~$\psi^k(x)$ is $\mathcal Y_\tau$-loxodromic.
Since $\|\cdot\|_{d_\infty} \le \| \cdot \|_{d_\tau}$ and $\|\psi(\cdot)\|_{d_\infty} = \lambda \|\cdot\|_{d_\infty}$, we get $ \log \lambda \le \underset{m \to \infty} \liminf \, \frac{1}{m} \log \| \psi^m(x)\|_{d_\tau}$ and~$x$ is $\mathcal Y_\tau$-loxodromic.
Finally, $\log \lambda = \underset{m \to \infty} \lim \, \frac{1}{m} \log \| \psi^m(x)\|_{\mathcal T}$ since~$d_\tau$ and the combinatorial metric on~$\mathcal T$ are bilipschitz.

\smallskip
\noindent \emph{Case--ii:}
$\pi(\tau_*^m(l))$ has diameter $\le 2C'$ for all $m \ge 0$.
Any leaf segment in~$\tau_*^m(l)~(m \ge 0)$ have $d_\tau$-length $\le 2C'$ by Case~2 analysis.
Let~$N$ be the number of vertices in a fundamental domain of $x$ acting on~$l$.
By the train track property, the fundamental domain of~$\tau_*^m(l)$ is covered by~$N+1$ leaf segments and $\|\psi^m(x)\|_{\mathcal T} \le K \|\psi^m(x)\|_{d_\tau} \le 2C'K(N+1)$ for some $K \ge 1$ and all $m \ge 0$.
But~$x$ acts on~$\mathcal Y_\tau$ by an isometry, and~$\pi(l) \subset \mathcal Y_\tau$ is $x$-invariant;
so~$x$ must be $\mathcal Y_\tau$-elliptic.
\end{proof}

We now introduce the notion of \emph{factored} forests.
Suppose the stable laminations~$\mathcal L^+[\tau]$ have components $\Lambda_i^+~(1 \le i \le k)$.
The $\mathcal F$-orbits of edges in~$\mathcal T$ can be partitioned into \emph{blocks}~$\mathcal B_i$ indexed by the components $\Lambda_i^+ \subset \mathcal L^+[\tau]$.
For $p,q \in \mathcal T$, let~$d_\tau^{(i)}$ be the $d_\tau$-length of the intersection of the interval $[p,q]$ and the subforest spanned by~$\mathcal B_i$;
this defines an $\mathcal F$-invariant convex pseudometric~$d_\tau^{(i)}$ on~$\mathcal T$.
The metric~$d_\tau$ is a sum of the pseudometrics~$d_\tau^{(i)}$, denoted $\Sigma_{i=1}^k d_\tau^{(i)}$;
we call $\Sigma_{i=1}^k d_\tau^{(i)}$ a \underline{factored} $\mathcal F$-invariant convex metric and $(\mathcal T, \Sigma_{i=1}^k d_\tau^{(i)})$ a \underline{factored} $\mathcal F$-forest.
This factored metric is special: the \emph{factors} $d_\tau^{(i)}~(1 \le i \le k)$ are pairwise \underline{mutually singular}: for $i \neq j$, there are intervals (e.g.~the leaf segments) with positive $d_\tau^{(i)}$-length and 0 $d_\tau^{(j)}$-length --- such factored pseudometrics will be denoted by $\oplus_{i=1}^k d_\tau^{(i)}$ to invoke the idea of independence in direct sums.
The limit pseudometrics~$d_\infty^{(i)}$ are pairwise mutually singular since~$\pi$ is surjective and isometric on leaf segments;
thus $d_\infty = \oplus_{i=1}^k d_\infty^{(i)}$.
The next lemma is the cornerstone of our universality result:

\begin{restate}{Lemma}{lem-convergence}[{cf.~\cite[Lemma~3.4]{BFH97}}] 
Let~$\psi\colon \mathcal F \to \mathcal F$ be an automorphism,~$\tau \colon \mathcal T \to \mathcal T$ an expanding irreducible train track for~$\psi$ with eigenmetric~$d_\tau$, $(\mathcal Y_\tau, d_\infty)$ the limit forest for~$[\tau]$, and $\lambda \defeq \lambda[\tau]$.

If $(\mathcal T, d_\tau) \to (\mathcal Y, \delta)$ is an equivariant PL-map and the $k$-component lamination~$\mathcal L^+[\tau]$ is in $\mathbb R(\mathcal Y, \delta) \subset \mathbb R(\mathcal T)$, then the sequence $(\mathcal Y \psi^{mk}, \lambda^{-mk} \delta)_{m \ge 0}$ converges to $(\mathcal Y_\tau, \oplus_{i=1}^k c_i \, d_\infty^{(i)})$, where $d_\infty = \oplus_{i=1}^k \, d_\infty^{(i)}$ and $c_i > 0$.
\end{restate}

\begin{rmk} Factored $\mathcal F$-forests are needed for this lemma when $k \ge 2$; 
moreover, the sequence $(\mathcal Y \psi^m, \lambda^{-m}\delta)_{m \ge 0}$ will not converge in general (but is asymptotically periodic) when $k \ge 2$.
Convergence is in the subspace of \emph{translation distance functions} in $\mathbb R_{\ge 0}^{\mathcal F}$ with the product topology.
\end{rmk}

We give the proof in Section~\ref{SubsecConverge}.
In particular, if $\tau'\colon \mathcal T' \to \mathcal T'$ is another expanding irreducible  train track for~$\psi$ and $\mathcal F[\mathcal T'] = \mathcal F[\mathcal T]$, then the limit forest for~$[\tau']$ is equivariantly homothetic to~$(\mathcal Y_\tau, d_\infty)$ --- set $(\mathcal Y, \delta) \defeq (\mathcal T', d_{\tau'})$, apply the lemma, then observe that the sequence~$(c_i)_{i=1}^k$ must be constant in this case.
A minimal very small $\mathcal F$-forest $(\mathcal Y, \delta)$ is an \underline{expanding forest} for~$[\psi]$ like~$\mathcal Y_\tau$ if it is nondegenerate and there is:

\begin{enumerate}
\item a $\psi$-equivariant expanding homothety $(\mathcal Y, \delta) \to (\mathcal Y, \delta)$; and
\item an equivariant PL-map $(\mathcal T, d_\tau) \to (\mathcal Y, \delta)$.
\end{enumerate}

\begin{cor}\label{cor-unique}
Let~$\psi\colon \mathcal F \to \mathcal F$ be an automorphism,~$\tau \colon \mathcal T \to \mathcal T$ an expanding irreducible train track for~$\psi$, and $(\mathcal Y_\tau, d_\infty)$ the limit forest for~$[\tau]$.
Any expanding forests for~$[\psi]$ like~$\mathcal Y_\tau$ is uniquely equivariantly homothetic to $(\mathcal Y_\tau, d_\infty)$.
\end{cor}

\begin{proof}
Let $(\mathcal Y, \delta)$ be an expanding forest for~$[\psi]$ like~$\mathcal Y_\tau$, $f \colon (\mathcal T, d_\tau) \to (\mathcal Y, \delta)$ an equivariant PL-map with cancellation constant $C \defeq C[f]$, $g \colon (\mathcal Y, \delta) \to (\mathcal Y, \delta)$ the $\psi$-equivariant expanding $s$-homothety, $x \in \mathcal F$ a $\mathcal Y$-loxodromic element.
By equivariance of~$f$, the element~$x$ is $\mathcal T$-loxodromic with axis $l_x \subset \mathcal T$.
Let $[p_0, x \cdot p_0] \subset l_x$ be (the closure of) a fundamental domain of~$x$ acting on~$l_x$.
The interval $[p_0,x \cdot p_0]$ is a concatenation of $N \ge 1$ leaf segments (of $\mathcal L^+[\tau]$).
Choose $m \gg 1$ so that $\| \psi^m(x) \|_{\delta} = s^m \| x \|_\delta > 2C N$.
Note that the action of $\psi^m(x)$ on its axis has a fundamental domain $[p_m, \psi^m(x) \cdot p_m]$ covered by~$N$ leaf segments as~$\tau$ is a train track.
So $\delta(f(p_m), f(\psi^m(x) \cdot p_m)) > 2C N$ and, by the pigeonhole principle, $[p_m, \psi^m(x) \cdot p_m]$ contains a leaf segment $[q, r]$ with $\delta(f(q), f(r)) > 2C$.

Let $l \supset [q,r]$ represent some leaf $[l] \in \mathcal L^+[\tau]$.
Bounded cancellation implies the components of $l \setminus [q,r]$ have $f$-images with disjoint closures.
By quasiperiodicity of~$[l]$ and equivariance of~$f$, both ends of~$l$ have unbounded $f$-images, i.e.~$[l] \in \operatorname{dom}(f_*) = \mathbb R(\mathcal Y, \delta)$ (Corollary~\ref{cor-tightenings}, Claim~\ref{claim-liftembed}).
Finally, the closure of~$[l]$ in~$\mathbb R(\mathcal T)$, a component $\Lambda_i^+ \subset \mathcal L^+[\tau]$, is a subset of $\mathbb R(\mathcal Y, \delta)$ by quasiperiodicity of~$[l]$.
Note that the $\psi$-equivariant homothety~$g$ induces a homeomorphism $g_* \colon \mathbb R(\mathcal Y, \delta) \to \mathbb R(\mathcal Y, \delta)$ that is the restriction of the homeomorphism $\tau_* \colon \mathbb R(\mathcal T) \to \mathbb R(\mathcal T)$.
So $\mathcal L^+[\tau] \subset \mathbb R(\mathcal Y, \delta)$ since~$\tau_*$ acts transitively on the~$k$ components of~$\mathcal L^+[\tau]$.
Set $\lambda \defeq \lambda[\tau]$;
by Lemma~\ref{lem-convergence}, the sequence $(\mathcal Y\psi^{mk}, \lambda^{-mk}\delta)_{m \ge 0}$ converges to the factored $\mathcal F$-forest $(\mathcal Y_\tau, \oplus_{i=1}^k c_i \, d_\infty^{(i)})$ for some $c_i > 0$.
Yet $(\mathcal Y, \delta)$ is equivariantly isometric to $(\mathcal Y\psi, s^{-1} \delta)$;
thus $s = \lambda$, $c_i = c_{i+1}~(i<k)$, $(\mathcal Y, \delta)$ is equivariantly isometric to $(\mathcal Y_\tau, c_1 \, d_\infty)$, and the equivariant isometry is unique~\cite[Theorem~3.7]{CM87}.
\end{proof}

\subsubsection {Iterated turns}\label{Subsubsec-iteratedturns}

We have already shown how iterating an edge in~$\mathcal T$ by the train track~$\tau$ produces the system of stable laminations~$\mathcal L^+[\tau]$.
Later, we will consider how~$\mathcal L^+[\tau]$ determines laminations in (a free splitting of) the subgroup system~$\mathcal G[\mathcal Y_\tau]$.

Let~$\mathcal T'$ be a free splitting of~$\mathcal F$ whose free factor system~$\mathcal F[\mathcal T']$ is trivial.
Then there is an equivariant PL-map $f \colon (\mathcal T', d') \to (\mathcal T, d_\tau)$.
Let~$\gamma$ be a line in $(\mathcal Y_\tau, d_\infty)$, $\pi^*(\gamma)$ its lift to $(\mathcal T, d_\tau)$, and $f^*(\pi^*(\gamma))$ its lift to $(\mathcal T', d')$.
Denote the ends of~$\gamma$ by $\varepsilon_i~(i = 1,2)$.
Let $T \subset \mathcal T$ be the component containing~$\pi^*(\gamma)$, and $T' \subset \mathcal T'$, $Y_\tau \subset \mathcal Y_\tau$, and $F \subset \mathcal F$ be the matching components.
Denote the first return maps of~$\tau$, $h$, and~$\psi$ on~$T$, $Y_\tau$, and~$F$ by~$\tilde \tau$, $\tilde h$, and~$\varphi$ respectively.
For the rest of the section, redefine~$\lambda$ to be the stretch factor of the expanding homothety~$\tilde h$. 

Suppose~$\circ$ is a point on the line~$\gamma$ with a nontrivial stabilizers $G_\circ \defeq \operatorname{Stab}_F(\circ)$.
Let~$d_i~(i = 1,2)$ be the direction at~$\circ$ containing~$\varepsilon_i$.
By Gaboriau--Levitt index theory, $\tilde h^j(\circ) = y \cdot \circ$ and $\tilde h^j(d_i) = ys_i \cdot d_i$ for some  $y \in F$, $s_i \in G_\circ$, and minimal $j \ge 1$.
Since~$F$ acts on~$Y_\tau$ with trivial arc stabilizers, the elements $ys_1, ys_2$ are unique and $s_1^{-1}s_2 \in G_\circ$ is independent of the chosen $y \in F$.

Set $y_0 \defeq \epsilon$ to be the trivial element and $y_{m+1} \defeq \varphi^{mj}(ys_1)y_m$ for $m \ge 0$.
Let~$T'(G)$ be the characteristic convex subset of~$T'$ for a nontrivial subgroup $G \le F$.
Since~$T'$ is simplicial, the characteristic convex subset~$T'(G)$ is {closed}, and we have the {closest point retraction} $T' \to T'(G)$;
it extends uniquely to the {ends-completions}.
Let~$q_{i,m}'$ be the closest point projection of~$f^*(\pi^*(\tilde h_*^{mj}(\varepsilon_i)))$ to~$T'(\varphi^{mj}(G_\circ))$.
Set $\tau_\circ \defeq (ys_1)^{-1} \cdot \tilde \tau^{j}$ and $h_\circ \defeq (ys_1)^{-1} \cdot \tilde h^j$ to be $\psi_\circ$-equivariant maps for an automorphism $\psi_\circ \colon F \to F$ in the outer class~$[\varphi^j]$.
As~$h_\circ$ fixes~$\circ$, 
we get $\psi_\circ(G_\circ) = G_\circ$ and $y_{m}^{-1} \cdot T'(\varphi^{mj}(G_\circ))$ is the characteristic convex subset for $\psi_\circ^{m}(G_\circ) = G_\circ$.
Thus $q_{i,m} \defeq y_{m}^{-1} \cdot q_{i,m}'$ is in~$T'(G_\circ)$ for $i = 1,2$ and $m \ge 0$. 
The interval $[q_{1,m}, q_{2,m}]$  in~$T'(G_\circ)$, i.e.~the closest point projection of~$f^*(\pi^*(h_\circ^m(\gamma)))$, is the \underline{turn} in~$f^*(\pi^*(h_\circ^m(\gamma)))$ about~$T'(G_\circ)$.

Since $h_\circ(d_1) = d_1$, the ends $h_{\circ *}^m(\varepsilon_1)~(m \ge 0)$ are in fact ends of~$d_1$.
If $h_{\circ *}(\varepsilon_1) = \varepsilon_1$, then the sequence $(q_{1,m})_{m \ge 0}$ is constant.
Otherwise, the ends $h_{\circ *}^m(\varepsilon_1)$ are distinct for $m \ge 0$.
Let~$\gamma_{1,m}$ be the line in~$d_1$ determined by~$h_{\circ *}^{m+1}(\varepsilon_1)$ and~$h_{\circ *}^m(\varepsilon_1)$.
As~$h_\circ$ is an expanding homothety, the distance $d_\infty(\circ, \gamma_{1,m}) > 0$ from~$\circ$ to~$\gamma_{1,m}$ grows exponentially in~$m$.
So $d_\infty(\circ, \gamma_{1,M_1}) > 2C[\pi \circ f]$ for some minimal $M_1 \ge 0$, and the line $f^*(\pi^*(\gamma_{1,m}))$ is disjoint from~$T'(G_\circ)$ for $m \ge M_1$ by bounded cancellation (see Figure~\ref{fig - disjoint line}). 
In particular, the ends~$f^*(\pi^*(h_{\circ *}^{m+1}(\varepsilon_1)))$ and~$f^*(\pi^*(h_{\circ *}^m(\varepsilon_1)))$ have the same closest point projection to~$T'(G_\circ)$ and the sequence $(q_{1,m})_{m \ge M_1}$ is constant.

\begin{figure}[ht]
 \centering 
 \includegraphics{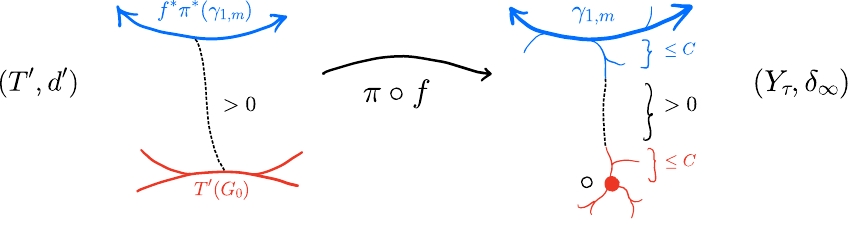}
 \caption{For $m \ge M_1$, the line $f^*(\pi^*(\gamma_{1,m}))$ cannot intersect~$T'(G_\circ)$.}
 \label{fig - disjoint line}
\end{figure}

Since $h_\circ(d_2) = s_1^{-1}s_2 \cdot d_2$, the ends $f^*(\pi^*(h_{\circ *}^{m+1}(\varepsilon_2)))$ and $\psi_\circ^{m}(s_1^{-1}s_2) \cdot f^*(\pi^*(h_{\circ *}^{m}(\varepsilon_2)))$ have the same closest point projection to~$T'(G_\circ)$ for $m \gg 1$ by similar bounded cancellation reasoning, i.e.~$q_{2,m+1} = \psi_\circ^m(s_1^{-1}s_2)\cdot q_{2,m}$ for some minimal $M_2 \ge 0$ and all $m \ge M_2$.

Set $M = \max(M_1, M_2)$.
The sequence $[q_{1,M+m}, q_{2,M+m}]_{m \ge 0}$ of intervals is well-defined for the line~$\gamma$ and point $\circ \in \gamma$ as~$M_1$ and~$M_2$ were chosen minimally.
An \emph{iterated turn over~$T'(G_\circ)$ rel.~$\left.\psi_\circ\right|_{G_\circ}$} is any such sequence of intervals. 
More generally, we define an \underline{iterated turn} over~$T$ rel.~$\varphi$:
pick arbitrary points $p_i \in T~(i=1,2)$ and elements $x_i \in F$;
set $p_{i,0} \defeq p_i$ and $p_{i,m+1} \defeq \varphi^m(x_i)\cdot p_{i,m}$ for $m \ge 0$;
the sequence $[p_{1,m}, p_{2,m}]_{m \ge 0}$ is the iterated turn denoted by $(p_1, p_2: x_1, x_2; \varphi)_{T}$. 
Any iterated turn $(p_1, p_2: x_1, x_2; \varphi)_{T}$ translates to a unique \emph{normal} form $(p_1, p_2: \epsilon, x_1^{-1}x_2; \tilde \varphi)_{T}$ with $\tilde \varphi \colon y \mapsto x_1^{-1}\varphi(y)x_1$.

\smallskip
We now characterize the growth of an iterated turn over~$T$ rel.~$\varphi$:

\begin{prop}\label{prop-limitturns}
Let~$\psi\colon \mathcal F \to \mathcal F$ be an automorphism, $\tau \colon \mathcal T \to \mathcal T$ an expanding irreducible train track for~$\psi$ with eigenmetric~$d_\tau$, and $(\mathcal Y_\tau, d_\infty)$ the limit forest for~$[\tau]$.
Choose a nondegenerate component~$T \subset \mathcal T$, corresponding components $F \subset \mathcal F$, $Y_\tau \subset \mathcal Y_\tau$, and a positive iterate~$\psi^k$ that preserves~$F$.
Let $\tilde h \colon (Y_\tau, d_\infty) \to (Y_\tau, d_\infty)$ be the $\varphi$-equivariant $\lambda$-homothety, where~$\varphi$ is in the outer automorphism~$[\left.\psi^k\right|_F]$ and $\lambda \defeq (\lambda[\tau])^k$.
Finally, for $i = 1,2$, pick $p_i \in T$ and $x_i \in F$.

\smallskip
The point~$p_{i,m} \defeq \varphi^{-1}(x_i)\cdots \varphi^{-m}(x_i) \cdot p_i$ in $(T\varphi^m, \lambda^{-m}d_\tau)$ converges to~$\star_i$ in $(\overline{Y}_\tau, d_\infty)$ as $m \to \infty$, where~$\star_i$ is the unique fixed point of $x_i^{-1} \cdot \tilde h$ in the metric completion $(\overline Y_\tau, d_\infty)$;
concretely: 
\[ \lim_{m \to \infty} \lambda^{-m} d_\tau(p_{1,m}, p_{2,m}) 
= d_\infty(\star_1, \star_2).\]

If $x_1^{-1}x_2$ fixes~$\star_1$, then $\star_1 = \star_2$ and the $m^{th}$~term $[p_{1,m}, p_{2,m}]$ of the iterated turn $(p_1, p_2: x_1, x_2; \varphi)_T$ has $d_\tau$-length $ \le (m+1)A$ for some constant $A \ge 1$.
Otherwise, $\star_1 \neq \star_2$ and the iterated turn has arbitrarily long leaf segments of~$\mathcal L^+[\tau]$.
\end{prop}

\noindent The limit $[\star_1, \star_2] \subset \overline{Y}_\tau$ of an iterated turn is independent of the points $p_1, p_2 \in T$.
Thus we introduce the notion of an \underline{\emph{algebraic} iterated turn} over~$F$ rel.~$\varphi$, denoted $(x_1, x_2; \varphi)_F$.

\begin{proof}
Let $p_1, p_2 \in T$, $x_1, x_2 \in F$, and $\pi\colon (\mathcal T, d_\tau) \to (\mathcal Y_\tau, d_\infty)$ be the constructed equivariant metric PL-map.
For $i = 1,2$, set $p_{i,0} \defeq p_i$ and $p_{i,m+1} \defeq \varphi^{m}(x_i) \cdot p_{i,m}$ for $m \ge 0$. 
Recall that $T$ and $T\varphi^m$ are the same pretrees but with different actions;
thus, in~$T\varphi^m$, we have $p_{i,m} = \varphi^{-1}(x_i)\cdots \varphi^{-m}(x_i) \cdot p_i$ for $m \ge 0$.
As $\pi \colon (T, d_\tau) \to (Y_\tau, d_\infty)$ is an equivariant metric PL-map, so is the composition
\[
\pi_m \colon (T\varphi^m, \lambda^{-m}d_\tau) \overset{\pi}\longrightarrow (Y_\tau\varphi^m, \lambda^{-m}d_\infty) \overset{\tilde h^{-m}}\longrightarrow (Y_\tau, d_\infty).
\]
The point~$p_i$ in $(T, d_\tau)$ projects (via~$\pi$) to~$\pi(p_i)$ in $(Y_\tau, d_\infty)$; the point~$p_{i,m}$ in $(T \varphi^m, \lambda^{-m}d_\tau)$ projects (via~$\pi_m$) to 
\[ \begin{aligned}
\pi_m(p_{i,m}) & \defeq \tilde h^{-m}(\pi(p_{i,m})) \\
&= \varphi^{-1}(x_i)\cdots \varphi^{-m}(x_i) \cdot \tilde h^{-m}(\pi(p_{i})) & \quad & (\, p_{i,m},p_i \in T\varphi^m\,)\\
&= (x_i^{-1} \cdot \tilde h)^{-m}(\pi(p_{i})) & & (\,p_i \in T\,)
\end{aligned}\]
in $(Y_\tau, d_\infty)$ for $m \ge 1$ --- in the last line, $x_i^{-1} \cdot \tilde h$ is a $\lambda$-homothety $(Y_\tau, d_\infty) \to (Y_\tau, d_\infty)$.
Since $(x_i^{-1} \cdot \tilde h)^{-1}$ is contracting, the point~$\pi_m(p_{i,m})$ converges (as $m \to \infty$) to the unique fixed point~$\star_i$ of~$(x_i^{-1}\cdot \tilde h)^{-1}$ (and~$x_i^{-1} \cdot \tilde h$) in the metric completion $(\overline{Y}_\tau, d_\infty)$ by the contraction mapping theorem; 
note that $x_1^{-1}x_2 \cdot \star_1 = \star_1$ if and only if $\star_1 = \star_2$.
Thus the $\pi_m$-projection of the point~$p_{i,m}$ in~$(T\varphi^{m}, \lambda^{-m} d_\tau)$ converges (as $m \to \infty$) to~$\star_i$ in $(\overline{Y}_\tau, d_\infty)$;
in particular, 
\[ \lim_{m \to \infty} \lambda^{-m} d_\infty(\pi(p_{1,m}), \pi(p_{2,m})) = \lim_{m \to \infty} d_\infty(\pi_m(p_{1,m}), \pi_m(p_{2,m})) = d_\infty(\star_1, \star_2). \]

Let $\tilde \tau \colon T \to T$ be the $\varphi$-equivariant translate of a component of~$\tau^k$.
The interval $[p_{1,m}, p_{2,m}] \subset T$, the $m^{th}$ term in $(p_1, p_2:x_1,x_2;\varphi)_T$, is covered by these $2m+1$ intervals:
\[ \begin{aligned}
&\varphi^{m-1}(x_1) \cdots \varphi(x_1) \cdot [x_1 \cdot p_1, \tilde \tau(p_1)], \ldots,~\varphi^{m-1}(x_1) \cdot [\tilde \tau^{m-2}(x_1 \cdot p_1), \tilde \tau^{m-1}(p_1)],\\
&[\tilde \tau^{m-1}(x_1\cdot p_1), \tilde \tau^{m}(p_1)],~[\tilde \tau^{m}(p_1), \tilde \tau^{m}(p_2)],~[\tilde \tau^{m}(p_2), \tilde \tau^{m-1}(x_2 \cdot p_2)], \\
&\varphi^{m-1}(x_2) \cdot [\tilde \tau^{m-1}(p_2), \tilde \tau^{m-2}(x_2 \cdot p_2)],\ldots,~\varphi^{m-1}(x_2) \cdots \varphi(x_2) \cdot [\tilde \tau(p_2), x_2\cdot p_2].
\end{aligned} \]
Set $D \defeq \max\{ d_\tau(x_i \cdot p_i, \tilde \tau(p_i) : i = 1,2 \}$ and $D' \defeq \frac{D}{\lambda - 1}$.

Recall that $\underset{m' \to \infty}\lim \lambda^{-m'} d_\tau(\tilde \tau^{m'}(p_{1,m}), \tilde \tau^{m'}(p_{2,m})) = d_\infty(\pi(p_{1,m}), \pi(p_{2,m}))$.
For $m' \ge 0$, we get a similar covering of $[p_{1,m+m'}, p_{2,m+m'}]$ by $2m'+1$ intervals with the ``middle'' $[\tilde \tau^{m'}(p_{1,m}), \tilde \tau^{m'}(p_{2,m})]$.
Since $\tilde \tau$ is $\lambda$-Lipschitz with respect to~$d_\tau$, the sum of the $d_\tau$-lengths of all intervals but the middle in this covering is $\le \lambda^{m'} 2D'$.
By the triangle inequality,
\[  \lambda^{-(m+m')} \left| d_\tau(p_{1,m+m'}, p_{2,m+m'}) - d_\tau(\tilde \tau^{m'}(p_{1,m}), \tilde \tau^{m'}(p_{2,m})) \right| \le \lambda^{-m} 2D'. \]
Fix $\epsilon > 0$; then $\lambda^{-m} 2D' < \epsilon$ and $\left| \lambda^{-m} d_\infty(\pi(p_{1,m}), \pi(p_{2,m})) - d_\infty(\star_1, \star_2) \right| < \epsilon$ for some $m \gg 1$.

Similarly, 
\[ \begin{aligned} 
\lambda^{-m}\left| \lambda^{-m'} d_\tau(\tilde \tau^{m'}(p_{1,m}), \tilde \tau^{m'}(p_{2,m})) - d_\infty(\pi(p_{1,m}), \pi(p_{2,m})) \right| &< \epsilon \\
\text{and } 
\left|\lambda^{-(m+m')} d_\tau(p_{1,m+m'}, p_{2,m+m'}) - d_\infty(\star_1, \star_2)\right| &< 3\epsilon \text{ for } m' \gg 1, \\ 
 \text{i.e.}~
 \lim_{m \to \infty} \lambda^{-m} d_\tau(p_{1,m}, p_{2,m}) = d_\infty(\star_1, \star_2).&
 \end{aligned} \]

Let $N(u,v)$ be the number of vertices in an interval $(u,v) \subset T$;
set $N$ to be the maximum of $N(p_1,p_2)$, $N(x_1 \cdot p_1, \tilde \tau(p_1))$, and $N(\tilde \tau(p_2), x_2 \cdot p_2)$.
As~$\tilde \tau$ is a train track, the interval $[p_{1,m}, p_{2,m}]$ is covered by $(2m+1)(N+1)$ leaf segments.

\smallskip
Suppose $\star_1 = \star_2$.
We claim that any leaf segment (of~$\mathcal L^+[\tau]$) in $[p_{1,m}, p_{2,m}]$ has uniformly (in~$m\ge 0$) bounded $d_\tau$-length --- this implies $[p_{1,m}, p_{2,m}]$ has $d_\tau$-length $\le (2m+1)(N+1)B$ for some bounding constant $B \ge 1$.
We mimic Case~2 from the proof of Proposition~\ref{prop-limitloxodromics}.
For the contrapositive, suppose some term $[p_{1,m}, p_{2,m}]$ has a leaf segment with $d_\tau$-length $L > 2(C[\pi] + D')$.
By the train track property, bounded cancellation, and interval covering, $[p_{1, m+m'}, p_{2, m+m'}]$ has a leaf segment with $d_\tau$-length $\ge \lambda^{m'} (L - 2C[\pi] - 2D')$ for $m' \ge 0$; in $(T\varphi^{m+m'}, \lambda^{-(m+m')}d_\tau)$, $[p_{1, m+m'}, p_{2, m+m'}]$ has length $ \ge \lambda^{-m}(L-2C[\pi] - 2D')$.
In the limit (as $m' \to \infty$), $d_\infty(\star_1, \star_2) \ge \lambda^{-m}(L-2C[\pi] - 2D') > 0$.

Suppose $\star_1 \neq \star_2$.
Set $L  \defeq \frac{1}{2} d_\infty(\star_1, \star_2) > 0$;
then $\lambda^{-m}d_\tau(p_{1,m}, p_{2,m}) > L$ for some $m \gg 1$.
By the pigeonhole principle, the interval $[p_{1,m}, p_{2,m}]$ has a leaf segment with $d_\tau$-length $\frac{\lambda^{m}L}{(2m+1)(N+1)}$, which can be arbitrarily large (in~$m$).
\end{proof}

\subsubsection {Nested iterated turns}\label{Subsubsec-nestedturns}

The first part of the previous subsection explains how a line in $(\mathcal Y_\tau, d_\infty)$ determines algebraic iterated turns over~$\mathcal G[\mathcal Y_\tau]$.
We now give a similar discussion for an iterated turn over~$\mathcal T'$. 

Recall how $f, T, T', Y_\tau, F, \tilde \tau$, $\tilde h$, and~$\varphi$ were chosen and~$\lambda$ was redefined in the previous subsection.
Pick points $p_1', p_2' \in T'$ and elements $x_1, x_2 \in F$.
Set $T_m' \defeq T'\varphi^m$, $T_m \defeq T\varphi^m$, $p_{i,0}' \defeq p_i'$, $p_{i,m}' \defeq \varphi^{-1}(x_i)\cdots\varphi^{-m}(x_i) \cdot p_i'$ in~$T_m'$, and $p_{i,m} = f(p_{i,m}')$ for $m \ge 1$ and $i = 1,2$.
By Proposition~\ref{prop-limitturns}, the point~$p_{i,m}$ in $(T_m, \lambda^{-m} d_\tau)$ converges (as $m\to\infty$) to~$\star_i$, the unique fixed point of $x_i^{-1}\cdot \tilde h$ in the metric completion $(\overline{Y}_\tau, d_\infty)$.
The $\lambda$-homothety $h_i \defeq x_i^{-1} \cdot \tilde h$ is $\varphi_i$-equivariant for some automorphism $\varphi_i \colon F \to F$ in the outer class~$[\varphi]$.
Set $G_1 \defeq \operatorname{Stab}_{F}(\star_1)$.

\smallskip
\noindent \emph{Case--a:} $s \defeq x_1^{-1}x_2 \in G_1$.
Suppose~$G_1$ is not trivial, and let~$a_{i,m}$ be the {closest point projection} of~$p_{i,m}'$ to~$T'(\varphi^m(G_1))$ for $m \ge 0$.
As $\tilde h(\star_1) = x_1 \cdot \star_1$ and~$\tilde h$ is $\varphi$-equivariant, we get $T'(\varphi^{m+1}(G_1)) = \varphi^m(x_1) \cdot T'(\varphi^m(G_1))$, $a_{1,m+1} = \varphi^m(x_1) \cdot a_{1,m}$, and 
\[\begin{aligned} a_{2,m+1} &= \varphi^m(x_1) \varphi^m(s) \cdot a_{2,m} \\
&= \varphi^m(x_1)\cdots \varphi(x_1) x_1 \varphi_1^m(s)\cdots \varphi_1^m(s)s \cdot a_{2,0} \quad \text{ for } m \ge 0. \end{aligned}\]
Thus the closest point projection to~$T'(\varphi^m(G_1))$ of the $m^{th}$ term of the given iterated turn $(p_1',p_2':x_1,x_2;\varphi)_{T'}$ is a translate of the $m^{th}$ term in $(a_{1,0}, a_{2,0} : \epsilon, s;\left.\varphi_1\right|_{G_1})_{T'(G_1)}$, where $m \ge 0$ and~$\epsilon$ is the trivial element.
%
%
Hence we have an algebraic iterated turn $(\epsilon, s;\left.\varphi_1\right|_{G_1})_{G_1}$ that is well-defined for the algebraic iterated turn $(x_1, x_2;\varphi)_F$.

\begin{figure}[ht]
 \centering 
 \includegraphics{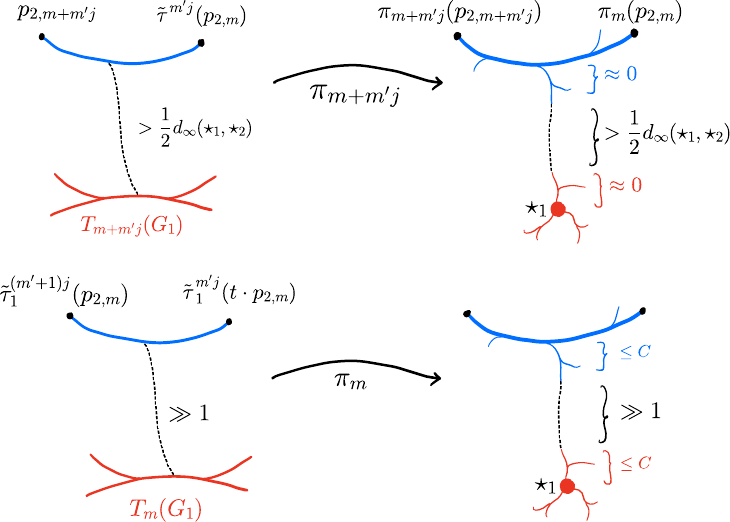}
 \caption{The two figures illustrating certain closest point projections are the same.}
 \label{fig - far lines}
\end{figure}

\smallskip
\noindent \emph{Case--b:} $\star_1 \neq \star_2$. 
Suppose~$G_1$ is not trivial --- the argument is symmetric if $\operatorname{Stab}_F(\star_2)$ is not trivial --- and let~$d$ be the direction at~$\star_1$ containing~$\star_2$. 
By Gaboriau--Levitt index theory, $h_1^j(d) = t \cdot d$ for some $t \in G_1$ and minimal $j \ge 1$.
For $m \gg 1$, $\pi_m(p_{2,m}) =  h_2^{-m}(\pi(p_2))$ is in the direction~$d$ since $h_2^{-m}(\pi(p_2)) \to \star_2$ in $(\overline{Y}_\tau, d_\infty)$.
For $m \gg 1$ and $m' \ge 0$, the interval $[p_{2, m+m'j}, \tilde \tau^{m'j}(p_{2,m})]$ in $(T_{m+m'j}, \lambda^{-m-m'j}d_\tau)$ is disjoint from~$T_{m+m'j}(G_1)$ by bounded cancellation (see Figure~\ref{fig - far lines}, top); or equivalently, the interval $[p_{2,m+m'j}, \tilde \tau^{m'j}(p_{2,m})]$ in~$T_m$ is disjoint from~$T_m(\varphi^{m'j}(G_1))$.
In fact, the $\lambda^{-m}d_\tau$-distance in~$T_m$ from $[p_{2,m+m'j}, \tilde \tau^{m'j}(p_{2,m})]$ to~$T_m(\varphi^{m'j}(G_1))$ can be arbitrarily large for $m' \gg 1$.

Set $z_0 \defeq \epsilon$ and $z_{m'+1} \defeq \varphi^{m'}(x_1)z_{m'}$.
Let $b_{i,m'}'~(i=1,2)$ be the closest point projection of~$p_{i,m+m'j}'$ to~$T_m'(\varphi^{m'j}(G_1)) = z_{m'j} \cdot T_m'(G_1)$ and set $b_{i,m'} \defeq z_{m'j}^{-1} \cdot b_{i,m'}'$ in $T_m'(G_1)$.
Following the definitions, $z_{m'j}^{-1} \cdot p_{1,m+m'j}' = p_{1,m}'$ in $T_m'$ and $z_{m'j}^{-1} \cdot \tilde \tau^{m'j} = \tau_1^{m'j}$ in~$T_m$, where $\tau_1 \defeq x_1^{-1} \cdot \tilde \tau$;
in particular, $b_{1,m'} = b_{1,0}$ for $m' \ge 0$.
Since $h_1^j(d) = t \cdot d$,  bounded cancellation implies the $\lambda^{-m}d_\tau$-distance in~$T_m$ from $[\tau_1^{(m'+1)j}(p_{2,m}), \varphi_1^{m'j}(t) \cdot \tau_1^{m'j}(p_{2,m})]$ to~$T_m(G_1)$ is arbitrarily large for $m' \gg 1$ (see Figure~\ref{fig - far lines}, bottom).

So $[z_{(m'+1)j}^{-1} \cdot p_{2,m+(m'+1)j}, \varphi_1^{m'j}(t)z_{m'j}^{-1} \cdot p_{2,m+m'j}]$ is arbitrarily far from~$T_m(G_1)$ by transitivity.
By bounded cancellation, $[z_{(m'+1)j}^{-1} \cdot p_{2,m+(m'+1)j}', \varphi_1^{m'j}(t)z_{m'j}^{-1}\cdot p_{2, m+m'j}']$ is disjoint from~$T_m'(G_1)$ for $m' \gg 1$, i.e.~$b_{2,m'+1} = \varphi_1^{m'j}(t) \cdot b_{2,m'}$ for $m' \gg 1$.
%
%
Thus, for some $M' \gg 1$, the sequence $[b_{1,M'+m'}, b_{2,M'+m'}]_{m'\ge 0}$ is an iterated turn over~$T_m'(G_1)$ rel.~$\varphi_1^j|_{G_1}$, denoted $(b_{1,M'}, b_{2,M'}: \epsilon, t;\varphi_1^j|_{G_1})_{T_m'(G_1)}$.
The corresponding algebraic iterated turn $(\epsilon, t; \varphi_1^j|_{G_1})_{G_1}$ is well-defined for $(x_1, x_2;\varphi)_F$.

\smallskip
Now suppose $\circ \in (\star_1, \star_2)$ has a nontrivial stabilizer $G_\circ \defeq \operatorname{Stab}_F(\circ)$.
Let~$d_i~(i = 1,2)$ be the direction at~$\circ$ containing~$\star_i$.
By index theory again,~$h_1^{\,l}(\circ) = x \cdot \circ$ and $h_1^{\,l}(d_i) = xs_i \cdot d_i$ for some  $x \in F$, $s_i \in G_\circ$, and minimal $l \ge 1$.
Since~$F$ acts on~$Y_\tau$ with trivial arc stabilizers, the elements $xs_1, xs_2$ are unique and $s_1^{-1}s_2 \in G_\circ$ is independent of the chosen $x \in F$.
For $m \gg 1$, $\pi_m(p_{i,m})$ is in the direction~$d_i$ since $\pi_m(p_{i,m}) \to \star_i$.
A variation of the bounded cancellation argument used in the preceding paragraphs proves the following.
For $m, m' \gg 1$,
the interval $[p_{i,m+m'l}, \tilde \tau^{m'l}(p_{i,m})]$ in $(T_m, \lambda^{-m}d_\tau)$ is far from~$T_m(\varphi^{m'l}(G_\circ))$.

Set $y_0 \defeq \epsilon$, $y_{m'+1} \defeq \varphi_\circ^{m'}(x)y_{m'}$, $\tau_{\circ} \defeq x^{-1} \cdot \tau_1^{\,l}$, and $h_{\circ} \defeq x^{-1} \cdot h_1^{\,l}$ to be $\varphi_{\circ}$-equivariant maps for an automorphism $\varphi_{\circ} \colon F \to F$ in the outer class~$[\varphi_1^j]$.
Let $c_{i,m'}' \in T_m'(\varphi^{m'l}(G_\circ))$ be the closest point projection of~$p_{i,m+m'l}'$ and set $c_{i,m'}'' \defeq z_{m'l}^{-1} \cdot c_{i,m'}' \in T_m'(\varphi_1^{m'l}(G_\circ))$.
Then $c_{i,m'} \defeq y_{m'}^{-1} \cdot c_{i,m'}'' \in T_m'(G_\circ)$ is the closest point projection of $y_{m'}^{-1}z_{m'l}^{-1}\cdot p_{i,m+m'l}'$. 
Since $h_\circ(d_i) = s_i \cdot d_i$, the interval $[\tau_\circ^{m'+1}(p_{i,m}), \varphi_\circ^{m'}(s_i)\cdot \tau_\circ^{m'}(p_{i,m})]$ is arbitrarily far from~$T_m(G_\circ)$ for $m' \gg 1$.
By transitivity, $[y_{m'+1}^{-1}z_{(m'+1)l}^{-1} \cdot p_{i,m+(m'+1)l}, \varphi_\circ^{m'}(s_i)y_{m'}^{-1}z_{m'}^{-1} \cdot p_{i,m+m'l}]$ is arbitrarily far from~$T_m(G_\circ)$.
As before, $[y_{m'+1}^{-1}z_{(m'+1)l}^{-1}\cdot p_{i,m+(m'+1)l}', \varphi_\circ^{m'}(s_i)y_{m'}^{-1}z_{m'}^{-1}\cdot p_{i, m+m'l}']$ is disjoint from~$T_m'(G_\circ)$ for $m' \gg 1$, i.e.~$c_{i,m'+1} = \varphi_\circ^{m'}(s_i)\cdot c_{i,m'}$ for $m' \gg 1$.
Thus, for some $M'' \gg 1$, the sequence $[c_{1,M''+m'}, c_{2,M''+m'}]_{m'\ge 0}$ is an iterated turn over~$T_m'(G_\circ)$ rel.~$\left.\varphi_\circ\right|_{G_\circ}$: $(c_{1,M''}, c_{2,M''}: s_1, s_2; \left.\varphi_\circ\right|_{G_\circ} )_{T_m'(G_\circ)}$.
It is a ``translate'' of the normalized iterated turn: $(c_{1,M''}, c_{2,M''}:\epsilon, s_1^{-1}s_2; \left.\varphi_\circ\right|_{G_\circ})_{T_m'(G_\circ)}$.
The corresponding algebraic iterated turn $(\epsilon, s_1^{-1}s_2; \left.\varphi_\circ\right|_{G_\circ})_{G_\circ}$ is well-defined for $(x_1, x_2;\varphi)_F$ and $\circ \in (\star_1, \star_2)$.

\subsection{Coordinate-free laminations}\label{SubsecCoordfreelams}

We have only defined the stable laminations for an expanding irreducible train track~$[\tau]$ representing~$[\psi]$. 
The free splitting~$\mathcal T$ of~$\mathcal F$ can be seen as a {coordinate system}, and we need a {coordinate-free} definition of stable laminations that applies to all outer automorphisms. 

Fix a proper free factor system~$\mathcal Z$ of~$\mathcal F$ and consider the set~$scv(\mathcal F, \mathcal Z)$ of all free splittings~$\mathcal T'$ of~$\mathcal F$ with $\mathcal F[\mathcal T'] = \mathcal Z$, i.e.~an element of~$\mathcal F$ is $\mathcal T'$-elliptic if and only if it is conjugate to an element of~$\mathcal Z$;
this set with some natural partial order is the \emph{spine of relative outer space}~\cite{CV86}.
For any pair of free splittings $\mathcal T_1, \mathcal T_2 \in scv(\mathcal F, \mathcal Z)$, there are \underline{changes of coordinates}, equivariant PL-maps $\mathcal T_1 \rightleftarrows \mathcal T_2$.
We saw in the discussion following Claim~\ref{claim-liftembed} that a change of coordinates $f\colon \mathcal T_1 \to \mathcal T_2$ induces a canonical homeomorphism $f_* \colon \mathbb R(\mathcal T_1) \to \mathbb R(\mathcal T_2)$ on the space of lines.
Denote the canonical homeomorphism class of $\mathbb R(\mathcal T_1) \cong \mathbb R(\mathcal T_2)$ by $\mathbb R(\mathcal F, \mathcal Z)$.
If~$\mathcal Z$ is the trivial free factor system of~$\mathcal F$, then we denoted the canonical homeomorphism class by $\mathbb R(\mathcal F)$ instead.

\smallskip
Fix an automorphism~$\psi \colon \mathcal F \to \mathcal F$ and a $[\psi]$-invariant proper free factor system~$\mathcal Z$.
Let $\psi_* \colon \mathbb R(\mathcal F, \mathcal Z) \to \mathbb R(\mathcal F, \mathcal Z)$ be the {canonical} induced homeomorphism on the space of lines: $f_* \circ g_{1*} = g_{2*} \circ f_*$ for any $\mathcal T_1, \mathcal T_2 \in scv(\mathcal F, \mathcal Z)$, equivariant PL-map $f \colon \mathcal T_1 \to \mathcal T_2$, and $\psi$-equivariant PL-maps $g_i \colon \mathcal T_i \to \mathcal T_i~(i=1,2)$.
A line $[l] \in \mathbb R(\mathcal F, \mathcal Z)$ \underline{weakly $\psi_*$-limits} to a lamination~$\Lambda \subset \mathbb R(\mathcal F, \mathcal Z)$ if the sequence $(\psi_*^n[l])_{n \ge 0}$ weakly limits to~$\Lambda$.

A coordinate-free definition of stable laminations boils down to characterizing the lines of a stable $\mathcal T$-lamination for~$[\tau]$ in a way that is independent of coordinates.
For the rest of the section, assume there is an equivariant PL-map $(\mathcal T, d) \to (\mathcal Y, \delta)$ and consider the canonical embedding $\mathbb R(\mathcal Y, \delta) \subset \mathbb R(\mathcal T)$.
Note that a lamination $\Lambda \subset \mathbb R(\mathcal Y, \delta)$ is contained in a canonical lamination~$\mathcal L \subset \mathbb R(\mathcal T)$: set~$\mathcal L$ to be the closure of~$\Lambda$ in $\mathbb R(\mathcal T)$.

\begin{claim*}[cf.~{\cite[Lemma~1.9(2)]{BFH97}}]
A line is quasiperiodic in $\mathbb R(\mathcal Y, \delta)$ if it is quasiperiodic in~$\mathbb R(\mathcal T)$.
{\rm(exercise)} \qed
\end{claim*}

\noindent So quasiperiodicity is a well-defined property for a line in $\mathbb R(\mathcal F, \mathcal Z)$;
moreover, the induced homeomorphism $\psi_* \colon \mathbb R(\mathcal F, \mathcal Z) \to \mathbb R(\mathcal F, \mathcal Z)$ preserves quasiperiodicity for any automorphism $\psi\colon \mathcal F \to \mathcal F$ that preserves~$\mathcal Z$ (up to conjugacy).

\smallskip
Suppose there is an expanding irreducible train track $\tau \colon \mathcal T \to \mathcal T$ for~$\psi$ with $\mathcal F[\mathcal T] = \mathcal Z$.
Recall that the eigenlines of~$[\tau^k]$ (for some~$k \ge 1$) are constructed by iterating an expanding edge;
more precisely, an eigenline~$[l]$ of~$[\tau^k]$ is the union $\bigcup_{n \ge 1} \tau^{kn}(\mathcal F \cdot e)$ for some edge~$e \subset l$.
The leaf segments $\tau^{kn}(e)$ determine a neighbourhood basis for~$[l]$ in the space of lines.

For a line $[l] \in \mathbb R(\mathcal F, \mathcal Z)$, a subset~$U \subset \mathbb R(\mathcal F, \mathcal Z)$ is a \emph{$\psi_*^k$-attracting neighbourhood} of~$[l]$ if $\psi_*^k(U) \subset U$ and $\{\psi_{*}^{kn}(U):n \ge 1\}$ is a neighbourhood basis for~$[l]$ in the space of lines.
A \underline{stable lamination} for~$[\psi]$ rel.~$\mathcal Z$ is the closure of a quasiperiodic line in $\mathbb R(\mathcal F, \mathcal Z)$ with a $\psi_*^k$-attracting neighbourhood for some $k \ge 1$.
Note that the homeomorphism $\psi_* \colon \mathbb R(\mathcal F, \mathcal Z) \to \mathbb R(\mathcal F, \mathcal Z)$ permutes the stable laminations for~$[\psi]$ rel.~$\mathcal Z$ and, by Lemma~\ref{lem-stableminperf}, each stable $\mathcal T$-lamination for~$[\tau]$ is identified with some stable lamination for~$[\psi]$ rel.~$\mathcal Z$.
Let~$\mathcal L^+_{\mathcal Z}[\psi]$ be the union of stable laminations for~$[\psi]$ rel.~$\mathcal Z$.

\begin{lem}[cf.~{\cite[Lemma~1.12]{BFH97}}]\label{lem-coordinatefreestable}
Let~$\psi\colon \mathcal F \to \mathcal F$ be an automorphism, $\tau \colon \mathcal T \to \mathcal T$ an expanding irreducible train track for~$\psi$, and $\mathcal Z \defeq \mathcal F[\mathcal T]$.
The stable laminations~$\mathcal L^+[\tau]$ for~$[\tau]$ are identified with the stable laminations~$\mathcal L^+_{\mathcal Z}[\psi]$ for~$[\psi]$ rel.~$\mathcal Z$.
\end{lem}

\noindent So~$\mathcal L^+_{\mathcal Z}[\psi]$ is a lamination system whose finitely many components are the stable laminations for~$[\psi]$ rel.~$\mathcal Z$, and these are transitively permuted by $\psi_* \colon \mathbb R(\mathcal F, \mathcal Z) \to \mathbb R(\mathcal F, \mathcal Z)$.

\begin{proof}[Sketch of proof]
Suppose a quasiperiodic line~$[l]$ in~$\mathcal T$ has a $\tau_*^k$-attracting neighbourhood~$U$ for some $k \ge 1$.
This forces any $\mathcal T$-loxodromic conjugacy class~$[x]$ with axis in~$U$ to have a translation distance that (eventually) grows under forward $[\psi^k]$-iteration.
In particular, the conjugacy class~$[x]$ is $\mathcal Y_\tau$-loxodromic, and the line~$[l]$, a weak $\psi_*^k$-limit of the $\mathcal T$-axis for~$[x]$, is a leaf in~$\mathcal L^+[\tau]$ by Proposition~\ref{prop-limitloxodromics}.
\end{proof}

The stable laminations $\mathcal L_{\mathcal Z}^+[\psi]$ are in the subspace $\mathbb R(\mathcal Y_\tau, d_\infty) \subset \mathbb R(\mathcal F, \mathcal Z)$.

\subsection{Constructing limit forests (2)}\label{SubsecLimittrees2}

This chapter has thus far focused on automorphims with expanding irreducible train tracks.
For the rest of the chapter, we extend our focus to all automorphisms.

Fix an automorphism $\psi \colon \mathcal F \to \mathcal F$, set $\mathcal F_1 \defeq \mathcal F$, $\psi_1 \defeq \psi$, and let $\mathcal Z$ be a $[\psi_1]$-invariant proper free factor system.
By Theorem~\ref{thm-irredtt}, there is an irreducible train track $\tau_1 \colon \mathcal T_1 \to \mathcal T_1$ for~$\psi_1$.
By $\psi_1$-equivariance of~$\tau_1$, the nontrivial vertex stabilizers of~$\mathcal T_1$ determine a $[\psi_1]$-invariant proper free factor system $\mathcal F_2 \defeq \mathcal F[\mathcal T_1]$.
The restriction of~$\psi_1$ to~$\mathcal F_2$ determines a unique outer class of automorphisms $\psi_2 \colon \mathcal F_2 \to \mathcal F_2$.
We can repeatedly apply Theorem~\ref{thm-irredtt} to $\psi_{i+1}~(i\ge 1)$ as long as $\lambda[\tau_i] = 1$ and $\mathcal F_{i+1} \defeq \mathcal F[\mathcal T_i]$ contains a noncyclic component.
Bass-Serre theory implies this process must stop;
we end up with a maximal sequence $(\tau_i)_{i=1}^n$ of irreducible train tracks with $\lambda[\tau_i] = 1$ for $1 \le i < n$ --- such a maximal sequence is called a \underline{descending sequence} of irreducible train tracks for~$[\psi]$ rel.~$\mathcal Z$.

This leads to our {working definition} of growth type: $[\psi]$~is polynomially growing rel.~$\mathcal Z$ if and only if $\lambda[\tau_n] = 1$~\cite[Proposition~III.1]{Mut24}.
For automorphisms that are polynomially growing rel.~$\mathcal Z$, define the limit forest to be degenerate.

\smallskip
Suppose~$[\psi]$ is exponentially growing rel.~$\mathcal Z$ and $\left(\tau_i \colon \mathcal T_i \to \mathcal T_i\right)_{i=1}^n$ is a descending sequence of irreducible train tracks for~$[\psi]$ rel.~$\mathcal Z$.
Sections~\ref{SubsecLimittrees}--\ref{SubsecLaminations} already cover the case $n = 1$, so we may assume $n > 1$ for the rest of the chapter.
Set $\lambda \defeq \lambda[\tau_n] > 1$, $\mathcal T_n^\circ \defeq \mathcal T_n$, $\tau_n^\circ \defeq \tau_n$, and $d_n^\circ$ the eigenmetric on~$\mathcal T_n^\circ$ for~$\tau_n^\circ$.
For $1 \le i < n$, we inductively form an equivariant \emph{simplicial blow-up}~$\mathcal T_{i}^\circ$ of~$\mathcal T_{i}$ rel.~$\mathcal T_{i+1}^\circ$:
the vertices with nontrivial stabilizers are equivariantly replaced by copies of corresponding components of~$\mathcal T_{i+1}^\circ$ and arbitrary vertices in~$\mathcal T_{i+1}^\circ$ are chosen as attaching points for the edges of~$\mathcal T_i$.
Let~$\tau_i^\circ \colon \mathcal T_i^\circ \to \mathcal T_i^\circ$ be the topological representative for~$\psi_i$ induced by~$\tau_i$ and~$\tau_{i+1}^\circ$.
As $\tau_i$ is a simplicial automorphism, we can make~$\tau_i^\circ$ a $\lambda$-Lipschitz map by assigning the same large enough length to the edges of~$\mathcal T_i$ in the blow-up~$\mathcal T_i^\circ$ when extending the metric~$d_{i+1}^\circ$ on~$\mathcal T_{i+1}^\circ$ to a metric~$d_i^\circ$ on~$\mathcal T_i^\circ$.
The topological representative $\tau^\circ \defeq \tau_i^\circ$ on $\mathcal T^\circ \defeq \mathcal T_1^\circ$ is an \underline{equivariant blow-up} of the descending sequence $(\tau_i)_{i=1}^n$.
Set $d^\circ \defeq d_1^\circ$ and identify $(\mathcal T_i^\circ, d_i^\circ)$ with the characteristic subforest of $(\mathcal T^\circ, d^\circ)$ for~$\mathcal F_{i}$.
We will abuse terminology and refer to~$d^\circ$ as the {eigenmetric} as well.
Translates of edges in $\mathcal T^\circ$ coming from~$\mathcal T_i$ form the \underline{$i^{th}$~stratum of~$\mathcal T^\circ$}: the $n^{th}$~stratum is \underline{exponential} while the rest are \underline{(relatively) polynomial}.

As in Section~\ref{SubsecLimittrees}, the maps $\tau^{\circ m}\colon(\mathcal T^\circ, d^\circ) \to (\mathcal T^\circ \psi^m, \lambda^{-m}d^\circ)$ converge (as $m\to \infty$) to an equivariant metric surjection $\pi^\circ \colon (\mathcal T^\circ, d^\circ) \to (\mathcal Y, \delta)$. 
The map~$\tau^\circ$ induces a $\psi$-equivariant $\lambda$-homothety $h \colon (\mathcal Y, \delta) \to (\mathcal Y, \delta)$ and $\pi^\circ$ semiconjugates~$\tau^\circ$ to~$h$.
By restricting to~$\mathcal T_i^\circ$, we have also constructed an equivariant metric surjection $\pi_i^\circ \colon (\mathcal T_i^\circ, d^\circ) \to (\mathcal Y_i, \delta)$ and $\psi_i$-equivariant $\lambda$-homothety~$h_i$ on $(\mathcal Y_i, \delta)$ for $2 \le i \le n$.

The $\mathcal F_n$-forest $(\mathcal Y_n, \delta)$ is the limit forest for~$[\tau_n^\circ]$; 
so it is a nondegenerate minimal $\mathcal F_n$-forest with trivial arc stabilizers.
For induction, assume $(\mathcal Y_i, \delta)$ is a nondegenerate minimal $\mathcal F_i$-forest with trivial arc stabilizers for $2 \le i \le n$.
Equivariantly collapse~$\mathcal T_2^\circ$ in $(\mathcal T^\circ, d^\circ)$ to get the $\mathcal F$-forest $(\mathcal T_1, d_1)$.
For $m \ge 0$, the metric free splitting $(\mathcal T^\circ\psi^m, \lambda^{-m} d^\circ)$ is an equivariant {metric} blow-up of $(\mathcal T_1 \psi^m, \lambda^{-m} d_1)$ rel. $(\mathcal T_2^\circ\psi_2^{m}, \lambda^{-m} d^\circ)$.
Since $\tau_1 \colon (\mathcal T_1, d_1) \to (\mathcal T_1 \psi, d_1)$ is an equivariant isometry, the limit $(\mathcal Y, \delta)$ is equivariantly isometric to an equivariant metric blow-up of $(\mathcal T_1, d_1)$ rel.~$(\mathcal Y_2, \delta)$ whose top stratum (edges coming from~$\mathcal T_1$) have then been equivariantly collapsed, also known as a \emph{graph of actions (with degenerate skeleton)} --- more details are given in the next subsection.
Thus $(\mathcal Y, \delta)$ is a nondegenerate minimal $\mathcal F$-forest with trivial arc stabilizers.
See~\cite[Theorem~IV.1]{Mut24} for a direct construction of $(\mathcal Y, \delta)$ as a graph of actions.
%
%
This sketches the general case of Proposition~\ref{prop-limittree}.
The $\mathcal F$-forest $(\mathcal Y,\delta)$ is the \underline{limit forest} for~$[\tau_i]_{i=1}^n$.

\subsubsection {Decomposing limit forests}

We now give a hierarchical decomposition of the limit forest $(\mathcal Y, \delta)$ and its space of lines.

Choose an iterate~$[\tau_1^{k'}]$ that fixes all $\mathcal F$-orbits of branches in~$\mathcal T_1$.
Pick an edge~$e$ in~$\mathcal T_1$ and one of its endpoints~$p$.
Replace~$\psi^{k'}$ with an automorphism in its outer class~$[\psi^{k'}]$ if necessary and assume~$\tau_1^{k'}$ fixes~$p$ and~$e$.
Identify $(\mathcal T^\circ \psi^{mk'}, \lambda^{-mk'} d^\circ)$ with an equivariant metric blow-up of $(\mathcal T_1, \lambda^{-mk'} d_1)$ rel.~$(\mathcal T_2^\circ \psi_2^{mk'}, \lambda^{-mk'} d^\circ)$ for $m \ge 0$, then let~$p_m \in \mathcal T_2^\circ\psi_2^{mk'}$ be the attaching point of~$e$ to~$\mathcal T_2^\circ \psi_2^{mk'}$ corresponding to the endpoint~$p$.
Since $\tau_1^{k'}$ fixes~$e$ and~$p$, we get $p_m = p_0$ for $m \ge 1$.
As in the first part of the proof for Proposition~\ref{prop-limitturns}, the sequence $(p_m)_{m \ge 0}$ converges to the unique fixed point~$\star$ of~$h_2^{k'}$ in the metric completion $(\overline{\mathcal Y}_2, \delta)$.
So, in the description of $(\mathcal Y, \delta)$ as a graph of actions, the edge~$e$ is collapsed and identified with~$\star$.
Thus the closure~$\widehat{\mathcal Y}_2$ of~$\mathcal Y_2$ in $(\mathcal Y, \delta)$ is the union of~$\mathcal Y_2$ with the $\mathcal F_2$-orbits of attaching points~$\star$ as the pair $(e, p)$ ranges over the $\mathcal F$-orbit representatives~$e$ of edges and their endpoints~$p$.
For the same reasons, we inductively get a similar description of the closure~$\widehat{\mathcal Y}_{i+1}$ of~$\mathcal Y_{i+1}$ in $(\mathcal Y_i, \delta)$ for $2 \le i < n$.

\begin{rmk}
Constructing $(\mathcal Y, \delta)$ directly by iterating~$\tau^\circ$ allows us to lift metric properties of $(\mathcal Y,\delta)$ to dynamical properties of~$\tau^\circ$ through the semiconjugacy $\pi^\circ \circ \tau^\circ = h \circ \pi^\circ$;
this viewpoint is used in the Section~\ref{SubsecLaminations2}.
On the other hand, constructing $(\mathcal Y, \delta)$ directly as we did in~\cite[Theorem~IV.1]{Mut24} (and sketched in this subsection) gives us a nice structural description of intervals in the limit forest.
This is explained in the next subsection and will be a key component of Chapter~\ref{SecTopDomLimit}!
\end{rmk}

For $1 < i \le n$, any two translates of $\mathcal T_i^\circ \subset \mathcal T^\circ$ by elements of~$\mathcal F$ either coincide or are disjoint by construction.
This induces a canonical closed embedding of $\mathbb R(\mathcal F_{i}, \mathcal Z) $ into $\mathbb R(\mathcal F, \mathcal Z)$ (exercise).
Similarly, any two intersecting translates of $\mathcal Y_i \subset \mathcal Y$ by elements of~$\mathcal F$ either coincide or have degenerate intersection.
This also induces a canonical closed embedding $\mathbb R(\mathcal Y_i, \delta) \subset \mathbb R(\mathcal Y, \delta)$.
Finally, the constructed equivariant metric map~$\pi^\circ$ induces a canonical embedding of the topological pair $(\mathbb R(\mathcal Y, \delta), \mathbb R(\mathcal Y_i, \delta))$ into $(\mathbb R(\mathcal F, \mathcal Z), \mathbb R(\mathcal F_i, \mathcal Z))$.

\subsubsection {Intervals in limit forests}

Here is an inductive description of intervals in the limit forest $(\mathcal Y, \delta)$ in terms of the limit forest for~$[\tau_n^\circ]$.
For $1 \le i \le n$, the characteristic subforest $(\mathcal Y_i, \delta)$ of $(\mathcal Y, \delta)$ for~$\mathcal F_{i}$ is the limit forest for~$[\tau_j]_{j=i}^n$.
For $1 < i \le n$, let~$\widehat{\mathcal Y}_i$ be the closure of~$\mathcal Y_i$ in $(\mathcal Y_{i-1}, \delta)$.

It follows from the blow-up (and collapse) description of~$\mathcal Y_{i-1}$ that its closed intervals are finite concatenations of closed intervals in translates of~$\widehat{\mathcal Y}_i$. 
As shown in the previous subsection, the $\mathcal F_{i}$-orbits~$[p]$ of points in $\widehat{\mathcal Y}_i \setminus \mathcal Y_i$ are fixed by the extension of~$h_i^{k'}$ to~$\widehat{\mathcal Y}_i$ for some $k' \ge 1$.
As $p \notin \mathcal Y_i$, it has exactly one direction~$d_p$ in~$\widehat{\mathcal Y}_i$.
This direction's $\mathcal F_{i}$-orbit~$[d_p]$ is also fixed (setwise) by the expanding homothety~$h_i^{k'}$, and~$d_p$ determines a \underline{singular eigenray}~$\rho_p \subset \widehat{\mathcal Y}_i$ of~$[h_i^{k'}]$ based at~$p$.
For any point $q \in {\mathcal Y}_i$, the closed interval $[p,q] \subset \widehat{\mathcal Y}_i$ is a concatenation of an initial segment of the singular eigenray~$\rho_p$ and a closed interval in~${\mathcal Y}_i$;
therefore, closed intervals in~${\mathcal Y}_{i-1}$ are finite concatenations of translates of closed intervals in~${\mathcal Y}_i$ and initial segments of singular eigenrays of~$[h_i^{k'}]$ for some $k' \ge 1$.

Let $\mathcal L_{\mathcal Z}^+[\psi_n] = \mathcal L^+[\tau_n]$ be the $k$-component stable laminations for $[\tau_n^\circ] = [\tau_n]$ and $\oplus_{j=1}^k \delta_j$ the factored $\mathcal F_n$-invariant convex metric on~$\mathcal Y_n$ indexed by components $\Lambda_j^+ \subset \mathcal L_{\mathcal Z}^+[\psi_n]$.
By the inductive description of intervals in~$\mathcal Y$ and the fact~$h_n^k$ is a $\lambda^k$-homothety with respect to each factor~$\delta_j$, we get:~$\delta_j$ equivariantly extends to~$\mathcal Y$; $\delta = \oplus_{j=1}^k \delta_j$ is a factored $\mathcal F$-invariant convex metric on~$\mathcal Y$; and~$h^k$ is a $\lambda^k$-homothety with respect to each factor~$\delta_j$.

\smallskip
The lamination $\mathcal L_{\mathcal Z}^+[\psi_n] \subset \mathbb R(\mathcal Y_n, \delta)$ can be seen as a $(\mathcal Y, \delta)$-lamination since $\mathbb R(\mathcal Y_n, \delta)$ is a closed subspace of $\mathbb R(\mathcal Y, \delta)$.
Note that closed edges of $\mathcal T_n = \mathcal T_n^\circ$ are leaf segments (of~$\mathcal L_{\mathcal Z}^+[\psi_n]$);
thus any closed interval in~$\mathcal T_n^\circ$ is a finite concatenation of leaf segments.
As the equivariant PL-map $\pi_n^\circ \colon (\mathcal T_n^\circ, d^\circ) \to (\mathcal Y_n, \delta)$ is surjective and isometric on leaf segments, we get: 

\begin{lem}\label{lem-leafsegments}
Let~$\tau_n \colon \mathcal T_n \to \mathcal T_n$ be an expanding irreducible train track and~$(\mathcal Y_n, \delta)$ its limit forest.
Any closed interval in~$\mathcal Y_n$ is a finite concatenation of leaf segments of~$\mathcal L^+[\tau_{n}]$.\qed
\end{lem}

This lemma no longer holds when $n \ge 2$ and we consider closed intervals in~$\widehat{\mathcal Y}_n$.
To account for this, let \underline{$n^{th}$ level leaf blocks} in~$\mathcal Y$ be leaf segments.
By the lemma, any interval of~$\mathcal Y_n$ is a finite concatenation of~$n^{th}$ level leaf blocks.

Inductively define the \underline{$(i-1)^{st}$ level leaf blocks} in~$\mathcal Y~(1 < i \le n)$ to be the $i^{th}$ level leaf blocks or (translates of) closed intervals in singular eigenrays~$\rho \subset \widehat{\mathcal Y}_i$ of $[h_{i}]$-iterates.
By the earlier description of intervals and induction hypothesis, any interval of~$\mathcal Y_{i-1}$ is a finite concatenation of $(i-1)^{st}$ level leaf blocks.
The $1^{st}$ level leaf blocks are simply \underline{leaf blocks} of~$\mathcal L_{\mathcal Z}^+[\psi_{n}]$.
Altogether, we have a generalization of Lemma~\ref{lem-leafsegments} in terms of leaf blocks:

\begin{lem}\label{lem-leafblocks}
Let~$(\tau_i \colon \mathcal T_i \to \mathcal T_i)_{i=1}^n$ be a descending sequence of irreducible train tracks for an automorphism $\psi \colon \mathcal F \to \mathcal F$ and~$(\mathcal Y, \delta)$ the corresponding limit forest.
Any closed interval in~$\mathcal Y$ is a finite concatenation of leaf blocks of~$\mathcal L_{\mathcal Z}^+[\psi_{n}]$, where $\mathcal Z \defeq \mathcal F[\mathcal T_n]$. \qed
\end{lem}

\subsection{Stable laminations (2)}\label{SubsecLaminations2}

Fix an automorphism $\psi\colon \mathcal F \to \mathcal F$ with an invariant proper free factor system~$\mathcal Z'$.
Let $\left(\tau_i \colon \mathcal T_i \to \mathcal T_i\right)_{i=1}^n$ be a descending sequence of irreducible train tracks rel.~$\mathcal Z'$ with $\lambda[\tau_n] > 1$, $(\mathcal Y,\delta)$ be the limit forest for~$[\tau_i]_{i=1}^n$,~$\mathcal T^\circ$ an equivariant blow-up of free splittings $(\mathcal T_i)_{i=1}^n$ with eigenmetric~$d^\circ$,
and $\mathcal Z \defeq \mathcal F[\mathcal T^\circ]$. 
The characteristic convex subsets of~$\mathcal T^\circ$ for~$\mathcal F_n \defeq \mathcal F[\mathcal T_{n-1}]$ are identified with the free splitting~$\mathcal T_n$.

\begin{claim}\label{claim-anystable}
The stable laminations~$\mathcal L^+_{\mathcal Z}[\psi_{n}]$ for~$[\psi_n]$ in $\mathbb R(\mathcal F_{n}, \mathcal Z)$ are identified with the stable laminations~$\mathcal L^+_{\mathcal Z}[\psi]$ for~$[\psi]$ in $\mathbb R(\mathcal F, \mathcal Z)$.
\end{claim}

\noindent Note that $\mathcal L_{\mathcal Z}^+[\psi] = \mathcal L_{\mathcal Z}^+[\psi_n]$ is in the subspace $\mathbb R(\mathcal Y_n, \delta) \subset \mathbb R(\mathcal Y, \delta) \subset \mathbb R(\mathcal F, \mathcal Z)$.

\begin{proof}[Sketch of proof]
Since $\lambda[\tau_i]=1$ for $i < n$, no quasiperiodic line in $\mathbb R(\mathcal F, \mathcal F_{n})$ has a $\psi_*^k$-attracting neighbourhood for any $k \ge 1$.
Thus any stable lamination for~$[\psi]$ in $\mathbb R(\mathcal F, \mathcal Z)$ is contained in $\mathbb R(\mathcal F_{n}, \mathcal Z)$ and corresponds to a stable lamination for~$[\psi_{n}]$.
\end{proof}

We generalize Proposition~\ref{prop-limitturns} by characterizing limits of iterated turns over~$\mathcal T^\circ$:

\begin{thm}\label{thm-limitturns}
Let~$\psi\colon \mathcal F \to \mathcal F$ be an automorphism with an invariant proper free factor system~$\mathcal Z'$, $\left(\tau_i \colon \mathcal T_i \to \mathcal T_i\right)_{i=1}^n$ a descending sequence of irreducible train tracks rel.~$\mathcal Z'$ with $\lambda[\tau_n]>1$, $(\mathcal Y, \delta)$ the limit forest for~$[\tau_i]_{i=1}^n$, and~$\mathcal T^\circ$ an equivariant blow-up of free splittings~$(\mathcal T_i)_{i=1}^n$ with eigenmetric~$d^\circ$.
Choose a nondegenerate component~$T^\circ \subset \mathcal T^\circ$, corresponding components $F \subset \mathcal F$, $Y \subset \mathcal Y$, and a positive iterate~$\psi^k$ that preserves~$F$.
Let $\tilde h \colon (Y, \delta) \to (Y, \delta)$ be the $\varphi$-equivariant $\lambda$-homothety, where~$\varphi$ is in the outer class~$[\left.\psi^k\right|_F]$ and $\lambda \defeq (\lambda[\tau_n])^k$.
Finally, for $\iota = 1,2$, pick $p_\iota \in T^\circ$ and $x_\iota \in F$.

\smallskip
The point~$p_{\iota,m} \defeq \varphi^{-1}(x_\iota)\cdots \varphi^{-m}(x_\iota) \cdot p_\iota$ in $(T^\circ\varphi^m, \lambda^{-m}d^\circ)$ converges to~$\star_\iota$ in $(\overline{Y}, \delta)$ as $m \to \infty$, where~$\star_\iota$ is the unique fixed point of $x_\iota^{-1} \cdot \tilde h$ in the metric completion $(\overline Y, \delta)$.


If $x_1^{-1}x_2$ fixes~$\star_1$, then $\star_1 = \star_2$ and the term $[p_{1,m}, p_{2,m}]~(m \ge 0)$ of the iterated turn $(p_1, p_2: x_1, x_2; \varphi)_{T^\circ}$ has $d^\circ$-length $ \le \alpha(m)$ for some (degree~$n$) polynomial~$\alpha$.
Otherwise, $\star_1 \neq \star_2$ and the iterated turn weakly limits to a component of~$\mathcal L_{\mathcal Z}^+[\psi]$, where $\mathcal Z \defeq \mathcal F[\mathcal T^\circ]$.
\end{thm}

\noindent 
An iterated turn $[p_{1,m}, p_{2,m}]_{m \ge 0}$ weakly limits to a component~$\Lambda^+ \subset \mathcal L_{\mathcal Z}^+[\psi]$ if the term $[p_{1,m}, p_{2,m}]$ contains a leaf segment of~$\Lambda^+$ with arbitrarily large $d^\circ$-length as $m \to \infty$.

\begin{proof}[Sketch of proof]
Let $\tilde \tau^{\circ} \colon (T^\circ, d^\circ) \to (T^\circ, d^\circ)$ be the $\varphi$-equivariant $\lambda$-Lipschitz topological representative induced by the irreducible train tracks~$(\tau_i)_{i=1}^n$ and $\pi^\circ \colon (T^\circ, d^\circ) \to (Y, \delta)$ the equivariant metric map constructed using $\tilde \tau^\circ$-iteration.
Even though~$\pi^\circ$ may fail to be a PL-map, it still has a cancellation constant $C[\pi^\circ] \ge 0$ as a limit of equivariant metric maps with uniformly bounded cancellation constants.
The proof of the first part is the same as in Proposition~\ref{prop-limitturns} using~$\pi^\circ$,~$\tilde \tau^{\circ}$, and the $\varphi$-equivariant $\lambda$-homothety~$\tilde h$.

The interval $[p_{1,m}, p_{2,m}] \subset T^\circ$, a term in the sequence $(p_1, p_2:x_1,x_2;\varphi)_{T^\circ}$, is covered by certain $2m+1$ intervals as in the proof of Proposition~\ref{prop-limitturns}.
Since~$\tilde \tau^{\circ}$ is induced by a descending sequence $(\tau_i)_{i=1}^n$ of irreducible train tracks, the intervals $[\tilde \tau^{\circ (l-1)}(x_1 \cdot p_1), \tilde \tau^{\circ l}(p_1)]$, $[\tau^{\circ l}(p_1), \tilde \tau^{\circ l}(p_2)]$, and $[\tilde \tau^{\circ l}(p_2), \tilde \tau^{\circ (l-1)}(x_2 \cdot p_2)]$ are covered by~$\alpha(l)$ polynomial strata edges and leaf segments (of~$\mathcal L_{\mathcal Z}^+[\psi]$) for some degree~$(n-1)$ polynomial~$\alpha$.
So the interval $[p_{1,m}, p_{2,m}]$ is covered by $\alpha(m) + \sum_{l=1}^{m} 2\alpha(l)$ polynomial strata edges and leaf segments (of~$\mathcal L_{\mathcal Z}^+[\psi]$).
Note that $\alpha(m) + \sum_{l=1}^{m} 2\alpha(l) \le \beta(m)$ for some degree~$n$ polynomial~$\beta$.

\smallskip
Assume $\star_1 = \star_2$, where $\star_\iota$ is the unique fixed point of $\tilde h_\iota \defeq x_\iota^{-1} \cdot \tilde h$ in metric completion $(\overline Y, \delta)$ for $\iota = 1,2$.
The proof given in Proposition~\ref{prop-limitturns} implies there is a uniform bound on the $d^\circ$-length of leaf segments in $[p_{1,m}, p_{2,m}]$.
Consequently, the $d^\circ$-length of $[p_{1,m}, p_{2,m}]$ is $\le \beta(m) B$ for some constant $B \ge 1$.

Assume $\star_1 \neq \star_2$.
Set $L \defeq \frac{1}{2} \delta(\star_1, \star_2) > 0$;
then $\delta(\tilde h_1^{-m}(\pi^\circ(p_1)), \tilde h_2^{-m}(\pi^\circ(p_2))) > L$ 
and $d^\circ(p_{1,m}, p_{2,m}) > \lambda^{m} L$ for $m \gg 1$.
The contribution of polynomial strata to the $d^\circ$-length of $[p_{1,m}, p_{2,m}]$ is at most~$\beta(m) B'$ for some constant $B' \ge 1$;
the exponential stratum edges intersecting the interval are covered by~$\beta(m)$ leaf segments.
By the pigeonhole principle, the interval $[p_{1,m}, p_{2,m}]$, a term in the iterated turn $(p_1, p_2: x_1, x_2; \varphi)_{T^\circ}$, has a leaf segment of $d^\circ$-length $\ge \frac{\lambda^m L - \beta(m)B'}{\beta(m)} \gg 1$.
Quasiperiodicity implies the iterated turn weakly limits to a component of~$\mathcal L_{\mathcal Z}^+[\psi]$.
\end{proof}

\begin{rmk}
The argument given in Subsection~\ref{Subsubsec-nestedturns} applies in this general context involving a descending sequence of irreducible train tracks; 
it describes how an iterated turn over~$\mathcal F$ determines (nested) iterated turns over~$\mathcal G[\mathcal Y]$.
\end{rmk}

As in Proposition~\ref{prop-limitloxodromics}, we can characterize the elements in~$\mathcal F$ that are $\mathcal Y$-loxodromic:


\begin{thm}\label{thm-limitloxodromics}
Let~$\psi\colon \mathcal F \to \mathcal F$ be an automorphism with an invariant proper free factor system~$\mathcal Z'$, $\left(\tau_i \colon \mathcal T_i \to \mathcal T_i\right)_{i=1}^n$ a descending sequence of irreducible train tracks rel.~$\mathcal Z'$ with $\lambda \defeq \lambda[\tau_n] > 1$, $(\mathcal Y, \delta)$ the limit forest for~$[\tau_i]_{i=1}^n$, $\mathcal T^\circ$ an equivariant blow-up of free splittings~$(\mathcal T_i)_{i=1}^n$, 
and $\mathcal Z \defeq \mathcal F[\mathcal T^\circ]$.

If $x \in \mathcal F$ is a $\mathcal T^\circ$-loxodromic element, then the following statements are equivalent:
\begin{enumerate}
\item\label{thm-limitlox-cond-lox} the element~$x$ is $\mathcal Y$-loxodromic;
\item\label{thm-limitlox-cond-exp} the element~$x$ $[\psi]$-grows exponentially rel.~$\mathcal Z$ with rate~$\lambda$; and
\item\label{thm-limitlox-cond-lim} the axis for the conjugacy class~$[x]$ in $\mathbb R(\mathcal F, \mathcal Z)$ weakly $\psi_*$-limits to~$\mathcal L^+_{\mathcal Z}[\psi]$.
\end{enumerate}
The restriction of~$\psi$ to the $[\psi]$-invariant subgroup system~$\mathcal G[\mathcal Y]$ of $\mathcal Y$-point stabilizers is polynomially growing rel.~$\mathcal Z$ with degree~$< n$.
\end{thm}

\begin{proof}[Sketch of proof]
Set $\lambda \defeq \lambda[\tau_n]$, $\mathcal F_1 \defeq \mathcal F$, and $\mathcal F_{i+1} \defeq \mathcal F[\mathcal T_{i}]$ for $1 \le i < n$.
Under the canonical embedding $\mathbb R(\mathcal F_i, \mathcal Z) \subset \mathbb R(\mathcal F, \mathcal Z)$, we identify the stable laminations $\mathcal L^+_{\mathcal Z}[\psi]$ and~$\mathcal L^+_{\mathcal Z}[\psi_i]$.
Let~$\mathcal T^\circ$ be an equivariant blow-up of free splittings~$(\mathcal T_i)_{i=1}^n$ and $\mathcal T_i^\circ \subset \mathcal T^\circ$ the characteristic convex subsets for~$\mathcal F_i$.
Suppose $x \in \mathcal F_1$ is a $\mathcal T^\circ$-loxodromic element.
The equivalence between Conditions~\ref{thm-limitlox-cond-lox}--\ref{thm-limitlox-cond-lim} is given by Proposition~\ref{prop-limitloxodromics} if~$x$ is conjugate to an element of $\mathcal F_{n}$.
Assume $n \ge 2$ and, up to conjugacy,~$x \in \mathcal F_i$ is $\mathcal T_i$-loxodromic for some~$i < n$.

Recall that $\tau^\circ \colon (\mathcal T^\circ, d^\circ) \to (\mathcal T^\circ, d^\circ)$ is a $\psi$-equivariant $\lambda$-Lipschitz topological representative induced by the irreducible train tracks~$(\tau_i)_{i=1}^n$ and $\pi^\circ \colon (\mathcal T^\circ, d^\circ) \to (\mathcal Y, \delta)$ is the constructed equivariant metric map.
In particular, $\underset{m \to \infty}\limsup \, \frac{1}{m} \log \| \psi^m(x)\|_{d^\circ} \le \log \lambda$.

Suppose~$[\tau_i^{k'}]$ (for some $k' \ge 1$) fixes all $\mathcal F_i$-orbits of vertices and edges in~$\mathcal T_i$.
Let $l^\circ \subset \mathcal T_i^\circ$ be the axis for $x \in \mathcal F_i$.
The axis~$l^\circ$ projects to the axis~$l$ of~$x$ in~$\mathcal T_i$;
write~$l$ as a biinfinite concatenation of edges $\cdots e_{-1} \cdot e_0 \cdot e_1 \cdots$ and identify $e_j \subset \mathcal T_i$ with its lift to~$\mathcal T_i^\circ$.
For $m \ge 0$ and any integer~$j$, let~$w_{j,m}$ be the closed interval in~$\mathcal T_i^\circ$ between (lifts of)~$\tau_i^{m}(e_j)$ and~$\tau_i^{m}(e_{j+1})$;
in fact, $w_{j,m}$ is in a component of $\mathcal F_i \cdot \mathcal T_{i+1}^\circ \subset \mathcal T_i^\circ$.
Since~$[\tau_i^{k'}]$ fixes the $\mathcal F_i$-orbits $[e], [e']$ and the vertex of~$\mathcal T_i$ between them, the sequence $(w_{j,mk'+r})_{m \ge 0}$, up to translation, is an iterated turn over~$\mathcal T_{i+1}^\circ$ rel.~$\psi_{i+1}^{k'}$ for $0 \le r < k'$;
by Theorem~\ref{thm-limitturns}, the iterated turn limits to an interval~$w_{j,r}^*$ in a translate of a component of $\widehat{\mathcal Y}_{i+1} \subset \mathcal Y_i$.

The intervals~$w_{j,m}, w_{j+1,m}$ are always in distinct components of $\mathcal F_i \cdot \mathcal T_{i+1}^\circ$;
therefore, the limit intervals $w_{j,r}^*, w_{j+1,r}^*$ have degenerate intersection.
By the equivariance of the limits, the union $l_* \defeq \bigcup_j w_{j,0}^*$ is an $x$-invariant arc.
If some limit interval~$w_{j,0}^*$ is not degenerate, then~$x$ is $\mathcal Y_i$-loxodromic and~$l_*$ is its $\mathcal Y_i$-axis;
otherwise,~$l_*$ is degenerate and~$x$ is $\mathcal Y_i$-elliptic.

\smallskip
\noindent \underline{Case 1:}
$x$ is $\mathcal Y_i$-loxodromic, i.e.~some limit interval~$w_{j,0}^*$ is not degenerate.
By Theorem~\ref{thm-limitturns}, the iterated turn~$(w_{j,mk'})_{m \ge 0}$ over~$\mathcal T_{i+1}^\circ$ rel.~$\psi_{i+1}^{k'}$ weakly limits to a component of~$\mathcal L_{\mathcal Z}^+[\psi_{i+1}]$.
So $[l^\circ] \in \mathbb R(\mathcal F, \mathcal Z)$ weakly $\psi_*^{k'}$-limits to a component of~$\mathcal L_{\mathcal Z}^+[\psi]$.
Finally,~$[l^\circ]$ weakly $\psi_*$-limits to~$\mathcal L_{\mathcal Z}^+[\psi]$ since~$\psi_*$ acts transitively on the components of~$\mathcal L_{\mathcal Z}^+[\psi]$.
As~$\pi^\circ$ is an equivariant metric map, $\|\cdot\|_{\delta} \le \| \cdot \|_{d^\circ}$ and $\log \lambda \le \underset{m \to \infty} \liminf \, \frac{1}{m} \log \| \psi^m(x)\|_{d^\circ}$.

\smallskip
\noindent \underline{Case 2:}
$x$ is $\mathcal Y_i$-elliptic, i.e.~each limit interval~$w_{j,0}^*$ is degenerate.
By Theorem~\ref{thm-limitturns}, the interval~$w_{j,mk}$ has $d^\circ$-length is bounded above by some degree~$(n-i)$ polynomial (in~$m$).
Thus $\|\psi^{mk}(x)\|_{d^\circ}$ is bounded above by a degree~$(n-i)$ polynomial.
By $\psi$-equivariance of the homothety~$h_i$, the elements $\psi(x), \ldots, \psi^{k-1}(x)$ are $\mathcal Y_i$-elliptic as well.
The same argument implies $\|\psi^{m}(x)\|_{d^\circ}$ is bounded above by a degree~$(n-i)$ polynomial.
\end{proof}

We conclude the chapter by stating the extension of Lemma~\ref{lem-convergence} to all limit forests:

\begin{restate}{Lemma}{lem-convergence2}
Let~$\psi\colon \mathcal F \to \mathcal F$ be an automorphism, $\mathcal Z'$ a $[\psi]$-invariant proper free factor system, $\left(\tau_i \colon \mathcal T_i \to \mathcal T_i\right)_{i=1}^n$ a descending sequence of irreducible train tracks for~$[\psi]$ rel.~$\mathcal Z'$ with $\lambda \defeq \lambda[\tau_n] > 1$, $(\mathcal Y, \delta)$ the limit forest for~$[\tau_i]_{i=1}^n$, $(\mathcal Y', \delta')$ a minimal $\mathcal F$-forest with trivial arc stabilizers, and $\mathcal Z \defeq \mathcal F[\mathcal T_n]$.

If~$\mathcal Z$ is $\mathcal Y'$-elliptic and the $k$-component lamination~$\mathcal L_{\mathcal Z}^+[\psi]$ is in $\mathbb R(\mathcal Y', \delta') \subset \mathbb R(\mathcal F, \mathcal Z)$, then the limit of $(\mathcal Y' \psi^{mk}, \lambda^{-mk} \delta')_{m \ge 0}$ is $(\mathcal Y, \oplus_{j=1}^k c_j \, \delta_j)$, where  $\delta = \oplus_{j=1}^k \, \delta_j$ and $c_j > 0$.
\end{restate}

Again, we postpone the proof to Section~\ref{SubsecConverge2}.
If~$(\tau_i')_{i=1}^{n'}$ is another descending sequence for~$[\psi]$ with $\mathcal F[\mathcal T_{n'}'] = \mathcal Z$, then its limit forest $(\mathcal Y', \delta')$ is equivariantly homothetic to $(\mathcal Y, \delta)$; 
therefore, $(\mathcal Y, \delta)$ is \emph{the} \underline{limit forest} for~$[\psi]$ rel.~$\mathcal Z$ (up to rescaling of~$\delta$). 
A nondegenerate minimal very small $\mathcal F$-forest $(\mathcal Y', \delta')$ is an \underline{expanding forest} for~$[\psi]$ rel.~$\mathcal Z$ if:

\begin{enumerate}
\item the $\mathcal F$-forest $(\mathcal Y'\psi, \delta')$ is equivariantly isometric to $(\mathcal Y', s \delta')$ for some $s > 1$; and
\item the free factor system~$\mathcal Z$ is $\mathcal Y'$-elliptic.
\end{enumerate}
\begin{cor}\label{cor-unique2}
Let~$\psi\colon \mathcal F \to \mathcal F$ be an automorphism and $\left(\tau_i \colon \mathcal T_i \to \mathcal T_i\right)_{i=1}^n$ a descending sequence of irreducible train tracks for~$[\psi]$ with $\lambda[\tau_n] > 1$.
Any expanding forests for~$[\psi]$ rel.~$\mathcal F[\mathcal T_n]$ is the limit forest for~$[\psi]$ rel.~$\mathcal F[\mathcal T_n]$.
\end{cor}

\noindent We will end the paper with a complete generalization of this corollary (Theorem~\ref{thm-unique}).

\begin{proof}[Sketch of proof]
Let $(\mathcal Y', \delta')$ be an expanding forest for~$[\psi]$ rel.~$\mathcal Z \defeq \mathcal F[\mathcal T_n]$ and $x \in \mathcal F$ a $\mathcal Y'$-loxodromic element.
The proof is essentially the proof of Corollary~\ref{cor-unique} with two main changes. First, choose $m \gg 1$ so that $\| \psi^m(x) \|_{\delta'} > \alpha(m)(2C[f]+B')$ for some polynomial~$\alpha$ and constant $B' \ge 1$ determined by~$x$;
therefore, a fundamental domain of~$\psi^m(x)$ acting on its axis has a leaf segment $[q, r]$ with $\delta'(f(q), f(r)) > 2C[f]$ by the pigeonhole principle. 
For the second change, we need $(\mathcal Y', \delta')$ to have trivial arc stabilizers in order to conclude the proof by invoking Lemma~\ref{lem-convergence2} instead of Lemma~\ref{lem-convergence}.

The minimal very small $\mathcal F$-forest $(\mathcal Y', \delta')$ has finitely many orbits of branch points~\cite{GL95}; 
it decomposes as some graph of actions whose skeleton is not degenerate in the forest if and only if the forest does not have dense orbits~\cite{Lev94}.
Any $\psi$-equivariant homothety must be an isometry if the skeleton were not degenerate.
Since $(\mathcal Y', \delta')$ admits a $\psi$-equivariant expanding $s$-homothety, the skeleton must be degenerate and the forest has dense orbit.
Very small $\mathcal F$-forests with dense orbits have trivial arc stabilizers~\cite[Lemma~4.2]{LL03}.
\end{proof}

For a nondegenerate minimal $\mathcal F$-forest $(\mathcal Y', \delta')$, the \underline{projective stabilizer} $\operatorname{Stab}[\mathcal Y', \delta']$ is the subgroup of automorphisms $\varphi \colon \mathcal F \to \mathcal F$ for which $\|\varphi(\cdot)\|_{\delta'} = s_\varphi \| \cdot \|_{\delta'}$ for some $s_\varphi > 0$.
The function $\operatorname{SF} \colon \operatorname{Stab}[\mathcal Y', \delta'] \to \mathbb R_{>0}$ that maps $\varphi \mapsto s_\varphi$ is a homomorphism called the \emph{stretch factor homomorphism} --- $\mathbb R_{>0}$ is considered multiplicatively.

\begin{cor}\label{cor-cyclicimage}
Let~$\operatorname{SF} \colon \operatorname{Stab}[\mathcal Y', \delta'] \to \mathbb R_{>0}$ be the stretch factor homomorphism for some nondegenerate minimal very small $\mathcal F$-forest $(\mathcal Y', \delta')$.
The image of~$\operatorname{SF}$ is cyclic.
\end{cor}
\begin{proof}
Suppose $\operatorname{SF}(\psi) > 1$ for some $\psi \in  \operatorname{Stab}[\mathcal Y', \delta']$.
Then~$\psi$ is exponentially growing since any $\mathcal Y'$-loxodromic element $[\psi]$-grows exponentialy with rate at least $\operatorname{SF}(\psi)$.
Set $\mathcal F_1 \defeq \mathcal F$, $\psi_1 \defeq \psi$, and let $(\mathcal Y_1, \delta_1)$ be the limit forest for~$[\psi_1]$ rel.~some $[\psi_1]$-invariant proper free factor system~$\mathcal F_2$.
If~$\mathcal F_2$ is not $\mathcal Y'$-elliptic, then the restrictions~$\psi_2$ of~$\psi_1$ to~$\mathcal F_2$ are in the projective stabilizer of the nondegenerate characteristic subforest of $(\mathcal Y', \delta')$ for~$\mathcal F_2$ and have the same stretch factor~$\operatorname{SF}(\psi)$.

By repeatedly considering limit forests and taking restrictions, we may assume some free factor system~$\mathcal F_n$ is not $\mathcal Y'$-elliptic while a nested proper free factor system~$\mathcal F_{n+1}$ is for some $n \ge 1$.
Then the characteristic subforest of $(\mathcal Y', \delta')$ for the free factor system~$\mathcal F_n$ is an expanding forest for~$[\psi_n]$ rel.~$\mathcal F_{n+1}$.
By Corollary~\ref{cor-unique2}, this subforest is equivariantly homothetic to the limit forest $(\mathcal Y_n, \delta_n)$ for~$[\psi_n]$ rel.~$\mathcal F_{n+1}$.
In particular, $\operatorname{SF}(\psi)$ is the exponential growth rate for~$[\psi_n]$ rel.~$\mathcal F_{n+1}$ and is bounded away from~$1$ by a uniform constant that depends only on~$\mathcal F$. 
Thus the image of~$\operatorname{SF}$ is discrete, and discrete subgroups of~$\mathbb R_{> 0}$ are cyclic.
\end{proof}

\newpage
\section{Main constructions}\label{SecTopDomLimit}

The limit forest produced by our proof of Proposition~\ref{prop-limittree} is universal for an outer automorphism and some choice of an invariant proper free factor system (Corollary~\ref{cor-unique2}). 
%
Our goal is to remove the latter dependence on an invariant proper free factor system.

\subsection{Assembling limit hierarchies}\label{SubsecAssembHier}

This section first summarizes the main result of the paper's prequel~\cite{Mut24}.
The general strategy follows closely the construction of limit forests sketched in Section~\ref{SubsecLimittrees2}.

Fix an exponentially growing automorphism $\psi\colon \mathcal F \to \mathcal F$ and set $\mathcal G_1 \defeq \mathcal F$, $\psi_1 \defeq \psi$.
By our proof of Proposition~\ref{prop-limittree}, there is a nondegenerate limit forest $(\mathcal Y_1, \delta_1)$ for~$[\psi_1]$ rel.~$\mathcal Z_1$ (some proper free factor system of~$\mathcal G_1$) and a unique $\psi_1$-equivariant expanding $\lambda_1$-homothety $h_1 \colon (\mathcal Y_1, \delta_1) \to (\mathcal Y_1, \delta_1)$.
Thus $\mathcal Y_1$-loxodromic elements in~$\mathcal F$ $[\psi]$-grow exponentially rel.~$\mathcal Z_1$ with rate~$\lambda_1$.
By Gaboriau--Levitt index theory and $\psi_1$-equivariance of~$\tau_1$, the nontrivial point stabilizers of~$\mathcal Y_1$ determine a $[\psi_1]$-invariant malnormal subgroup system $\mathcal G_2 \defeq \mathcal G[\mathcal Y_1]$ with strictly lower complexity than~$\mathcal G_1$.
The restriction of~$\psi_1$ to~$\mathcal G_2$ determines a unique outer class of automorphisms $\psi_2 \colon \mathcal G_2 \to \mathcal G_2$.

We can repeatedly apply Proposition~\ref{prop-limittree} to~$\psi_{i+1}~(i \ge 1)$ as long as~$\psi_{i+1}$ is exponentially growing.
This inductive invokation of Proposition~\ref{prop-limittree} eventually stops since the complexity of~$\mathcal G_i$ is a strictly decreasing (in~$i$) positive integer.
In the end, we have a maximal sequence $(\mathcal Y_i, \delta_i)_{i=1}^n$ of nondegenerate limit forests for~$[\psi_i]$ rel.~$\mathcal Z_i$ each with a unique $\psi_i$-equivariant expanding $\lambda_i$-homothety~$h_i$ on~$(\mathcal Y_i, d_i)$ ---
such a maximal sequence of limit forests is a \underline{descending sequence} of limit forests for~$[\psi]$. 
By construction, an element $x \in \mathcal F$ has a conjugate in~$\mathcal G_{n+1}$ if and only if~$x$ $[\psi]$-grows polynomially!

\smallskip
In Section~\ref{SubsecLimittrees2}, the blow-ups of free splittings $(\mathcal T_i)_{i=1}^n$ were arbitary and done inductively upwards (i.e.~started with $i = n$).
We then used a limiting argument to produce the final limit forest $(\mathcal Y, \delta)$.
For this section, the blow-ups of limit forests $(\mathcal Y_i, \delta_i)_{i=1}^n$ will not be arbitrary but will make use of the expanding homotheties~$(h_i)_{i=1}^n$;
moreover, it will be done inductively downwards (i.e.~starts with $i = 1$) to produce an $\mathcal F$-pseudoforest $(\mathcal T, (\delta_i)_{i=1}^n)$.

Set $(\mathcal X^{(1)}, \delta_1) \defeq (\mathcal Y_1, \delta_1)$ and $g^{(1)} \defeq h_1$. 
For $1 < i \le n$, we inductively construct the \emph{equivariant pseudoforest blow-up} $(\mathcal X^{(i)}, (\delta_j)_{j=1}^i)$ of the $\mathcal F$-pseudoforest $(\mathcal X^{(i-1)}, (\delta_j)_{j=1}^{i-1})$ rel.~the $\mathcal G_i$-forest $(\mathcal Y_i, \delta_i)$ and expanding homotheties~$g^{(i-1)}$ and~$h_i$.
Here is a sketch:

Let $(\overline{\mathcal Y}_i, \delta_i)$ be the metric completion and~$\bar h_i$ the extension to the metric completion.
For $1 \le j < i$, assume that $(\mathcal Y_j, \delta_j)$ is equivariantly isometric to the associated $\mathcal G_j$-forest for the $\mathcal G_j$-invariant convex pseudometric~$\delta_j$ restricted to~$\mathcal X^{(i-1)}(\mathcal G_j)$, the characteristic convex subsets of~$\mathcal X^{(i-1)}$ for~$\mathcal G_j$.
Since the hierarchy $(\delta_j)_{j=1}^{i-1}$ has full support, $(\mathcal X^{(i-1)}(\mathcal G_{i-1}), \delta_{i-1})$ is equivariantly isometric to $(\mathcal Y_{i-1},  \delta_{i-1})$ and the nontrivial point stabilizers of~$\mathcal X^{(i-1)}$ are conjugates in~$\mathcal F$ of $\mathcal G_i$-components.
The points of~$\mathcal X^{(i-1)}$ with nontrivial stabilizers are replaced by corresponding copies of $\overline{\mathcal Y}_i$-components;
this produces a unique set system~$\widehat{\mathcal X}^{(i)}$ with an $\mathcal F$-action that is the \emph{equivariant set blow-up} of $\mathcal X^{(i-1)}$ rel.~$\overline{\mathcal Y}_i$:
it comes with an equivariant injection $\iota_i \colon \overline{\mathcal Y}_i \to \widehat{\mathcal X}^{(i)}$ and an equivariant surjection $\kappa_i \colon \widehat{\mathcal X}^{(i)} \to \mathcal X^{(i-1)}$ that is a bijection on the complement $\widehat{\mathcal X}^{(i)} \setminus \mathcal F \cdot \iota_i(\overline{\mathcal Y}_i)$.
Consequently, there is a unique $\psi$-equivariant induced permutation~$g^{(i)} \colon \widehat{\mathcal X}^{(i)} \to \widehat{\mathcal X}^{(i)}$ induced by~$g^{(i-1)}$ and~$\bar h_i$ --- $\kappa_i$ semiconjugates~$\hat g^{(i)}$ to~$g^{(i-1)}$ while~$\iota_i$ conjugates~$\bar h_i$ to the restriction $\left.g^{(i)}\right|_{\iota_i(\overline{\mathcal Y}_i)}$.

There are plenty of equivariant interval functions~$[\cdot, \cdot]^{(i)}$ on~$\widehat{\mathcal X}^{(i)}$ {compatible} with~$\mathcal X^{(i-1)}$ and~$\mathcal Y_i$ --- \emph{compatibility} means the injection~$\iota_i$ and surjection~$\kappa_i$ map intervals to intervals.
Some compatible $\mathcal F$-pretrees $(\widehat{\mathcal X}^{(i)}, [\cdot, \cdot]^{(i)})$ are real~\cite[Proposition~IV.3]{Mut24} and they naturally inherit an $\mathcal F$-invariant hierarchy~$(\hat \delta_j)_{j=1}^i$ with full support: $(\hat \delta_j)_{j=1}^{i-1}$ is the pullback $\kappa_i^*(\delta_j)_{j=1}^{i-1}$ and~$\hat \delta_i$ is the pushforward $\iota_{i *} \delta_i$ extended equivariantly to the orbit $\mathcal F \cdot \iota_i(\overline{\mathcal Y}_i)$;
moreover, for $1 \le j \le i$, $(\mathcal Y_j, \delta_j)$ is equivariantly isometric to the associated $\mathcal G_j$-forest for the $\mathcal G_j$-invariant convex pseudometric~$\hat \delta_j$ restricted to~$\widehat{\mathcal X}^{(i)}(\mathcal G_j)$. 

\begin{claim*}[{\cite[Theorem~IV.4]{Mut24}}]
Since~$\bar h_i$ is expanding, the permutation~$g^{(i)}$ is a pretree-automorphism for a unique real compatible $\mathcal F$-pretree $(\widehat{\mathcal X}^{(i)}, [\cdot, \cdot]_g^{(i)})$. \qed
\end{claim*}

\begin{rmk}
This is the main technical result of~\cite{Mut24}.
Its proof uses Gaboriau--Levitt's index inequality and the contraction mapping theorem.
\end{rmk}

We now fix the interval function $[\cdot, \cdot]_g^{(i)}$ but omit it for brevity.
By construction, the $\mathcal F$-pseudoforest $(\widehat{\mathcal X}^{(i)}, (\hat \delta_j)_{j=1}^i)$ has trivial arc stabilizers and~$g^{(i)}$ is an expanding homothety with respect to~$(\hat \delta_j)_{j=1}^i$.
Finally, let $\mathcal X^{(i)} \subset \widehat{\mathcal X}^{(i)}$ be the characteristic convex subsets for~$\mathcal F$ and~$(\delta_j)_{j=1}^i$ the restriction of the hierarchy~$(\hat \delta_j)_{j=1}^i$ to~$\mathcal X^{(i)}$, then replace the maps $\iota_i, \kappa_i$, and~$g^{(i)}$ with their restrictions to~$\mathcal X^{(i)}$;
so $({\mathcal X}^{(i)}, (\delta_j)_{j=1}^i)$ is a minimal $\mathcal F$-pseudoforest.

At the $n^{th}$ iteration, we have a minimal $\mathcal F$-pseudoforest $(\mathcal T, (\delta_i)_{i=1}^n) \defeq (\mathcal X^{(n)}, (\delta_i)_{i=1}^n)$ with trivial arc stabilizers, unique for the descending sequence $(\mathcal Y_i, \delta_i)_{i=1}^n$;
the $\psi$-equivariant pretree-automorphism $h \defeq g^{(n)}$ on $(\mathcal T, (\delta_i)_{i=1}^n)$ is a $(\lambda_i)_{i=1}^n$-homothety, where $\lambda_i>1$ is the scaling factor for the homothety~$h_i$;
lastly, an element $x \in \mathcal F$ is $\mathcal T$-elliptic if and only if~$x$ has a conjugate in $\mathcal G_{n+1}$.
The real $\mathcal F$-pretrees~$\mathcal T$ are the \underline{limit pretrees} for $(\mathcal Y_i)_{i=1}^n$ and the $\mathcal F$-pseudoforest $(\mathcal T, (\delta_i)_{i=1}^n)$ is the \underline{limit pseudoforest} for $(\mathcal Y_i, \delta_i)_{i=1}^n$.
To summarize,

\begin{thm}[cf.~{\cite[Theorem~III.3]{Mut24}}]\label{thm-limitpseudotree} Let $\psi \colon \mathcal F \to \mathcal F$ be an automorphism.
Then there is:
\begin{enumerate}
\item a minimal $\mathcal F$-pseudoforest $(\mathcal T, (\delta_i)_{i=1}^n)$ with trivial arc stabilizers;
\item a $\psi$-equivariant expanding homothety $h \colon (\mathcal T, (\delta_i)_{i=1}^n) \to (\mathcal T, (\delta_i)_{i=1}^n)$; and
\item an element~$x \in \mathcal F$ is $\mathcal T$-loxodromic if and only if~$x$ $[\psi]$-grows exponentially.
\end{enumerate}
The real pretrees~$\mathcal T$ are degenerate if and only if~$[\psi]$ is exponentially growing. \qed
\end{thm}

Without metrics, there is not much one can do to compare limit pretrees.
On the other hand, we do not expect limit pseudoforests to be well-defined (even up to homothety) for a given outer automorphism --- this would be equivalent to the existence of a canonical descending sequence of limit forests.
The new idea is to pick a limit pseudoforest~$(\mathcal T, (\delta_i)_{i=1}^n)$ and {normalize} its hierarchy~$(\delta_i)_{i=1}^n$ using the {attracting laminations} for~$[\psi]$.
For the normalized hierarchy, the associated top level forest will be universal;
in particular, it is independent of any choices made in its construction.

\subsection{Attracting laminations}\label{SubsecAttrLams}

Fix an exponentially growing automorphism $\psi\colon \mathcal F \to \mathcal F$ 
with a descending sequence $(\mathcal Y_i, \delta_i)_{i=1}^n$ of limit forests.
Let $\mathcal G_1 = \mathcal F$, $\mathcal G_{i+1} = \mathcal G[\mathcal Y_i]$, and $[\psi_i]$ be the restriction of~$[\psi]$ to~$\mathcal G_i$ for $i \ge 1$.
Each limit forest $(\mathcal Y_i, \delta_i)$ has matching 
stable laminations~$\mathcal L_{\mathcal Z_i}^+[\psi_i]$ for~$[\psi_i]$ rel.~$\mathcal Z_i$, where~$\mathcal Z_i$ is a $[\psi_i]$-invariant proper free factor system of~$\mathcal G_i$.
By Claim~\ref{claim-liftembed}, $\mathbb R(\mathcal G_i, \mathcal Z_i)$ is canonically identified with a subspace of $\mathbb R(\mathcal G_i)$ via a lifting map.
As~$\mathcal G_{i+1}$ is a malnormal subgroup system of~$\mathcal G_i$, the space of lines~$\mathbb R(\mathcal G_{i+1})$ is canonically identified with a closed subspace of~$\mathbb R(\mathcal G_i)$ (exercise).
By transitivity, $\mathbb R(\mathcal G_n) \subset \mathbb R(\mathcal G_{n-1}) \subset \cdots \subset \mathbb R(\mathcal G_0) = \mathbb R(\mathcal F)$.

Consider this chain of canonical embeddings: $\mathbb R(\mathcal G_i, \mathcal Z_i) \subset \mathbb R(\mathcal G_i) \subset \mathbb R(\mathcal F)$. 
Quasiperiodicity is not preserved by the first embedding but a weaker form of it is.
A line~$[l]$ is \underline{birecurrent} in an $\mathcal F$-forest if any closed interval $I \subset l$ has infinitely many translates contained in both ends of~$l$;
quasiperiodic lines are birecurrent.

An \underline{attracting lamination} for~$[\psi]$ in $\mathbb R(\mathcal F)$ is the closure of a birecurrent line in~$\mathbb R(\mathcal F)$ with a $\psi_*^k$-attracting neighbourhood for some $k\ge 1$.
The set of all attracting laminations for~$[\psi]$ is {canonical} as it is defined using canonical constructs:~$\mathbb R(\mathcal F)$ and the homeomorphism $\psi_* \colon \mathbb R(\mathcal F) \to \mathbb R(\mathcal F)$.
Note that~$\psi_*$ permutes the attracting laminations for~$[\psi]$.

\begin{rmk}
This definition is from~\cite[Definition~3.1.5]{BFH00}.
Shortly, we will define {topmost} attracting laminations as done in~\cite[Section~6]{BFH00}.
\end{rmk}


\begin{lem}[cf.~{\cite[Lemma~3.1.4]{BFH00}}]\label{lem-liftbirecur} Let $f \colon (\mathcal T,d) \to (\mathcal Y, \delta)$ be an equivariant PL-map. 
A line is birecurrent in $\mathbb R(\mathcal Y,\delta)$ if and only if it is birecurrent in~$\mathbb R(\mathcal T)$.
{\rm(exercise)}  \qed
\end{lem}

So leaves of~$\mathcal L_{\mathcal Z_i}^+[\psi_i]$ are birecurrent in~$\mathbb R(\mathcal G_i)$ and hence~$\mathbb R(\mathcal F)$;
moreover, a $\psi_{i*}^k$-attracting neighbourhood of a line in $\mathbb R(\mathcal G_i, \mathcal Z_i)$ will lift to a $\psi_*^k$-attracting neighbourhood of the same line in~$\mathbb R(\mathcal F)$. (exercise)
Thus the closure in~$\mathbb R(\mathcal F)$ of a stable lamination for~$[\psi_i]$ rel.~$\mathcal Z_i$, i.e.~a component of~$\mathcal L_{\mathcal Z_i}^+[\psi_i]$, is an attracting lamination for~$[\psi]$. 

\begin{lem}[cf.~{\cite[Lemma~3.1.10]{BFH00}}]\label{lem-coordinatefreeattracting}
Let $\psi\colon \mathcal F \to \mathcal F$ be an exponentially growing automorphism with a descending sequence $(\mathcal Y_i, \delta_i)_{i=1}^n$ of limit forests.
The components of stable laminations~$\mathcal L_{\mathcal Z_i}^+[\psi_i]~(1 \le i \le n)$ determine all the attracting laminations for~$[\psi]$.
\end{lem}

\begin{proof}[Sketch of proof]
Suppose that $[l] \in \mathbb R(\mathcal F)$ is a birecurrent line with a $\psi_*^k$-attracting neighbourhood for some $k \ge 1$.
If~$\mathcal G_{n+1} \neq \emptyset$, then either it consists of only cyclic components or the restriction of~$\psi_{n}$ to~$\mathcal G_{n+1}$ is polynomially growing.
Either way, $\mathcal G_{n+1}$ cannot support an attracting lamination of~$\psi_n$.
Let $i \le n$ be the maximal index for which $\mathbb R(\mathcal G_{i}) \subset \mathbb R(\mathcal F)$ contains~$[l]$.
Birecurrence in~$\mathbb R(\mathcal F)$ and Lemma~\ref{lem-liftbirecur} imply~$[l]$ is birecurrent in $\mathbb R(\mathcal G_{i}, \mathcal Z_i)$ with a $\psi_{i*}^k$-attracting neighbourhood for some $k \ge 1$.
Following the proof of Claim~\ref{claim-anystable}, assume some descending chain $(\mathcal F_{i,j})_{j=2}^{n_i}$ of proper free factor systems of $\mathcal F_{i,1} \defeq \mathcal G_i$ was used to construct $(\mathcal Y_i, \delta_i)$;
then any birecurrent line in $\mathbb R(\mathcal G_i, \mathcal Z_i)$ with a $\psi_{i*}^k$-attracting neighbourhood is in $\mathbb R(\mathcal F_{i, n_i}, \mathcal Z_i)$.
The proof of Lemma~\ref{lem-coordinatefreestable} (with ``birecurrence'' in place of ``quasiperiodicity'') implies $[l] \in \mathcal L_{\mathcal Z_i}^+[\psi_i]$.
\end{proof}


The finite set of all attracting laminations for~$[\psi]$ is canonical (by definition) and partially ordered by inclusion;
an attracting lamination for~$[\psi]$ is \underline{topmost} if it is maximal in this partial order.
By Lemma~\ref{lem-coordinatefreestable},~$\psi_{i*}$ transitively permutes the components of~$\mathcal L_{\mathcal Z_i}^+[\psi_i]$;
so the closure in $\mathbb R(\mathcal F)$ of~$\mathcal L_{\mathcal Z_i}^+[\psi_i] \subset \mathbb R(\mathcal G_i, \mathcal Z_i)$ is a $\psi_*$-orbit~$\mathcal L_i^+[\psi]$ of attracting laminations for~$[\psi]$.
The goal is to {normalize} any limit pseudoforest $(\mathcal T, (d_i)_{i=1}^n)$ so that the levels are related to the partial order of the attracting laminations.


The next proposition is a repackaging of Theorem~\ref{thm-limitloxodromics} in the language of this chapter:

\begin{prop}\label{prop-limitloxodromics2}
Let $\psi\colon \mathcal F \to \mathcal F$ be an exponentially growing automorphism with a limit pseudoforest $(\mathcal T, (\delta_i)_{i=1}^n)$.

For a nontrivial element $x \in \mathcal F$, the following statements are equivalent:
\begin{enumerate}
\item\label{prop-limitlox2-cond-lox} the element~$x$ is $\mathcal T$-loxodromic;
\item\label{prop-limitlox2-cond-exp} the element~$x$ $[\psi]$-grows exponentially; and
\item\label{prop-limitlox2-cond-lim} the axis for~$x$ in $\mathbb R(\mathcal F)$ weakly $\psi_*$-limits to an attracting lamination.
\end{enumerate}
\end{prop}

\begin{proof} The equivalence between Conditions \ref{prop-limitlox2-cond-lox}--\ref{prop-limitlox2-cond-exp} is part of Theorem~\ref{thm-limitpseudotree}. Suppose $x \in \mathcal F$ is $\mathcal T$-loxodromic and the limit pseudoforest $(\mathcal T, (\delta_i)_{i=1}^n)$ is constructed from the descending sequence of limit forests $(\mathcal Y_i, \delta_i)$ for $1 \le i \le n$.
By construction, the element~$x$ is conjugate to a $\mathcal Y_i$-loxodromic element $y \in \mathcal G_i$ for some $i \le n$;
in particular,~$x$ and~$y$ have the same axis in $\mathbb R(\mathcal G_i) \subset \mathbb R(\mathcal F)$.
The axis for~$y$ in $\mathbb R(\mathcal G_i, \mathcal Z_i) \subset \mathbb R(\mathcal G_i)$ weakly $\psi_{i *}$-limits to the stable laminations~$\mathcal L_{\mathcal Z_i}^+[\psi_i] \subset \mathbb R(\mathcal G_i, \mathcal Z_i)$ by Theorem~\ref{thm-limitloxodromics};
therefore, the shared axis for~$y$ and~$x$ in~$\mathbb R(\mathcal F)$ weakly $\psi_*$-limits to the attracting laminations for~$[\psi]$ determined by~$\mathcal L_{\mathcal Z_i}^+[\psi_i]$, i.e.~the closure of $\mathcal L_{\mathcal Z_i}^+[\psi_i]$ in $\mathbb R(\mathcal F)$.

Conversely, suppose $x \in \mathcal F$ is $\mathcal T$-elliptic.
Then~$x$ is must be conjugate to a $\mathcal Y_n$-elliptic element $y \in \mathcal G_n$.
If~$y$ is conjugate to an element of~$\mathcal Z_i$, then the shared axis for~$y$ and~$x$ in the closed subspace $\mathbb R(\mathcal Z_i) \subset \mathbb R(\mathcal F)$ cannot weakly $\psi_*$-limit to the attracting lamination for~$[\psi]$ determined by a component of~$\mathcal L_{\mathcal Z_i}^+[\psi_i]$ --- such an attracting lamination contains lines not in $\mathbb R(\mathcal Z_i)$.
If~$y$ is not conjugate to an element of~$\mathcal Z_i$, then the axis for~$y$ in $\mathbb R(\mathcal G_i, \mathcal Z_i)$ does not weakly $\psi_{i *}$-limit to~$\mathcal L_{\mathcal Z_i}^+[\psi_i]$ by Theorem~\ref{thm-limitloxodromics};
therefore, the shared axis for~$y$ and~$x$ in~$\mathbb R(\mathcal F)$ cannot weakly $\psi_*$-limit to the attracting lamination for~$[\psi]$ determined by a component of~$\mathcal L_{\mathcal Z_i}^+[\psi_i]$.
By Lemma~\ref{lem-coordinatefreeattracting}, we have exhausted all possibilities when $1 \le i \le n$, and the axis for~$x$ in $\mathbb R(\mathcal F)$ cannot weakly $\psi_*$-limit to an attracting lamination for~$[\psi]$.
\end{proof}

\subsection{Pseudolaminations}\label{SubsecPseudoleaf}

Fix an exponentially growing automorphism $\psi \colon \mathcal F \to \mathcal F$ with a descending sequence $(\mathcal Y_i, \delta_i)_{i=1}^n$ of limit forests, and let $(\mathcal T, (\delta_i)_{i=1}^n)$ be the limit pseudoforest for $(\mathcal Y_i, \delta_i)_{i=1}^n$.
Recall that $\mathcal G_1 = \mathcal F$, $\mathcal G_{i+1} = \mathcal G[\mathcal Y_i]$, and $[\psi_i]$ is the restriction of~$[\psi]$ to~$\mathcal G_i$ for $i \ge 1$.
For $1 \le i \le n$, the stable laminations~$\mathcal L_{\mathcal Z_i}^+[\psi_i]$ are contained in $\mathbb R(\mathcal Y_i, \delta_i) \subset \mathbb R(\mathcal G_i, \mathcal Z_i)$, where~$\mathcal Z_i$ is some $[\psi_i]$-invariant proper free factor system of~$\mathcal G_i$.

Let $\mathcal T_i \subset \mathcal T$ be the characteristic convex subsets for~$\mathcal G_i$.
By construction of $(\mathcal T, (\delta_i)_{i=1}^n)$,~$\delta_i$ restricts to a $\mathcal G_i$-invariant convex pseudometric on~$\mathcal T_i$ whose associated $\mathcal G_i$-forest can be equivariantly identified with $(\mathcal Y_i, \delta_i)$.
Fix such an identification, and let $\kappa_i \colon \mathcal T_i \to \mathcal Y_i$ denote the natural equivariant collapse map.
The stable laminations~$\mathcal L_{\mathcal Z_i}^+[\psi_i]$ are in $\mathbb R(\mathcal Y_i, \delta_i)$;
their leaves have unique lifts (via $\kappa_i$) to $\mathcal T_i \subset \mathcal T$;
we call these \underline{pseudoleaves} of~$\mathcal L_{\mathcal T}^+[\psi_i]$.
A \underline{pseudoleaf segment} of $\mathcal L_{\mathcal T}^+[\psi_i]$ is a closed interval in a (representative of a) pseudoleaf with nondegenerate $\kappa_i$-image in~$\mathcal Y_i$.

Remarkably, the pseudoleaf segments detect weak $\psi_*$-limits of elements in attracting laminations.
Let $\mathcal L_i^+[\psi]$ be the attracting laminations for~$[\psi]$ determined by~$\mathcal L_{\mathcal Z_i}^+[\psi_i]$, i.e.~the closure in~$\mathbb R(\mathcal F)$ of the stable laminations~$\mathcal L_{\mathcal Z_i}^+[\psi_i]$.

\begin{prop}\label{prop-limitpseudolox2}

Let $\psi\colon \mathcal F \to \mathcal F$ be an exponentially growing automorphism with a limit pseudoforest $(\mathcal T, (\delta_i)_{i=1}^n)$.
For $1 \le j \le n$ and $\mathcal T$-loxodromic $x \in \mathcal F$, the axis for~$x$ in~$\mathcal T$ contains a pseudoleaf segment of~$\mathcal L_{\mathcal T}^+[\psi_j]$ if and only if the axis for~$x$ in $\mathbb R(\mathcal F)$ weakly $\psi_*$-limits to the attracting laminations~$\mathcal L_j^+[\psi]$.
\end{prop}


\begin{proof}
Let $(\mathcal T, (\delta_i)_{i=1}^n)$ be the limit pseudoforest for the descending sequence $(\mathcal Y_i, \delta_i)_{i=1}^n$ of limit forests for~$[\psi]$.
For $i \le n$, pick a descending sequence $(\tau_{i,j})_{j=1}^{n_i}$ of irreducible train tracks for~$[\psi_i]$ rel.~$\mathcal Z_i$; 
we can assume~$\tau_{i+1,j}$ is defined on a free splitting of~$\mathcal Z_i$ for some $j<n_{i+1}$ since~$[\psi_{i+1}]$ is polynomially growing rel.~$\mathcal Z_i$ (Theorem~\ref{thm-limitloxodromics}).
The train tracks $(\tau_{i,j})_{j=1}^{n_i}$ induce a $\psi_i$-equivariant $\lambda_i$-Lipschitz PL-map $\tau_i^\circ \colon (\mathcal T_i^\circ, d_i^\circ) \to (\mathcal T_i^\circ, d_i^\circ)$.
Fix a metric free splitting $(\mathcal T^\star, d^\star)$ of~$\mathcal F$ that is the metric blow-up of $(\mathcal T_1^\circ, d_1^\circ)$, $(\mathcal T_{i+1}^\circ(\mathcal Z_{i}), d_{i+1}^\circ)$ for $i < n$, and some metric free splitting $(\mathcal T_{n+1}^\circ, d_{n+1}^\circ)$ of~$\mathcal Z_n$ whose free factor system $\mathcal F[\mathcal T_{n+1}^\circ]$ is trivial.
As the $\mathcal G_i$-orbit of~$\mathcal T_i^\circ(\mathcal Z_{i-1})$ is $\tau_i^\circ$-invariant, the maps~$(\tau_i^\circ)_{i=1}^n$ induce a $\psi$-equivariant PL-map~$\tau^\star$ on $(\mathcal T^\star, d^\star)$.

Let $x \in \mathcal F$ be a $\mathcal T$-loxodromic element. 
By construction, the element~$x$ is conjugate to a $\mathcal Y_i$-loxodromic $y_i \in \mathcal G_i$ for some $i \le n$;
let~$l_i^\circ$ be the $\mathcal T_i^\circ$-axis for~$y_i$.
If $j = i$, then the equivalence in the proposition's statement follows from Theorem~\ref{thm-limitloxodromics}.
For the rest of the proof, we prove the equivalence when $j > i$.
As we are going to invoke the same argument in the next proof, we mostly forget that~$l_i^\circ$ is a $\mathcal T_i^\circ$-axis for a $\mathcal Y_i$-loxodromic element and only use the fact $[l_i^\circ] \in \mathbb R(\mathcal Y_i, \delta_i)$, i.e.~$l_i^\circ$ projects to a line~$\gamma_i$ in $(\mathcal Y_i, \delta_i)$.

\smallskip
Suppose the $\mathcal T$-axis for~$x$ contains a pseudoleaf segment of~$\mathcal L_{\mathcal T}^+[\psi_j]$ for some $j > i$.
Then the $\mathcal T$-axis for~$y_i$ contains a pseudoleaf segment~$\sigma_j$ of~$\mathcal L_{\mathcal T}^+[\psi_j]$ and~$\kappa_i(\sigma_{j})$ is a point $\circ_i \in \gamma_i$ with nontrivial point stabilizer $G_{\circ_i} \defeq \operatorname{Stab}_{\mathcal G_i}(\circ_i)$.
In Subsection~\ref{Subsubsec-iteratedturns}, we describe how the line~$\gamma_i$ in $(\mathcal Y_i, \delta_i)$ and point~$\circ_i \in \gamma_i$ determine an algebraic iterated turn $(\epsilon, s_{i+1,1}^{-1}s_{i+1,2}; \varphi_{i+1})_{\mathcal G_{i+1}}$.
Any iterated turn $(\beta_{i+1,m})_{m \ge 0}$ over~$\mathcal T_{i+1}^\circ$ realizing this algebraic iterated turn limits to an interval $[\star_{i+1,1}, \star_{i+1,2}]$ in the metric completion $(\overline{\mathcal Y}_{i+1}, \delta_{i+1})$ by Theorem~\ref{thm-limitturns}, and $[\star_{i+1,1}, \star_{i+1,2}]$ contains~$\kappa_{i+1}(\sigma_j)$.

If $j = i+1$, then $[\star_{i+1,1}, \star_{i+1,2}] \supset \kappa_{i+1}(\sigma_{i+1})$ is not degenerate and $(\beta_{i+1,m})_{m \ge 0}$ weakly limits to a component of~$\mathcal L_{\mathcal Z_{i+1}}^+[\psi_{i+1}]$ by Theorem~\ref{thm-limitturns}.
Otherwise, for $k \ge i+1$, 
assume~$\kappa_{k}(\sigma_j)$ is a point~$\circ_{k}$ in the interval $[\star_{k,1}, \star_{k,2}] \subset \overline{\mathcal Y}_{k}$ corresponding to the algebraic iterated turn $(\epsilon, s_{k,1}^{-1}s_{k,2}; \varphi_{k})_{\mathcal G_{k}}$,
where $\circ_{k}$ has nontrivial stabilizer~$G_{\circ_{k}}$. 
By the discussion in Subsection~\ref{Subsubsec-nestedturns} (and remark after Theorem~\ref{thm-limitturns}), 
the algebraic iterated turn over~$\mathcal G_{k}$
and point~$\circ_{k}$ in $[\star_{k,1}, \star_{k,2}]$ determine an algebraic iterated turn $(\epsilon, s_{k+1,1}^{-1}s_{k+1,2}; \varphi_{k+1})_{\mathcal G_{k+1}}$
that limits to $[\star_{k+1,1}, \star_{k+1,2}] \subset \overline{\mathcal Y}_j$;
morevoer, $[\star_{k+1,1}, \star_{k+1,2}]$ contains $\kappa_{k+1}(\sigma_j)$.
By induction, $[\star_{j,1}, \star_{j,2}]$ contains $\kappa_{j}(\sigma_{j})$.
Since~$\kappa_{j}(\sigma_{j})$ is not degenerate, any realization $(\beta_{j,m})_{m \ge 0}$ over~$\mathcal T_j^\circ$ of the algebraic iterated turn $(\epsilon, s_{j,1}^{-1}s_{j,2}; \varphi_{j})_{\mathcal G_{j}}$ weakly limits to a component of~$\mathcal L_{\mathcal Z_{j}}^+[\psi_{j}]$ by Theorem~\ref{thm-limitturns}.

In either case ($j \ge i+1$), any realization over~$\mathcal T^\star$ of $(\epsilon, s_{j,1}^{-1}s_{j,2}; \varphi_{j})_{\mathcal G_j}$ weakly limits to (the closure in~$\mathbb R(\mathcal F)$ of) a component of $\mathcal L_{\mathcal Z_{j}}^+[\psi_{j}]$ (bounded cancellation).
If $j > i+1$, any realization over~$\mathcal T^\star$ of $(\epsilon, s_{i+1,1}^{-1}s_{i+1,2}; \varphi_{i+1})_{\mathcal G_{i+1}}$ weakly limits to a component of~$\mathcal L_{\mathcal Z_{j}}^+[\psi_{j}]$ by transitivity.
Hence the shared axis for~$y_i$ and~$x$ in~$\mathbb R(\mathcal F)$ weakly $\psi_*$-limits to a component of~$\mathcal L_{\mathcal Z_{j}}^+[\psi_{j}]$.
As $\psi_{j*} \colon \mathbb R(\mathcal G_j, \mathcal Z_j)  \to \mathbb R(\mathcal G_j, \mathcal Z_j)$ acts transitively on the components of~$\mathcal L_{\mathcal Z_{j}}^+[\psi_{j}]$, the axis for~$x$ in~$\mathbb R(\mathcal F)$ weakly $\psi_*$-limits to~$\mathcal L_{j}^+[\psi]$, the closure in~$\mathbb R(\mathcal F)$ of~$\mathcal L_{\mathcal Z_{j}}^+[\psi_{j}]$.

\smallskip
Conversely, suppose the axis~$[l^\star]$ for~$y_i$ (and~$x$) in~$\mathbb R(\mathcal F)$ weakly $\psi_*$-limits to~$\mathcal L_j^+[\psi]$ for some $j > i$.
Using $(\mathcal T^\star, d^\star)$-coordinates,
the axis~$\tau_*^{\star m}(l^\star)$ contains arbitrarily $d_j^\circ$-long leaf segments of~$\mathcal L_j^+[\psi]$ for $m \gg 1$.
So~$\tau_*^{\star M}(l^\star)$ has a $\mathcal L_{j}^+[\psi]$-leaf segment~$I^\star \subset \mathcal T^\star(\mathcal Z_{j-1})$ with $d_{j}^\circ$-length $L > C' \defeq \frac{2C[\tau^\star]}{\lambda_{j}-1}$ for $M \gg 1$.
As~$\tau_{j}^\circ$ is a train track on leaves of~$\mathcal L_{\mathcal Z_{j}}^+[\psi_{j}]$,~$\tau_*^{\star m}(l^\star)$ has a $\mathcal L_{j}^+[\psi]$-leaf segment surviving from~$I^\star$ with $d_j^\circ$-length $>\lambda_{j}^{M-m}(L-C')$ for $m \ge M$.

Let $\rho_{i} \colon (\mathcal T^\star(\mathcal G_{i}), d^\star) \to (\mathcal T_{i}^\circ, d_{i}^\circ)$ be an arbitrary equivariant PL-map.
The $\rho_i$-image of $I^\star \subset \tau_*^{\star M}(l^\star)$ is a vertex~$v \in \tau_{i *}^{\circ M}(l_i^\circ)$ with nontrivial stabilizer.
Since a nondegenerate part of~$I^\star$ survives in~$\psi_*^{m}(l^\star)$ for all $m \ge M$, we have $\tau_i^{\circ (m-M)}(v) \in \tau_{i *}^{\circ m}(l_i^\circ)$ for all $m \ge M$ and $h_i^{-M}(\pi_i^\circ(v)) \in \gamma_i$ has a nontrivial stabilizer $G_v \defeq \operatorname{Stab}_{\mathcal G_i}(h_i^{-M}(\pi_i^\circ(v)))$, where~$h_i$ is the $\psi_i$-equivariant $\lambda_i$-homothety on $(\mathcal Y_i, \delta_i)$.
As before, the line~$\gamma_i$, point $h_i^{-M}(\pi_i^\circ(v)) \in \gamma_i$, and equivariant PL-maps $\rho_{i+1}$, \dots, $\rho_{j}$ determine nested iterated turns over $\mathcal T_{i+1}^\circ$, \dots, $\mathcal T_{j}^\circ$ limiting to intervals in $\overline{\mathcal Y}_{i+1}$, \dots, $\overline{\mathcal Y}_{j}$.
By the computation in the previous paragraph and quasiperiodicity of stable laminations, the last iterated turn over~$\mathcal T_j^\circ$ weakly limits to a component of~$\mathcal L_{\mathcal Z_{j}}^+[\psi_{j}]$.
So the corresponding interval $[\star_{j,1}, \star_{j,2}] \subset \overline{\mathcal Y}_{j}$ is not degenerate (Theorem~\ref{thm-limitturns}) and the $\mathcal T$-axis for~$y_i$ has an intersection with~$\mathcal T_{j}$ whose $\kappa_{j}$-image is~$[\star_{j,1}, \star_{j,2}]$.
By the description of intervals in~$\mathcal Y_{j}$, $[\star_{j,1}, \star_{j,2}]$ contains a leaf segment of $\pi_{j *}^\circ(\mathcal L_{\mathcal Z_{j}}^+[\psi_{j}])$;
therefore, the $\mathcal T$-axes of~$y_i$ and~$x$ contain pseudoleaf segments of~$\mathcal L_{\mathcal T}^+[\psi_j]$.
\end{proof}

In fact, the containment relation on pseudoleaf segments detects the partial order on the set of attracting laminations:

\begin{claim}\label{claim-pseudoleafseg}
For $1\le i, j \le n$, a pseudoleaf segment of $\mathcal L_{\mathcal T}^+[\psi_i]$ contains a pseudoleaf segment of $\mathcal L_{\mathcal T}^+[\psi_j]$ if and only if~$\mathcal L_i^+[\psi]$ contains~$\mathcal L_j^+[\psi]$.
\end{claim}

\noindent We only sketch the proof as it is almost identical to the proof of Proposition~\ref{prop-limitpseudolox2}.

\begin{proof}[Sketch of proof]
There's nothing to show if $i = j$.
Without loss of generality, assume $i < j$;
certainly, $\mathcal L_j^+[\psi]$ does not contain $\mathcal L_i^+[\psi]$ and no pseudoleaf segment of~$\mathcal L_{\mathcal T}^+[\psi_j]$ can contain a pseudoleaf segment of~$\mathcal L_{\mathcal T}^+[\psi_i]$.
Let~$[l_i^\circ]$ be an eigenline in $(\mathcal T_i^\circ, d_i^\circ)$ of~$[\tau_i^{\circ k}]$ for some $k \ge 1$, and~$l^\star$ be the lift of~$l_i^\circ$ to $(\mathcal T^\star, d^\star)$.
The projection~$\gamma_i$ (of~$l_i^\circ$) is a line in $(\mathcal Y_i, \delta_i)$, and we denote by $l_i \subset \mathcal T_i$ its lift via~$\kappa_i$ to a pseudoleaf of~$\mathcal L_{\mathcal T}^+[\psi_i]$.

Suppose the pseudoleaf~$l_i$ of~$\mathcal L_{\mathcal T}^+[\psi_i]$ contains a pseudoleaf segment~$\sigma_j$ of~$\mathcal L_{\mathcal T}^+[\psi_j]$. 
Then 
$\kappa_i(\sigma_{j})$ is a point $\circ_i \in \gamma_i$ with nontrivial point stabilizer~$G_{\circ_i}$.
By the same argument as in the previous proof, the line~$l^\star$ in~$\mathbb R(\mathcal F)$ weakly $\psi_*$-limits to~$\mathcal L_j^+[\psi]$.
Note that $\psi_*^k[l^\star] = [l^\star]$ in~$\mathbb R(\mathcal F)$ as~$[l_i^\circ]$ is an eigenline for~$[\tau_i^{\circ k}]$;
moreover,~$\mathcal L_i^+[\psi]$ consists of the closures in~$\mathbb R(\mathcal F)$ of~$[l^\star]$, \dots,~$\psi_*^{k-1}[l^\star]$ since~$\mathcal L_i^+[\psi]$ is a $\psi_*$-orbit of attracting laminations.
So $\mathcal L_i^+[\psi] \supset \mathcal L_{j}^+[\psi]$.

Conversely, suppose $\mathcal L_i^+[\psi] \supset \mathcal L_{j}^+[\psi]$.
As~$\mathcal L_i^+[\psi]$ and~$\mathcal L_{j}^+[\psi]$ are $\psi_*$-orbits of attracting laminations, the line~$l^\star$ contains arbitrarily $d_{j}^\circ$-long leaf segments of~$\mathcal L_{j}^+[\psi]$.
By the same argument as in the previous proof, the pseudoleaf~$l_i$, and hence some pseudoleaf segment of~$\mathcal L_{\mathcal T}^+[\psi_i]$, contains a pseudoleaf segment of~$\mathcal L_{\mathcal T}^+[\psi_j]$.
\end{proof}

\subsection{Topmost forests}\label{SubsecTopmost}

Fix an exponentially growing automorphism $\psi \colon \mathcal F \to \mathcal F$ with a descending sequence $(\mathcal Y_i, \delta_i)_{i=1}^n$ of limit forests, and let $(\mathcal T, (\delta_i)_{i=1}^n)$ be the limit pseudoforest for $(\mathcal Y_i, \delta_i)_{i=1}^n$.
Each limit forest $(\mathcal Y_i, \delta_i)$ has 
stable laminations~$\mathcal L_{\mathcal Z_i}^+[\psi_i]$ for~$[\psi_i]$ rel.~$\mathcal Z_i$.
Let~$\mathcal L_{\mathcal T}^+[\psi_i]$ be the lifts to~$\mathcal T$ of leaves in~$\mathcal L_{\mathcal Z_i}^+[\psi_i]$,~$\mathcal L_i^+[\psi]$ the closure in~$\mathbb R(\mathcal F)$ of~$\mathcal L_{\mathcal Z_i}^+[\psi_i]$, and $\{\mathcal A_{j}^{top}[\psi_i]\}_{j=1}^{k_i}$ the subset of $\{\mathcal L_j^+[\psi]\}_{j=i}^n$ consisting of all topmost attracting laminations for~$[\psi_i]$.
So $\mathcal A_j^{top}[\psi_i] = \mathcal L_{\iota(i,j)}^+[\psi]$ for some subsequence~$(\iota(i,j))_{j=1}^{k_i}$ of $(j)_{j=i}^n$ with $\iota(i,1) = i$, and $(\iota(i,j))_{j=2}^{k_i}$ is a subsequence of $(\iota(i+1,j))_{j=1}^{k_{i+1}}$ if $k_i \ge 2$.

%
For $i \ge 1$, we say the $\mathcal G_i$-invariant hierarchy $(\delta_j)_{j=i}^n$ on the characteristic convex subsets $\mathcal T_i \subset \mathcal T$ for~$\mathcal G_i$ \underline{normalizes} to a factored $\mathcal G_i$-invariant convex pseudometric $\Sigma_{j=1}^{k_i} \delta_{\iota(i,j)}$ if the $\mathcal G_{\iota(i,j)}$-invariant convex pseudometric~$\delta_{\iota(i,j)}$ can be extended to a $\mathcal G_i$-invariant convex pseudometric, also denoted~$\delta_{\iota(i,j)}$, on~$\mathcal T_i$. 
The $\mathcal F$-invariant hierarchy $(\delta_i)_{i=1}^n$ normalizes to~$\delta_1$ if (and only if) $k_1=1$.

We may assume $k_1 \ge 2$ and the $\mathcal G_2$-invariant hierarchy $(\delta_i)_{i=2}^n$ normalizes to $\oplus_{j=1}^{k_2} \delta_{\iota(2,j)}$.
Let~$\widehat{\mathcal T}_2$ be the $\kappa_1$-preimage of the characteristic convex subsets $\mathcal Y_1(\mathcal G_2)$.
Suppose $\mathcal F_{1,1} \defeq \mathcal F$, \dots,~$\mathcal F_{1, m}$ are the proper free factor systems of~$\mathcal F$ used to construct $(\mathcal Y_1, \delta_1)$ and let $\mathcal T_{1, 1}$, \dots,~$\mathcal T_{1, m}$ be their corresponding characteristic convex subsets in~$\mathcal T$.
By Lemma~\ref{lem-leafsegments}, every closed interval
in the characteristic convex subsets~$\mathcal Y_1(\mathcal F_{1,m})$ is a finite concatenation of leaf segments of~$\mathcal L_{\mathcal Z_1}^+[\psi_1]$.
Thus every closed interval in $\mathcal T_{1, m}$ is a finite concatenation of pseudoleaf segments of~$\mathcal L_{\mathcal T}^+[\psi_1]$ and closed intervals in $\mathcal F_{1,m} \cdot \widehat{\mathcal T}_2$.

Fix $j \in \{2, \ldots, k_1\}$.
Since $\mathcal L_1^+[\psi]$ does not contain $\mathcal L_{\iota(1,j)}^+[\psi]$, Claim~\ref{claim-pseudoleafseg} implies the intersection of any pseudoleaf segment of~$\mathcal L_{\mathcal T}^+[\psi_1]$ with~$\widehat{\mathcal T}_2$ has 0 diameter with respect to the convex pseudometric~$\delta_{\iota(1,j)}$;
we say that $\mathcal L_{\mathcal Z_1}^+[\psi_1]$ and~$\delta_{\iota(1,j)}$ are \underline{independent}.
So the intersection of any closed interval in~$\mathcal T_{1,m}$ with $\mathcal F_{1,m} \cdot \widehat{\mathcal T}_2$ has finitely many components that are translates of closed intervals in~$\widehat{\mathcal T}_2$ with positive $\delta_{\iota(1,j)}$-diameter.
Thus~$\delta_{\iota(1,j)}$ can be extended to an $\mathcal F_{1,m}$-invariant convex pseudometric on~$\mathcal T_{1,m}$ that is mutually singular with~$\delta_1$.
By our inductive description of intervals in $\mathcal Y_1$ (Lemma~\ref{lem-leafblocks}), the convex pseudometric~$\delta_{\iota(1,j)}$ extends equivariantly to~$\mathcal T$ as $\lambda_{\iota(1,j)} > 1$.

As~$j$ was arbitrary, the $\mathcal F$-invariant hierarchy $(\delta_i)_{i=1}^n$ normalizes to the factored convex pseudometric $\oplus_{j=1}^k \delta_{\iota(j)}$, where $k \defeq k_1$ and $\iota(j) \defeq \iota(1,j)$.
Let $(\mathcal Y, \oplus_{j=1}^{k} \delta_{\iota(j)})$ be the associated factored $\mathcal F$-forest.
The real $\mathcal F$-pretrees~$\mathcal Y$ are minimal and have trivial arc stabilizers since the pseudometric~$\oplus_{j=1}^{k} \delta_{\iota(j)}$ on~$\mathcal T$ is convex.
The $\psi$-equivariant $(\lambda_i)_{i=1}^n$-homothety~$h$ induces a $\psi$-equivariant \underline{$\oplus_{j=1}^{k} \lambda_{\iota(j)}$-dilation} on $(\mathcal Y, \oplus_{j=1}^{k} \delta_{\iota(j)})$: a $\lambda_{\iota(j)}$-homothety with respect to each factor~$\delta_{\iota(j)}$.
By Proposition~\ref{prop-limitpseudolox2}, a nontrivial element of~$\mathcal F$ is {$\delta_{\iota(j)}$-loxodromic} if and only if its axis in $\mathbb R(\mathcal F)$ weakly $\psi_*$-limits to~$\mathcal A_{j}^{top}[\psi_1]$ --- here, $\delta_{\iota(j)}$-loxodromic means the element acts loxodromically on the associated $\mathcal F$-forest for~$\delta_{\iota(j)}$.
The factored $\mathcal F$-forest $(\mathcal Y, \oplus_{j=1}^{k} \delta_{\iota(j)})$ is the \underline{complete topmost limit forest} for $(\mathcal Y_i, \delta_i)_{i=1}^n$.
%
Given any subset of the $\psi_*$-orbits of topmost attracting laminations for~$[\psi]$, then one may consider the associated factored $\mathcal F$-forest for the sum of corresponding pseudometrics:

\begin{thm}\label{thm-topmostlimitexist} 
Let $\psi \colon \mathcal F \to \mathcal F$ be an automorphism and $\{ \mathcal A_j^{top}[\psi]\}_{j=1}^k$ a (possibly empty) subset of $\psi_*$-orbits of topmost attracting laminations for~$[\psi]$. 

Then there is:
\begin{enumerate}
\item a minimal factored $\mathcal F$-forest $(\mathcal Y, \oplus_{j=1}^k \delta_j)$ with trivial arc stabilizers;
\item \label{cond-topmost-expanding} a unique $\psi$-equivariant expanding dilation $f \colon (\mathcal Y, \oplus_{j=1}^k \delta_j) \to (\mathcal Y, \oplus_{j=1}^k \delta_j)$; and
\item \label{cond-topmost-loxodromics} for $1 \le j \le k$, a nontrivial element $x \in \mathcal F$ is $\delta_j$-loxodromic if and only if its axis in~$\mathbb R(\mathcal F)$ weakly $\psi_*$-limits to~$\mathcal A_j^{top}[\psi]$. \qed
\end{enumerate}
\end{thm}

Fix some index $\iota(j) \neq 1$, and let~$\mathcal X_{1,m}$ be the associated $\mathcal F_{1,m}$-forest for $\delta_1 \oplus \delta_{\iota(j)}$ on~$\mathcal T_{1,m}$.
Two lines in $(\mathcal X_{1,m}, \delta_1 \oplus \delta_{\iota(j)})$ representing leaves in~$\mathcal L_{\mathcal Z_1}^+[\psi]$ \underline{overlap} if they have a nondegenerate intersection;
overlapping generates an equivalence relation and each {overlapping} class is identified with its union in~$\mathcal X_{1,m}$.
Let $\operatorname{\underline{supp}}[\psi_1; {\mathcal Z_1}]$ denote the subgroup system 
corresponding to the (setwise) stabilizers of overlapping classes~$L_{\mathcal Z_1}^+$ --- this subgroup system, called the \underline{lower-support} of~$\mathcal L_{\mathcal Z_1}^+[\psi_1]$, is $[\psi]$-invariant.
The system~$\operatorname{\underline{supp}}[\psi_1; {\mathcal Z_1}]$ is not empty as there are $\mathcal Y_1$-loxodromic elements whose axis in~$\mathcal Y_1^*$ is contained in~$L_{\mathcal Z_1}^+$.
Note the number of components in~$\operatorname{\underline{supp}}[\psi_1; {\mathcal Z_1}]$ is at most the number of components in~$\mathcal L_{\mathcal Z_1}^+[\psi_1]$.
Let $(\widehat{\mathcal X}_{1,m}(\mathcal G_2), \delta_{\iota(j)})$ be the closure in $(\mathcal X_{1,m}, \delta_1 \oplus \delta_{\iota(j)})$ of the characteristic subforest for~$\mathcal G_{2}$.
By Lemma~\ref{lem-leafsegments} again, intervals in~$\mathcal X_1(\mathcal F_{1,n_1})$ are finite concatenations of leaf segments of $\mathcal L_{\mathcal Z_1}^+[\psi_1]$ and closed intervals in~$\widehat{\mathcal X}_{1,m}(\mathcal G_{2})$.

The overlapping classes~$L_{\mathcal Z_1}^+$ and the $\mathcal F_{1,m}$-orbits of components of~$\widehat{\mathcal X}_{1,m}(\mathcal G_{2})$ form an $\mathcal F_{1,m}$-invariant {transverse covering} of~$\mathcal X_{1,m}$ (see \cite[Definition~4.6]{Gui04}).
Let~$\mathcal S'$ be a simplicial $\mathcal F_{1,m}$-pretree: vertices (``component-vertices'') in equivariant bijective correspondence with the components of the transverse covering (overlapping classes~$L_{\mathcal Z_1}^+$ and translates of components of~$\widehat{\mathcal X}_{1,m}(\mathcal G_{2})$);
for each point in~$\mathcal X_{1,m}$ contained in exactly two components of the transverse covering, there is an edge between the corresponding component-vertices;
for each point contained in more than two components, there is a new vertex (''intersection-vertex'') and an edge connecting it to each relevant component-vertex.
By the blow-up construction, translates of components of~$\widehat{\mathcal X}_{1,m}(\mathcal G_{2})$ either coincide or are disjoint.
In particular, each intersection-vertex $v \in \mathcal S'$ with a nontrivial stabilizer is adjacent to a unique vertex $w \in \mathcal S'$ corresponding to a component of $\mathcal F_{1,m} \cdot \widehat{\mathcal X}_{1,m}(\mathcal G_{2})$ and the stabilizer of~$v$ fixes~$w$;
therefore, we can collapse all such edges $[v,w]$ to form a simplicial $\mathcal F_{1,m}$-pretree~$\mathcal S$ whose intersection-vertices have trivial stabilizers.

The $\mathcal F_{1,m}$-forest $(\mathcal X_{1,m}, \delta_1 \oplus \delta_{\iota(j)})$ is a graph of actions with {skeleton}~$\mathcal S$ and the nondegenerate ``vertex trees'' are the components of the transverse covering~\cite[Lemma~4.7]{Gui04}.
As the $\psi_{1,m}$-equivariant expanding dilation on~$(\mathcal X_{1,m}, \delta_1 \oplus \delta_{\iota(j)})$ permutes the overlapping classes (and components of $\mathcal F_{1,m} \cdot \widehat{\mathcal X}_{1,m}(\mathcal G_{2})$), it induces a $\psi_{1,m}$-equivariant simplicial automorphism $\sigma \colon \mathcal S \to \mathcal S$ that preserves the ``type'' of a vertex.

Let~$\mathcal T_1^\diamond$ be an equivariant blow-up of $(\mathcal T_{1, j})_{j=1}^{m-1}$,~$\mathcal S$, and~$\mathcal X_{1,m}(\mathcal G_2)$.
When the metric~$\delta_{\iota(j)}$ is extended appropriately to~$\mathcal T_1^\diamond$, the simplicial automorphisms $(\tau_{1,j})_{j=1}^{m-1}$,~$\sigma$, and the homothety~$f_2$ on~$\mathcal X_{1,m}(\mathcal G_2)$ induce a $\psi$-equivariant $\lambda_{\iota(j)}$-Lipschitz map $\tau^\diamond \colon (\mathcal T_1^\diamond, \delta_{\iota(j)}^\diamond) \to (\mathcal T_1^\diamond, \delta_{\iota(j)}^\diamond)$ that linearly extends~$f_2$.
Using $\tau^\diamond$-iteration, we define the limit forest $(\mathcal X_1, \delta_{\iota(j)})$ for $[\tau_i]_{i=1}^{n-1}$, $\sigma$, and~$f_2$.
Like the previous convergence criteria, the proof of the following lemma is postponed to Section~\ref{SubsecConverge3}.

\begin{restate}{Lemma}{lem-convergence3}
Let $\psi\colon \mathcal F \to \mathcal F$ be an automorphism, $\left(\tau_i \colon \mathcal T_i \to \mathcal T_i\right)_{i=1}^n$ a descending sequence of irreducible train tracks for~$[\psi]$, $\mathcal Z \defeq \mathcal F[\mathcal T_n]$, $\mathcal G$ the nontrivial point stabilizer system for the limit forest for~$[\psi]$ rel.~$\mathcal Z$, $[\psi_{\mathcal G}]$ the $[\psi]$-restriction to~$\mathcal G$, $(\mathcal Y_{\mathcal G}, \delta)$ a minimal $\mathcal G$-forest with trivial arc stabilizers such that $\mathcal L_{\mathcal Z}^+[\psi]$ and~$\delta$ are independent, $h_{\mathcal G} \colon (\mathcal Y_{\mathcal G}, \delta) \to (\mathcal Y_{\mathcal G}, \delta)$ a $\psi_{\mathcal G}$-equivariant $\lambda$-homothety, $\mathcal S$ a minimal simplicial $\mathcal F[\mathcal T_{n-1}]$-forest that is the skeleton for the graph of actions for~$\mathcal L_{\mathcal Z}^+[\psi]$ and~$\delta$, $\sigma \colon \mathcal S \to \mathcal S$ the corresponding simplicial automorphism, and $(\mathcal X, \delta)$ the limit forest for~$[\tau_i]_{i=1}^{n-1}$,~$\sigma$, and~$h_{\mathcal G}$.

If $(\mathcal Y', \delta')$ is a minimal $\mathcal F$-forest with trivial arc stabilizers, the characteristic subforest of $(\mathcal Y', \delta')$ for~$\mathcal G$ is equivariantly isometric to $(\mathcal Y_{\mathcal G}, \delta)$, and the lower-support~$\operatorname{\underline{supp}}[\psi; {\mathcal Z}]$ of $\mathcal L_{\mathcal Z}^+[\psi]$ is $\mathcal Y'$-elliptic, then the limit of $(\mathcal Y' \psi^{m}, \lambda^{-m} \delta')_{m \ge 0}$ is $(\mathcal X, \delta)$. 
\end{restate}

Fix a subset~$\{\mathcal A_j^{top}[\psi]\}_{j=1}^k$ of $\psi_*$-orbits of topmost attracting laminations for~$[\psi]$;
a \underline{topmost forest} for~$[\psi]$ is a factored $\mathcal F$-forest satisfying the conclusion of Theorem~\ref{thm-topmostlimitexist} with respect to this subset.
Lemma~\ref{lem-convergence3} is enough to prove the universality of topmost forests:

\begin{thm}\label{thm-topmostexpandunique}
Let $\psi \colon \mathcal F \to \mathcal F$ be an automorphism and $\{\mathcal A_j^{top}[\psi]\}_{j=1}^k$ a  (possibly empty) subset of $\psi_*$-orbits of topmost attracting laminations for~$[\psi]$.
Any topmost forest for~$[\psi]$ with respect to the given subset has a unique equivariant dilation to any corresponding topmost limit forest for~$[\psi]$.
\end{thm}

\noindent Thus the factored $\mathcal F$-forest $(\mathcal Y, \oplus_{j=1}^{k} \delta_{\iota(j)})$ is \emph{the} \underline{complete topmost forest} for~$[\psi]$ (up to rescaling of the factors). We omit the proof since we are about to prove something stronger in the next section (see Theorem~\ref{thm-dominatingexpandunique}).

Suppose $(\mathcal T, (\delta_i)_{i=1}^n)$ and $(\mathcal T', (\delta_i)_{i=1}^{n'})$ are two limit pseudoforests for~$[\psi]$.
Then $n = n'$ as they are exactly the number of $\psi_*$-orbits of attracting laminations for~$[\psi]$.
Using Theorem~\ref{thm-topmostlimitexist}, the hierarchies can be inductively normalized to $(\oplus_{j=1}^{k_i} \delta_{i,j})_{i=1}^d$ and $(\oplus_{j=1}^{k_i} \delta_{i,j}')_{i=1}^d$ respectively, where~$d$ is the length of the longest chain in the partial order of attracting laminations for~$[\psi]$ and $\delta_{i,j}, \delta_{i,j}'$ are indexed by the same $\psi_*$-orbit~$\mathcal A_{i,j}[\psi]$ of attracting laminations.
By inductively invoking Theorem~\ref{thm-topmostexpandunique} and uniqueness of the blow-up construction, the normalized pseudoforests $(\mathcal T, (\oplus_{j=1}^{k_i} \delta_{i,j})_{i=1}^d)$ and ($\mathcal T', (\oplus_{j=1}^{k_i} \delta_{i,j}')_{i=1}^d)$ are in the same equivariant dilation class and invariants of this class are invariants of~$[\psi]$!
In particular, $\mathcal T$ and~$\mathcal T'$ are equivariantly pretree-isomorphic.

\begin{cor}\label{cor-limitpretreeunique}
Any two limit pretrees for an automorphism $\psi \colon \mathcal F \to \mathcal F$ are equivariantly pretree-isomorphic. \qed
\end{cor}

We can now define more invariants of an attracting lamination:
let~$A$ be an attracting lamination for~$[\psi]$,~$\mathcal A[\psi]$ its $\psi_*$-orbit, and $(\delta_{i,j}, \lambda_{i,j})$ the corresponding pair of pseudometric and stretch factor in the normalized pseudoforest $(\mathcal T, (\oplus_{j=1}^{k_i} \delta_{i,j})_{i=1}^d)$;
then $\lambda(A) \defeq \lambda_{i,j}$ is a well-defined \underline{stretch factor} for~$A$.
Now let~$A$ be topmost, $\{ \mathcal B_{i',j'} \}$ be the whole subset of $\psi_*$-orbits of attracting laminations not contained in~$\mathcal A[\psi]$, and~$(\mathcal T_A, (\oplus_{j'=1}^{k_i'} \delta_{i',j'}')_{i'=1}^{d'})$ the associated normalized pseudoforest.
Then the \underline{upper-support} of~$\mathcal A[\psi]$ is the subgroup system of point stabilizers $\operatorname{\overline{supp}} \mathcal A[\psi] \defeq \mathcal G[\mathcal T_A]$.
Unlike the lower-support, the upper-support is always a malnormal subgroup system of finite type.
Note that components of the lower-support are conjugate into components of the upper-support.

\subsection{Dominating forests}\label{SubsecDominating}



Fix an exponentially growing automorphism $\psi \colon \mathcal F \to \mathcal F$.
Let $A \subset \mathbb R(\mathcal F)$ be an attracting lamination for~$[\psi]$ and~$\lambda(A)$ its stretch factor.
We say~$A$ is \underline{dominating} if $\lambda(A) > \lambda(A')$ whenever~$A'$ is an attracting lamination for~$[\psi]$ containing~$A$ and~$A' \neq A$;
topmost attracting laminations are vacuously dominating.
We will extend Theorem~\ref{thm-topmostlimitexist} to dominating attracting laminations by mimicking the reasoning in the previous section, focusing only on the changes needed for dominating attracting laminations.


Let $(\mathcal Y_i, \delta_i)_{i = 1}^n$ be a descending sequence of limit forests for~$[\psi]$, $(\mathcal L_i^+[\psi])_{i=1}^n$ the corresponding sequence of $\psi_*$-orbits of attracting laminations for~$[\psi]$, $(\mathcal T, (\delta_i)_{i=1}^n)$ the limit pseudoforest for $(\mathcal Y_i, \delta_i)_{i=1}^n$, and $\{\mathcal A_j^{dom}[\psi_i]\}_{j=1}^{k_i}$ the subset of $\{\mathcal L_j^+[\psi]\}_{j=i}^n$ consisting of all dominating attracting laminations for~$[\psi_i]$ --- recall that~$\mathcal Y_i$ are $\mathcal G_i$-pretrees and $[\psi_i]$ is the restriction of~$[\psi]$ to~$\mathcal G_i$.
As before, $\mathcal A_j^{dom}[\psi_i] = \mathcal L_{\iota(i,j)}^+[\psi]$ for some subsequence~$(\iota(i,j))_{j=1}^{k_i}$ of $(j)_{j=i}^n$ with $\iota(i,1) = i$. 

Suppose $k_1 \ge 2$ and the $\mathcal G_2$-invariant hierarchy $(\delta_i)_{i=2}^n$ normalizes to the factored $\mathcal G_2$-invariant convex pseudometric $\Sigma_{j=1}^{k_2} \delta_{\iota(2,j)}$ on the characteristic convex subsets $\mathcal T_2 \subset \mathcal T$ for~$\mathcal G_2$.
Fix some $j \in \{2, \ldots, k_1\}$.
The previous section discusses how to equivariantly extend $\delta_{\iota(1,j)}$ to~$\mathcal T$ when $\mathcal L_1^+[\psi]$ does not contain $\mathcal L_{\iota(1,j)}^+[\psi]$.
Assume for the rest of this section that $\mathcal L_{\iota(1,j)}^+[\psi] \subset \mathcal L_1^+[\psi]$;
thus $\lambda_1 < \lambda_{\iota(1,j)}$ as $\mathcal L_{\iota(1,j)}^+[\psi]$ is dominating.
Let $(\mathcal Y^*, (\delta_1, \delta_{\iota(1,j)}))$ be the associated $\mathcal F$-pseudoforest for the $\mathcal F$-invariant 2-level hierarchy $(\delta_1, \delta_{\iota(1,j)})$ and $h^*$ the $\psi$-equivariant pretree-automorphism on~$\mathcal Y^*$ induced by $h \colon \mathcal T \to \mathcal T$.

\smallskip
Let $\tau_1 \colon (\Gamma_1, d_1) \to (\Gamma_1, d_1)$ be the $\lambda_1$-Lipschitz topological representative for~$\psi$ used to construct $(\mathcal Y_1, \delta_1)$ through iteration.
Pick an equivariant blow-up~$\Gamma^\circ$ of $\Gamma_1$ rel.~$\mathcal Y^*(\mathcal Z) \subset \mathcal Y^*$, the characteristic convex subsets for the proper free factor system $\mathcal Z \defeq \mathcal F[\Gamma_1]$.
Since~$\mathcal Z$ is $\delta_1$-elliptic, $\delta_{\iota(1,j)}$ is a metric on $\mathcal Y^*(\mathcal Z)$.
The blow-up inherits an $\mathcal F$-invariant 2-level hierarchy $(d_1, \delta_{\iota(1,j)})$ with full support.
As~$\Gamma_1$ is simplicial, this hierarchy extends to a factored $\mathcal F$-invariant convex metric $d_1 \oplus \delta_{\iota(1,j)}$ on~$\Gamma^\circ$.

Let~$[\psi_{\mathcal Z}]$ be the restriction of~$[\psi]$ to~$\mathcal Z$ and~$h_{\mathcal Z}^*$ the $\psi_{\mathcal Z}$-equivariant ``restriction'' of~$h^*$ to $(\mathcal Y^*(\mathcal Z), \delta_{\iota(1,j)})$.
For a parameter $c > 0$, the topological representative~$\tau_1$ induces a $\psi$-equivariant map~$\tau_c^\circ$ on~$\Gamma^\circ$ that extends~$h_{\mathcal Z}^*$ and is linear with respect to $(c \, d_1) \oplus \delta_{\iota(1,j)}$ on edges from~$\Gamma_1$.
If $c \gg 1$, then~$\tau_c^\circ$ is $\lambda_{\iota(1,j)}$-Lipschitz with respect to $(c \, d_1 ) \oplus \delta_{\iota(1,j)}$ since $\lambda_1 < \lambda_{\iota(1,j)}$.
Through $\tau_c^\circ$-iteration, we define a limit forest $(\mathcal Y, \delta_{\iota(1,j)})$ for $\tau_1$ and $h_{\mathcal Z}^*$ whose characteristic subforest for~$\mathcal Z$ is identified with $(\mathcal Y^*(\mathcal Z), \delta_{\iota(1,j)})$ --- up to equivariant isometry, this limit forest is independent of the parameter~$c$;
moreover, there is an induced $\psi$-equivariant $\lambda$-homothety~$h$ on~$(\mathcal Y, \delta_{\iota(1,j)})$ that restricts to~$h_{\mathcal Z}^*$ on~$\mathcal Y^*(\mathcal Z)$. 

We now refine this construction of a limit forest.
For $n \ge 1$, set $d_n^\circ \defeq \lambda_1^{-n} d_1 \oplus \lambda_{\iota(1,j)}^{-n} \delta_{\iota(1,j)}$ and $\tau^\circ \defeq \tau_1^\circ$.
The map
$\tau^\circ \colon (\Gamma^\circ, d_0^\circ) \to  (\Gamma^\circ \psi, d_1^\circ)$ is equivariant and $(1+D)$-Lipschitz for some~$D \ge 0$.
In fact, $\tau^\circ \colon (\Gamma^\circ, d_n^\circ) \to  (\Gamma^\circ \psi, d_{n+1}^\circ)$  is $(1+D r^n)$-Lipschitz, where $r \defeq \lambda_1 \lambda_{\iota(1,j)}^{-1}$.
Set $p_n \defeq \prod_{i=0}^{n-1} (1+D r^i)$;
then $\tau^{\circ n} \colon (\Gamma^\circ, d_0^\circ) \to  (\Gamma^\circ \psi^n, p_n^{-1} d_n^\circ)$ is equivariant and metric;
moreover, the pullback of $p_n^{-1} d_n^\circ$ to~$\Gamma^\circ$ along~$\tau^{\circ n}$ converges to an $\mathcal F$-invariant pseudometric on~$\Gamma^\circ$ as $n \to \infty$.
Since $|r| < 1$, the sequence $(p_n)_{n=1}^\infty$ converges and the pullback of $d_n^\circ$ to~$\Gamma^\circ$ along~$\tau^{\circ n}$ converges to a factored $\mathcal F$-invariant pseudometric $\delta_1^\circ + \delta_{\iota(1,j)}^\circ$ on~$\Gamma^\circ$.
Let $(\Gamma^*, \delta_1^\circ + \delta_{\iota(1,j)}^\circ)$ be the associated factored $\mathcal F$-forest for this factored pseudometric on~$\Gamma^\circ$ --- as $\mathcal L_{\iota(1,j)}^+[\psi] \subset \mathcal L_1^+[\psi]$, one can show that $\delta_1^\circ$ and $\delta_{\iota(1,j)}^\circ$ are not mutually singular and $\delta_{\iota(1,j)}^\circ$ is actually a metric on~$\Gamma^*$.
By construction, the characteristic subforest for~$\mathcal Z$ in $(\Gamma^*, \delta_1^\circ + \delta_{\iota(1,j)}^\circ)$ is equivariantly isometric to $(\mathcal Y^*(\mathcal Z), \delta_{\iota(1,j)})$.
Similarly, the $\mathcal F$-forest $(\mathcal Y_1, \delta_1)$ is equivariantly isometric to the associated metric space for the pseudometric~$\delta_1^\circ$ on~$\Gamma^*$ or~$\Gamma^\circ$ (Corollary~\ref{cor-unique2}).

\begin{restate}{Lemma}{lem-convergence4}
Let $\psi\colon \mathcal F \to \mathcal F$ be an automorphism,~$\mathcal Z$ a $[\psi]$-invariant proper free factor system, $(\mathcal Y_{\mathcal Z}, \delta)$ a minimal $\mathcal Z$-forest with trivial arc stabilizers, $\left(\tau_i \colon \mathcal T_i \to \mathcal T_i\right)_{i=1}^n$ a descending sequence of irreducible train tracks for~$[\psi]$ with $\mathcal F[\mathcal T_n] = \mathcal Z$, $h_{\mathcal Z} \colon (\mathcal Y_{\mathcal Z}, \delta) \to (\mathcal Y_{\mathcal Z}, \delta)$ a $\psi_{\mathcal Z}$-equivariant $\lambda$-homothety, and $(\mathcal Y, \delta)$  the limit forest for~$[\tau_i]_{i=1}^n$ and~$h_{\mathcal Z}$, where $\lambda > \lambda[\tau_n]$ and $[\psi_{\mathcal Z}]$ is the $[\psi]$-restriction to~$\mathcal Z$.

If $(\mathcal Y', \delta')$ is a minimal $\mathcal F$-forest with trivial arc stabilizers and the characteristic subforest of $(\mathcal Y', \delta')$ for~$\mathcal Z$ is equivariantly isometric to $(\mathcal Y_{\mathcal Z}, \delta)$, then the limit of $(\mathcal Y' \psi^{m}, \lambda^{-m} \delta')_{m \ge 0}$ is $(\mathcal Y, \delta)$. 
\end{restate}

Again, the proof is postponed to Section~\ref{SubsecConverge4}.
Since the restriction of~$[\psi]$ to $\mathcal G_1$ is polynomially growing rel.~$\mathcal Z$, Lemma~\ref{lem-convergence4} implies the characteristic subforests $(\mathcal Y^*(\mathcal G_1), \delta_{\iota(1,j)})$ and $(\Gamma^*(\mathcal G_1), \delta_{\iota(1,j)}^\circ)$ for~$\mathcal G_1$ are equivariantly isometric.
By uniqueness of the blow-up construction,~$\Gamma^*$ is equivariantly pretree-isomorphic to~$\mathcal Y^*$;
through this pretree-isomorphism, we can identify $\delta_{\iota(1,j)}^\circ$ with an extension of~$\delta_{\iota(1,j)}$ to an $\mathcal F$-invariant convex pseudometric (in fact, metric) on~$\mathcal Y^*$.
Finally, we can lift~$\delta_{\iota(1,j)}$ to an $\mathcal F$-invariant convex pseudometric on~$\mathcal T$ since~$\mathcal T$ is an equivariant blow-up of~$\mathcal Y^*$.

As~$j$ was arbitrary, the $\mathcal F$-invariant hierarchy $(\delta_i)_{i=1}^n$ normalizes to the factored $\mathcal F$-invariant convex pseudometric $\Sigma_{j=1}^k \delta_{\iota(j)}$, where $k \defeq k_1$ and $\iota(j) \defeq \iota(1,j)$.
We call the associated factored $\mathcal F$-forest $(\mathcal Y, \Sigma_{j=1}^{k} \delta_{\iota(j)})$ the \underline{complete dominating limit forest} for $(\mathcal Y_i, \delta_i)_{i=1}^n$.
This proves the existence part of our main theorem:

\begin{thm}\label{thm-dominatinglimitexist} 
Let $\psi \colon \mathcal F \to \mathcal F$ be an automorphism and $\{\mathcal A_j^{dom}[\psi]\}_{j=1}^k$ a  (possibly empty) subset of $\psi_*$-orbits of dominating attracting laminations for~$[\psi]$.

Then there is:
\begin{enumerate}
\item a minimal factored $\mathcal F$-forest $(\mathcal Y, \Sigma_{j=1}^k \delta_j)$ with trivial arc stabilizers;
\item \label{cond-dominating-expanding} a unique $\psi$-equivariant expanding dilation $f \colon (\mathcal Y, \Sigma_{j=1}^k \delta_j) \to (\mathcal Y, \Sigma_{j=1}^k \delta_j)$; and
\item \label{cond-dominating-loxodromics} for $1 \le j \le k$, a nontrivial element $x \in \mathcal F$ is $\delta_j$-loxodromic if and only if its axis in~$\mathbb R(\mathcal F)$ weakly $\psi_*$-limits to~$\mathcal A_j^{dom}[\psi]$. \qed
\end{enumerate}
\end{thm}

Fix a subset~$\{\mathcal A_j^{dom}[\psi]\}_{j=1}^k$ of $\psi_*$-orbits of dominating attracting laminations for~$[\psi]$;
a \underline{dominating forest} for~$[\psi]$ is a factored $\mathcal F$-forest satisfying the conclusion of the previous theorem with respect to this subset.
Finally, we prove universality:

\begin{thm}\label{thm-dominatingexpandunique}
Let $\psi \colon \mathcal F \to \mathcal F$ be an automorphism and $\{\mathcal A_j^{dom}[\psi]\}_{j=1}^k$ a  (possibly empty) subset of $\psi_*$-orbits of dominating attracting laminations for~$[\psi]$.
Any dominating forest for~$[\psi]$ with respect to the given subset has a unique equivariant dilation to any corresponding dominating limit forest for~$[\psi]$.
\end{thm}
\begin{proof}
Let $(\mathcal Y_i, \delta_i)_{i=1}^n$ be a descending sequence of limit forests for~$[\psi]$, $\mathcal L_{\mathcal Z_i}^+[\psi_i] \subset \mathbb R(\mathcal G_i, \mathcal Z_i)$ the stable laminations for $(\mathcal Y_i, \delta_i)$, $\mathcal L_i^+[\psi]$ the closure of~$\mathcal L_{\mathcal Z_i}^+[\psi_i]$ in~$\mathbb R(\mathcal F)$, $\{\mathcal L_{\iota(j)}^+[\psi]\}_{j=1}^k$ a subset of $\psi_*$-orbits of dominating attracting laminations,
$(\mathcal Y^*, \Sigma_{j=1}^k \delta_{\iota(j)})$ the corresponding dominating limit forest for $(\mathcal Y_i, \delta_i)_{i=1}^n$, and $(\mathcal Y', \Sigma_{j=1}^k \delta_j')$ a corresponding dominating forest for~$[\psi]$.
Turn the factored metrics into hierarchies, and consider the pseudoforests $(\mathcal Y^*, (\delta_{\iota(j)})_{j=1}^k)$ and $(\mathcal Y', (\delta_j')_{j=1}^k)$.
By Theorem~\ref{thm-dominatinglimitexist}(\ref{cond-dominating-loxodromics}), $\delta_{\iota(1)}$ and $\delta_1'$ have the same maximal elliptic subgroup system~$\mathcal G$.

For induction, assume the $\mathcal G$-pseudoforests $(\mathcal Y^*(\mathcal G), (\delta_{\iota(j)})_{j=2}^k)$ and $(\mathcal Y'(\mathcal G), (\delta_j')_{j=2}^k)$ are equivariantly homothetic.
By uniqueness of the blow-up construction, it is enough to show that the associated $\mathcal F$-forests for~$\delta_{\iota(1)}$ and~$\delta_1'$ (on~$\mathcal Y^*$ and $\mathcal Y'$ respectively) are equivariantly homothetic.
So we may assume $k=1$ for the rest of the proof.
If $\iota(1) = 1$, then $(\mathcal Y^*, \delta_{\iota(1)})$ and $(\mathcal Y', \delta_1')$ are equivariantly homothetic by Lemma~\ref{lem-convergence2}.
Otherwise, $\iota(1) > 1$ and, for induction on complexity, we assume $(\mathcal Y^*(\mathcal G_2), \delta_{\iota(1)})$ and $(\mathcal Y'(\mathcal G_2), \delta_1')$ are equivariantly homothetic.
Either: 1) $\mathcal L_{\iota(1)}^+[\psi] \subset \mathcal L_1^+[\psi]$ and $\lambda_1 < \lambda_{\iota(1)}$ since $\mathcal L_{\iota(1)}^+[\psi]$ is dominating;
or 2) the lower-support $\operatorname{\underline{supp}}[\psi_1; {\mathcal Z_1}]$ of~$\mathcal L_{\mathcal Z_1}^+[\psi_1]$ is elliptic in~$\mathcal Y^*$ and~$\mathcal Y'$ by Theorem~\ref{thm-dominatinglimitexist}(\ref{cond-dominating-loxodromics}).
The $\mathcal F$-forests $(\mathcal Y^*, \delta_{\iota(1)})$ and $(\mathcal Y', \delta_1')$ are equivariantly homothetic by Lemmas~\ref{lem-convergence4} and~\ref{lem-convergence3} respectively, and we are done.
%
\end{proof}

\noindent Thus the factored $\mathcal F$-forest $(\mathcal Y, \Sigma_{j=1}^{k} \delta_{\iota(j)})$ is \emph{the} \underline{complete dominating forest} for~$[\psi]$.

\newpage
\section{Convergence criteria}\label{SecConverge}

This chapter adapts then extends Section~7 of Levitt--Lustig's paper~\cite{LL03};
they, in turn, gave complete details for the proof sketched by Bestvina--Feighn--Handel in~\cite[Lemma~3.4]{BFH97}.

\subsection{Proof of Lemma~\ref{lem-convergence}}\label{SubsecConverge}

Fix an automorphism $\psi \colon \mathcal F \to \mathcal F$ with an expanding irreducible train track $\tau \colon \mathcal T \to \mathcal T$.
Let $\lambda \defeq \lambda[\tau]$, $(\mathcal Y_\tau, d_\infty)$ be the limit forest for~$[\tau]$, $\pi \colon (\mathcal T, d_\tau) \to (\mathcal Y_\tau, d_\infty)$ the constructed equivariant metric PL-map, $\mathcal L^+[\tau] \subset \mathbb R(\mathcal T)$ the stable lamination for~$[\tau]$, and~$k \ge 1$ the number of components of~$\mathcal L^+[\tau]$.
Suppose  $f \colon (\mathcal T, d_\tau) \to (\mathcal Y, \delta)$ is an equivariant PL-map and~$\mathcal L^+[\tau]$ is in the canonically embedded subspace $\mathbb R(\mathcal Y, \delta) \subset \mathbb R(\mathcal T)$.

\begin{claim}[{cf.~\cite[Lemma~7.1]{LL03}}] \label{claim-leafconst}
There is a sequence~$c(f)$ of positive constants $c_i$ indexed by the components $\Lambda_i^+ \subset \mathcal L^+[\tau]$ such that 
\[ \lim_{m \to \infty} \lambda^{-mk} \delta(f(\tau^{mk}(p)), f(\tau^{mk}(q))) = c_i \, d_\infty(\pi(p),\pi(q))\]
for any leaf segment~$[p,q]$ of~$\Lambda_i^+$.
\end{claim}

\noindent Any two equivariant PL-maps $f, g \colon (\mathcal T, d_\tau) \to (\mathcal Y, \delta)$ are a bounded $\delta$-distance apart, and $c(f) = c(g)$.
So we can define $c(\mathcal Y, \delta) \defeq c(f)$;
note that $c(\mathcal Y, s \, \delta) = s \, c(\mathcal Y,  \delta)$ for $s > 0$.
Without loss of generality, rescale the metric~$\delta$ so that~$f$ is an equivariant metric PL-map.

\begin{proof}
Let $\nu^R \defeq \nu^R[\tau]$ (resp.~$\nu^L \defeq \nu^L[\tau]$) be the unique positive right (resp.~left) eigenvector for the irreducible transition matrix $A \defeq A[\tau]$ whose sum of entries is~1 (resp.~dot product $\langle \nu^L, \nu^R \rangle = k$).
Suppose $[p,q]$ is a leaf segment (of a component $\Lambda_i^+ \subset \mathcal L^+[\tau]$) with endpoints at vertices of~$\mathcal T$ and let $v \defeq v[p,q]$ be the vector counting the occurrences of~$[e]$ in $[p,q]$: $[e]$ is an $\mathcal F$-orbit of edges in~$\mathcal T$; the entries of $v = (v_e)$ are indexed by the $\mathcal F$-orbits~$[e]$; and~$v_e$ is the number of translates of~$e$ in~$[p,q]$.
The train track property gives us $v^{(m)} \defeq v[\tau^m(p), \tau^m(q)] = A^m v$.
Then, as $[p,q]$ is a leaf segment, the positive entries of~$v^{(mk)}$ are indexed in the same block $\mathcal B_i = \mathcal B(\Lambda_i^+)$ for all $m \ge 0$.
By Perron's theorem, if~$[e]$ is in the block~$\mathcal B_i$, then \[ \lim_{m \to \infty} \frac{v_e^{(mk)}}{\lambda^{mk} \langle \nu^L, v \rangle} = \nu_e^R.
\]

For small $\epsilon > 0$, fix $m_\epsilon \gg 1$ such that $\delta_e(m_\epsilon) \defeq \delta(f(\tau^{m_\epsilon k}(p_e)), f(\tau^{m_\epsilon k}(q_e))) > \epsilon^{-1} C[f]$ for every edge $e = [p_e, q_e]$ in~$\mathcal T$ --- we need the assumption $\mathcal L^+[\tau] \subset \mathbb R(\mathcal Y, \delta)$ for this.
The interval $[\tau^{(m_\epsilon+m)k}(p), \tau^{(m_\epsilon+m)k}(q)]$ is a union of $v^{(mk)}_e$-many translates of~$\tau^{m_\epsilon k}(e)$, as~$[e]$ ranges over all the orbits of edges in~$\mathcal T$.
In~$\mathcal Y$, we get
\[ 
\sum_{[e] \subset \mathcal T} v^{(mk)}_e (\delta_e(m_\epsilon)- 2C[f]) \le \delta(f(\tau^{(m_\epsilon+m)k}(p)), f(\tau^{(m_\epsilon+m)k}(q))) \le \sum_{[e] \subset \mathcal T} v^{(mk)}_e \delta_e(m_\epsilon).
\]
Divide by $\lambda^{(m_\epsilon + m)k} d_\infty(\pi(p), \pi(q))  = \lambda^{(m_\epsilon + m)k} \langle \nu^L, v \rangle$, and let $m \to \infty$:
\[\begin{aligned}
(1 - 2 \epsilon) \sum_{[e] \in \mathcal B_i} \nu^R_{e} \, & \frac{\delta_e(m_\epsilon)}{\lambda^{m_\epsilon k}} \le \liminf_{m \to \infty} \frac{\delta(f(\tau^{m k}(p)), f(\tau^{m k}(q)))}{\lambda^{mk} d_\infty(\pi(p), \pi(q))} \\
&\le \limsup_{m \to \infty} \frac{\delta(f(\tau^{m k}(p)), f(\tau^{m k}(q)))}{\lambda^{mk} d_\infty(\pi(p), \pi(q))} \le \sum_{[e] \in \mathcal B_i} \nu^R_{e} \, \frac{\delta_e(m_\epsilon)}{\lambda^{m_\epsilon k}}
\end{aligned}\]
Since~$f$ is a metric map, we have $\lambda^{-m_\epsilon k}\delta_e(m_\epsilon) \le \nu^L_e$.
So the $\liminf$ and $\limsup$ above are real, equal, and depend only on the block~$\mathcal B_i$ for~$\Lambda_i^+$.

If~$\epsilon$ is small, then $\epsilon^{-1}C[f] > 2C[f] + L$ for some $L > 0$;
by bounded cancellation, 
\[ \begin{aligned}
c_i \defeq \lim_{m \to \infty} \frac{\delta(f(\tau^{m k}(p)), f(\tau^{m k}(q)))}{\lambda^{m k} d_\infty(\pi(p), \pi(q))} &\ge  \lim_{m \to \infty} \frac{\| v^{(mk)} \|_1 L}{\lambda^{(m_\epsilon+m)k} \langle \nu^L, v \rangle} \ge \frac{\nu_e^R L}{\lambda^{m_\epsilon k}} > 0,
\end{aligned} \]
where $\| v^{(m)} \|_1$ is the sum of the entries in~$v^{(m)}$ and~$[e]$ is in the same block as $[p,q]$.

\smallskip
We now relax the restriction that~$[p,q]$ is an edge-path, i.e.~$p, q$ need not be vertices.
For $m \ge 0$, let $[\bar p_m, \bar q_m]$ be the shortest edge-path containing $[\tau^{mk}(p), \tau^{mk}(q)]$; 
for $m, m' \ge 0$,
\[\begin{aligned}
\frac{\delta(f(\tau^{mk}(\bar p_{m'})), f(\tau^{mk}(\bar q_{m'}))) - \lambda^{mk}2}{\lambda^{mk} (d_\infty(\pi(\bar p_{m'}), \pi(\bar q_{m'})) + 2)} &\le \frac{\delta(f(\tau^{(m+m')k}(p)), f(\tau^{(m+m')k}(q)))}{\lambda^{(m+m')k} d_\infty(\pi(p), \pi(q))} \\
&\le 
\frac{\delta(f(\tau^{mk}(\bar p_{m'})), f(\tau^{mk}(\bar q_{m'}))) + \lambda^{mk}2}{\lambda^{mk} (d_\infty(\pi(\bar p_{m'}), \pi(\bar q_{m'})) - 2)}.
\end{aligned}\]
Both upper and lower bounds converge to~$c_i$ as $m', m \to \infty$: $[\bar p_{m'}, \bar q_{m'}]$ is a leaf segment with endpoints at vertices of~$\mathcal T$, so 
\begin{align*}
 \lim_{m' \to \infty} \lim_{m \to \infty} & \frac{\delta(f(\tau^{mk}(\bar p_{m'})), f(\tau^{mk}(\bar q_{m'}))) \mp \lambda^{mk}2}{ \lambda^{mk}( d_\infty(\pi(\bar p_{m'}), \pi(\bar q_{m'})) \pm 2)} \\
 = & \lim_{m' \to \infty} \frac{ c_i \, d_\infty(\pi(\bar p_{m'}), \pi(\bar q_{m'})) \mp 2}{d_\infty(\pi(\bar p_{m'}), \pi(\bar q_{m'})) \pm 2} = c_i. \qedhere 
\end{align*}
\end{proof}

The next step is extending the claim to all intervals $[p,q] \subset \mathcal T$.
Set $(c_i)_{i=1}^k \defeq c(\mathcal Y, \delta)$ and let $d_\infty = \oplus_{i=1}^k d_\infty^{(i)}$ be the factorization indexed by the components $\Lambda_i^+ \subset \mathcal L^+[\tau]$.
For convenience, replace~$\psi$ with its iterate~$\psi^k$,~$\tau$ with~$\tau^k$, and~$\lambda$ with~$\lambda^k$.

\begin{claim}[{cf.~\cite[Lemma~7.2]{LL03}}] \label{claim-uniformconst}
For any $p_1, p_2 \in \mathcal T$,
\[ \lim_{m \to \infty} \lambda^{-m} \delta(f(\tau^{m}(p_1)), f(\tau^{m}(p_2))) = \sum_{i=1}^k c_i \, d_\infty^{(i)}(\pi(p_1), \pi(p_2)). \]
\end{claim}

\begin{proof}
Let $[p_1, p_2]$ be an interval in~$\mathcal T$ and $N(p_1,p_2)$ the number of vertices in $(p_1, p_2)$.
Suppose $\pi(p_1) = \pi(p_2)$, i.e.~$d_\infty(\pi(p_1), \pi(p_2)) = 0$.
Since~$f$ is a metric map, we get
\[ 0 \le \lambda^{-m} \delta(f(\tau^m(p_1)), f(\tau^m(p_2))) \le \lambda^{-m} d_\tau(\tau^m(p_1), \tau^m(p_2)), \]
and the limit of the middle term (as $m \to \infty$) is~$0$.
So we may assume $d_\infty(\pi(p_1), \pi(p_2)) > 0$.
For a given $m' \ge 0$, let $[\tau^{m'}(p_1), \tau^{m'}(p_2)]$ be a concatenation of $N'+1$ leaf segments $[q_j, q_{j+1}]_{j=0}^{N'}$ (of $\Lambda_{i(j)}^+ \subset \mathcal L^+[\tau]$) for some nonegative $N' \le N(p_1, p_2)$ and $i(j) \in \{1, \ldots, k\}$, 
where $q_0 = \tau^{m'}(p_1)$ and $q_{N'+1} = \tau^{m'}(p_2)$.
Then, by Claim~\ref{claim-leafconst},
\[\begin{aligned} 
\limsup_{m \to \infty} &\frac{\delta(f(\tau^{m+m'}(p_1)), f(\tau^{m+m'}(p_2)))}{\lambda^m} \le \lim_{m \to \infty} \sum_{j=0}^{N'}\frac{\delta(f(\tau^m(q_j)),f(\tau^m(q_{j+1})))}{\lambda^m} \\
 & = \sum_{j=0}^{N'} c_{i(j)} d_\infty(\pi(q_j), \pi(q_{j+1}))  = \sum_{i=1}^k c_i d_\tau^{(i)}(\tau^{m'}(p_1), \tau^{m'}(p_2)),
\end{aligned}\]
where the last equality comes from $d_\infty(\pi(q_j), \pi(q_{j+1})) = d_\tau^{(i(j))}(q_j, q_{j+1})$ since $[q_j, q_{j+1}]$ is a leaf segment.
Divide by~$\lambda^{m'}$, let $m' \to \infty$, and invoke the definition of~$d_\infty^{(i)}$ to get
\[\begin{aligned} 
\limsup_{m + m' \to \infty} &\frac{\delta(f(\tau^{m+m'}(p_1)), f(\tau^{m+m'}(p_2)))}{\lambda^{m+m'}} \le \sum_{i=1}^k c_i d_\infty^{(i)}(\pi(p_1), \pi(p_2)).
\end{aligned}\]
Using bounded cancellation, we get a lower bound:
\[\delta(f(\tau^{m+m'}(p_1)), f(\tau^{m+m'}(p_2))) \ge  \sum_{j=0}^{N'}\delta(f(\tau^m(q_j)),f(\tau^m(q_{j+1})))  - 2N' C[f], \]
which, after dividing by $\lambda^{m+m'}$ and letting $m \to \infty$ then $m' \to \infty$, leads to 
\[ \liminf_{m+m' \to \infty} \frac{\delta(f(\tau^{m+m'}(p_1)), f(\tau^{m+m'}(p_2)))}{\lambda^{m+m'}} \ge \sum_{i=1}^k c_i d_\infty^{(i)}(\pi(p_1), \pi(p_2)). \qedhere \] 
\end{proof}

Like in our construction of limit forests (Section~\ref{SubsecLimittrees}), let~$\delta_m^*$ be the pullback of~$\lambda^{-m}\delta$ via $f \circ \tau^m$ for $m \ge 0$.
Then~$\delta_m^*$ is an $\mathcal F$-invariant pseudometric on~$\mathcal T$ whose associated metric space is equivariantly isometric to $(\mathcal Y\psi^m, \lambda^{-m}\delta)$.
By Claim~\ref{claim-uniformconst}, the (pointwise) limit $\underset{m \to \infty} \lim \delta_m^*$ is the pullback of $\oplus_{i=1}^k c_i \, d_\infty^{(i)}$ via~$\pi$.
In other words, the sequence $(\mathcal Y\psi^m, \lambda^{-m} \delta)_{m \ge 0}$ converges to $(\mathcal Y_\tau, \oplus_{i=1}^k c_i \, d_\infty^{(i)})$ and we are done:

\begin{lem}[{cf.~\cite[Lemma~3.4]{BFH97}}]\label{lem-convergence}
Let~$\psi\colon \mathcal F \to \mathcal F$ be an automorphism,~$\tau \colon \mathcal T \to \mathcal T$ an expanding irreducible train track for~$\psi$, $(\mathcal Y_\tau, d_\infty)$ the limit forest for~$[\tau]$, and $\lambda \defeq \lambda[\tau]$.

If $(\mathcal T, d_\tau) \to (\mathcal Y, \delta)$ is an equivariant PL-map and the $k$-component lamination~$\mathcal L^+[\tau]$ is in $\mathbb R(\mathcal Y, \delta) \subset \mathbb R(\mathcal T)$, then the sequence $(\mathcal Y \psi^{mk}, \lambda^{-mk} \delta)_{m \ge 0}$ converges to $(\mathcal Y_\tau, \oplus_{i=1}^k c_i \, d_\infty^{(i)})$, where $d_\infty = \oplus_{i=1}^k \, d_\infty^{(i)}$ and $c_i > 0$. \qed
\end{lem}

\subsection{Proof of Lemma~\ref{lem-convergence2}}\label{SubsecConverge2}

Fix an automorphism $\psi\colon \mathcal F \to \mathcal F$ with an invariant proper free factor system~$\mathcal Z'$ and a descending sequence of irreducible train tracks $\left(\tau_i \colon \mathcal T_i \to \mathcal T_i\right)_{i=1}^n$ rel.~$\mathcal Z'$ with $\lambda \defeq \lambda[\tau_n] > 1$.
Let $\mathcal L_{\mathcal Z}^+[\psi] \subset \mathbb R(\mathcal F, \mathcal Z)$ be the $k$-component stable laminations for~$[\psi]$ rel.~$\mathcal Z \defeq \mathcal F[\mathcal T^\circ]$,~$\mathcal T^\circ$ an equivariant blow-up of the free splittings~$(\mathcal T_i)_{i=1}^n$, $\tau^\circ \colon \mathcal T^\circ \to \mathcal T^\circ$ a topological representative for~$[\psi]$ induced by $[\tau_i]_{i=1}^n$,~$d^\circ$ an $\mathcal F$-invariant convex metric on~$\mathcal T^\circ$ that extends~$d_n$ on~$\mathcal T_n$ such that~$\tau^\circ$ is $\lambda$-Lipschitz on $(\mathcal T^\circ, d^\circ)$, and $\pi^\circ \colon (\mathcal T^\circ, d^\circ) \to (\mathcal Y, \delta)$ the equivariant metric map to a limit forest constructed using $\tau^\circ$-iteration.
We denote by~$d_n$ again the $\mathcal F$-invariant convex pseudometric on~$\mathcal T^\circ$ that extends~$d_n$ on~$\mathcal T_n$.
Recall that the components $\Lambda_j^+ \subset \mathcal L_{\mathcal Z}^+[\psi]$ index the factorizations $d_n = \oplus_{j=1}^k d_n^{(j)}$ and $\delta = \oplus_{j=1}^k \delta_j$.
For convenience, set $\mathcal F_1 \defeq \mathcal F$ and $\mathcal F_{i+1} \defeq \mathcal F[\mathcal T_i]$, then replace~$\psi$ with~$\psi^k$,~$\tau^\circ$ with~$\tau^{\circ k}$, and~$\lambda$ with~$\lambda^k$.

Suppose $(\mathcal Y', \delta')$ is a minimal $\mathcal F$-forest with trivial arc stabilizers,~$\mathcal Z$ is $\mathcal Y'$-elliptic, and~$\mathcal L_{\mathcal Z}^+[\psi]$ is in $\mathbb R(\mathcal Y', \delta') \subset \mathbb R(\mathcal F, \mathcal Z)$. 
Let $(\mathcal Y_n', \delta')$ be the characteristic subforest of $(\mathcal Y', \delta')$ for~$\mathcal F_n$
and $f_n \colon (\mathcal T_n, d_n) \to (\mathcal Y_n', \delta')$ an equivariant PL-map.
Extend~$f_n$ to an equivariant PL-map
$f \colon (\mathcal T^\circ, d^\circ) \to (\mathcal Y', \delta')$. 
By Claim~\ref{claim-leafconst}, we can set $(c_j)_{j=1}^k \defeq c(\mathcal Y_n', \delta') > 0$.

\begin{claim} \label{claim-uniformconst2}
For any $p_1, p_2 \in \mathcal T^\circ$,
\[ \lim_{m \to \infty} \lambda^{-m} \delta'(f(\tau^{\circ m}(p_1)), f(\tau^{\circ m}(p_2))) = \sum_{j=1}^k c_j \, \delta_j(\pi^\circ(p_1), \pi^\circ(p_2)). \]
\end{claim}
\begin{proof}
Let $[p_1, p_2]$ be an interval in~$\mathcal T^\circ$ and assume $\delta(\pi^\circ(p_1), \pi^\circ(p_2)) > 0$ without loss of generality.
Given Claim~\ref{claim-uniformconst}, we may assume $n \ge 2$.
For $m' \ge 0$, the interval $[\tau^{\circ m'}(p_1), \tau^{\circ m'}(p_2)]$ is a concatenation of $\alpha(m')$ segments that are in~$\mathcal F \cdot \mathcal T_n$ or edges from $\mathcal T_i~(i > 1)$, where $\alpha(m')$ is bounded by a polynomial in~$m'$ of degree $ \le n-2$.
Set~$M$ to be the length of the longest edge from $\mathcal T_i~(i > 1)$ in $(\mathcal T^\circ, d^\circ)$.
For $m' \gg 0$, let $[q_{m',l}, q_{m',l+1}]_{l=0}^{N(m')}$ be the nondegenerate $(\mathcal F \cdot \mathcal T_n)$-segments.
As~$\tau^\circ$ and~$f$ are $\lambda$- and $L$-Lipschitz  respectively,
\[\begin{aligned} 
\delta'&(f(\tau^{\circ (m+m')}(p_1)), f(\tau^{\circ (m+m')}(p_2)) \\ 
&\le \sum_{l=0}^{N(m')} \delta'(f_n(\tau_n^{m}(q_{m',l})), f_n(\tau_n^{m}(q_{m',l+1}))) +  \alpha(m')\lambda^{m}LM.
\end{aligned}\]
Divide by $\lambda^{m+m'}$, let $m \to \infty$ then $m' \to \infty$, and invoke Claim~\ref{claim-uniformconst} and definition of~$\delta_j$:
\[\begin{aligned} 
&\limsup_{m + m' \to \infty} \frac{\delta'(f(\tau^{\circ (m+m')}(p_1)), f(\tau^{\circ{m+m'}}(p_2)))}{\lambda^{m+m'}} \\
&\le \lim_{m' \to \infty} \sum_{l = 0}^{N(m')} \sum_{j = 1}^k \frac{ c_j \delta_j(\pi^\circ(q_{m',l}), \pi^\circ(q_{m',l+1}))}{ \lambda^{m'} }  \\
&\le \lim_{m' \to \infty} \sum_{j = 1}^k \frac{ c_j d_n^{(j)}(\tau^{\circ m'}(p_1), \tau^{\circ m'}(p_2)) }{ \lambda^{m'} } 
= \sum_{j=1}^k c_j \delta_j(\pi^\circ(p_1), \pi^\circ(p_2)),
\end{aligned}\]
using the fact~$\pi^\circ$ is a metric map.
The intervals $[\pi^\circ(q_{m',l}),\pi^\circ(q_{m',l+1})]$ contribute at least \[ \lambda^{m'} \delta_j(\pi^\circ(p_1), \pi^\circ(p_2)) -\alpha(m')\left( M + 2C[\pi^\circ] \right)\]
to the $\delta_j$-length of $[\pi^\circ(\tau^{\circ m'}(p_1)), \pi^\circ(\tau^{\circ m'}(p_2))]$.
As before, bounded cancellation gives us:
\[ \begin{aligned}
 \delta'&(f(\tau^{\circ (m+m')}(p_1)), f(\tau^{\circ(m+m')}(p_2)))\\
&\ge \sum_{l=0}^{N(m')}\delta'(f_n(\tau_n^{m}(q_{m',l})), f_n(\tau_n^{m}(q_{m',l+1}))) - 2\, \alpha(m') C[f].
\end{aligned}\]
Divide by $\lambda^{m+m'}$ and let $m \to \infty$ then $m' \to \infty$ yields:
\[\begin{aligned}
&\liminf_{m + m' \to \infty} \frac{\delta'(f(\tau^{\circ (m+m')}(p_1)), f(\tau^{\circ{m+m'}}(p_2)))}{\lambda^{m+m'}} \\
&\ge \lim_{m' \to \infty} \sum_{j = 1}^k c_j \sum_{l = 0}^{N(m')}  \frac{ \delta_j(\pi^\circ(q_{m',l}), \pi^\circ(q_{m',l+1}))}{ \lambda^{m'} }
\ge \sum_{j=1}^k c_j \delta_j(\pi^\circ(p_1), \pi^\circ(p_2)),
\end{aligned}\]
where the last inequality comes from the contribution inequality above.
\end{proof}

The rest of the argument is the same as in the previous section.
Let~$\delta_m^*$ be pullback of $\lambda^{-m} \delta'$ via $f \circ \tau^{\circ m}$ for $m \ge 0$.
By Claim~\ref{claim-uniformconst2}, the limit $\underset{m \to \infty} \lim \delta_m^*$ is the pullback of $\oplus_{j=1}^k c_j \, \delta_j$ via~$\pi^\circ$ and we are done:

\begin{lem}\label{lem-convergence2}
Let~$\psi\colon \mathcal F \to \mathcal F$ be an automorphism, $\mathcal Z'$ a $[\psi]$-invariant proper free factor system, $\left(\tau_i \colon \mathcal T_i \to \mathcal T_i\right)_{i=1}^n$ a descending sequence of irreducible train tracks for~$[\psi]$ rel.~$\mathcal Z'$ with $\lambda \defeq \lambda[\tau_n] > 1$, $(\mathcal Y, \delta)$ the limit forest for~$[\tau_i]_{i=1}^n$, $(\mathcal Y', \delta')$ a minimal $\mathcal F$-forest with trivial arc stabilizers, and $\mathcal Z \defeq \mathcal F[\mathcal T_n]$.

If~$\mathcal Z$ is $\mathcal Y'$-elliptic and the $k$-component lamination~$\mathcal L_{\mathcal Z}^+[\psi]$ is in $\mathbb R(\mathcal Y', \delta') \subset \mathbb R(\mathcal F, \mathcal Z)$, then the limit of $(\mathcal Y' \psi^{mk}, \lambda^{-mk} \delta')_{m \ge 0}$ is $(\mathcal Y, \oplus_{j=1}^k c_j \, \delta_j)$, where  $\delta = \oplus_{j=1}^k \, \delta_j$ and $c_j > 0$. \qed
\end{lem}

\subsection{Sketch of Lemma~\ref{lem-convergence3}}\label{SubsecConverge3}

Fix an automorphism $\psi \colon \mathcal F \to \mathcal F$.
Let $(\tau_i \colon \mathcal T_i \to \mathcal T_i)_{i=1}^n$ be a descending sequence of irreducible train tracks for~$[\psi]$, $\mathcal Z \defeq \mathcal F[\mathcal T_n]$,  $\mathcal L_{\mathcal Z}^+[\psi]$ the stable lamination for~$[\psi]$ in~$\mathbb R(\mathcal F, \mathcal Z)$, $(\mathcal Y_1, \delta_1)$ the limit forest for~$[\psi]$ rel.~$\mathcal Z$, $\mathcal G \defeq \mathcal G[\mathcal Y_1]$, $[\psi_{\mathcal G}]$ the restriction of~$[\psi]$ to~$\mathcal G$, $(\mathcal Y_{\mathcal G}, \delta)$ a minimal $\mathcal G$-forest with trivial arc stabilizers, and $h_{\mathcal G} \colon (\mathcal Y_{\mathcal G}, \delta) \to (\mathcal Y_{\mathcal G}, \delta)$ a $\psi_{\mathcal G}$-equivariant $\lambda$-homothety.
Construct the equivariant psuedoforest blow-up $(\mathcal Y_1^*, (\delta_1, \delta))$ of $(\mathcal Y_1, \delta_1)$ rel. $(\mathcal Y_{\mathcal G}, \delta)$ and expanding homotheties representing $[\psi]$ and $[\psi_{\mathcal G}]$.
For this section, we will assume assume $\mathcal L_{\mathcal Z}^+[\psi]$ and~$\delta$ are independent:
the pseudoleaf segments for~$\mathcal L_{\mathcal Z}^+[\psi]$ in~$\mathcal Y_1^*$ have 0 $\delta$-diameter intersections with~$\mathcal Y_{\mathcal G}$.
Set $\mathcal F_n \defeq \mathcal F[\mathcal T_{n-1}]$ and $[\psi_n]$ to be the restriction of~$[\psi]$ to~$\mathcal F_n$;
the characteristic convex subset $\mathcal Y_1^*(\mathcal F_n) \subset \mathcal Y_1^*$ has a graph of actions decomposition with vertex forests $\widehat{\mathcal Y}_{\mathcal G}$ and the overlapping classes for~$\mathcal L_{\mathcal Z}^+[\psi]$.

Let the minimal simplicial $\mathcal F_n$-forest $\mathcal S$ be the skeleton for the graph of actions for~$\mathcal L_{\mathcal Z}^+[\psi]$ and~$\delta$.
By construction, there is a $\psi_n$-equivariant simplicial automorphism $\sigma \colon \mathcal S \to \mathcal S$.
The lower-support~$\operatorname{\underline{supp}}[\psi; {\mathcal Z}]$ of~$\mathcal L_{\mathcal Z}^+[\psi]$ is given by stabilizers of vertices in~$\mathcal S$ corresponding to overlapping classes.
Construct the equivariant blow-up~$\mathcal T^\diamond$ of $(\mathcal T_i)$, $\mathcal S$, and $\mathcal Y_{\mathcal G}$;
then extend the metric~$\delta$ to an $\mathcal F$-invariant convex metric~$d \oplus \delta$ on~$\mathcal T^\diamond$ so that the $\psi$-equivariant map $\tau_c^\diamond \colon (\mathcal T^\diamond, (c \, d) \oplus \delta) \to (\mathcal T^\diamond, (c \, d) \oplus \delta)$ induced by $[\tau_i]_{i=1}^{n-1}$, $\sigma$ and linearly extending~$h_{\mathcal G}$ is $\lambda$-Lipschitz for any parameter $c \gg 1$.
Let $d_c^\diamond \defeq (c \, d) \oplus \delta$;
for $c \gg 1$, construct using $\tau_c^\diamond$-iteration an equivariant metric surjection $\pi_c^\diamond \colon (\mathcal T^\diamond, d_c^\diamond) \to (\mathcal X, \delta)$ that extends the identification of $(\mathcal Y_{\mathcal G}, \delta)$ and semiconjugates~$\tau_c^\diamond$ to a $\psi$-equivariant $\lambda$-homothety on $(\mathcal X, \delta)$.

Suppose $(\mathcal Y', \delta')$ is a minimal $\mathcal F$-forest with trivial arc stabilizers and whose characteristic subforest for~$\mathcal G$ is equivariantly isometric to $(\mathcal Y_{\mathcal G}, \delta)$.
So if we also assume $\operatorname{\underline{supp}}[\psi; {\mathcal Z}]$ is $\mathcal Y'$-elliptic, then there is an equivariant map $f_c \colon (\mathcal T^\diamond, d_c^\diamond) \to (\mathcal Y', \delta')$ that linearly extends the identification of $(\mathcal Y_{\mathcal G}, \delta)$;
this is necessarily a Lipschitz map.
Pick any free splitting~$\mathcal T$ of~$\mathcal F$ with trivial $\mathcal F[\mathcal T]$.
Then any equivariant PL-map $\mathcal T \to \mathcal T^\diamond$ is surjective (by minimality) and composes with~$f_c$ to give (up to an equivariant homotopy rel.~the vertices) an equivariant PL-map with a cancellation constant.
So~$f_c$ must have a cancellation constant.
The proof of the next claim is a variation of Claim~\ref{claim-uniformconst2}'s proof:

\begin{claim} \label{claim-uniformconst3}
For any $p_1, p_2 \in \mathcal T^\diamond$,
\[ \lim_{m \to \infty} \lambda^{-m} \delta'(f_c(\tau_c^{\diamond m}(p_1)), f_c(\tau_c^{\diamond m}(p_2))) = \delta(\pi_c^\diamond(p_1), \pi_c^\diamond(p_2)). \]
\end{claim}

\begin{proof}[Sketch of proof]
For $m' \ge 0$, the interval $[\tau_c^{\circ m'}(p_1), \tau_c^{\circ m'}(p_2)]$ is a concatenation of $\alpha(m')$ segments that are in the orbit of~$\mathcal Y_{\mathcal G}$ or edges from $\mathcal T_i~(i \ge 1)$, where $\alpha(m')$ is bounded by a polynomial in~$m'$ of degree $\le n-1$.
With an almost identical argument, invoke the definition of~$\pi_c^\circ$ to conclude
\[ \lim_{m + m' \to \infty} \frac{\delta'(f_c(\tau_c^{\circ (m+m')}(p_1)), f_c(\tau_c^{\circ{m+m'}}(p_2)))}{\lambda^{m+m'}} = \delta(\pi_c^\circ(p_1), \pi_c^\circ(p_2)). \]
The set-up is simpler as~$\tau_c^\circ$ (resp.~$f_c$) is a $\lambda$-homothety (resp. isometry) on $(\mathcal Y_{\mathcal Z}, \delta)$.
\end{proof}

As in the previous section, we have proven the following:

\begin{lem}\label{lem-convergence3}
Let $\psi\colon \mathcal F \to \mathcal F$ be an automorphism, $\left(\tau_i \colon \mathcal T_i \to \mathcal T_i\right)_{i=1}^n$ a descending sequence of irreducible train tracks for~$[\psi]$, $\mathcal Z \defeq \mathcal F[\mathcal T_n]$, $\mathcal G$ the nontrivial point stabilizer system for the limit forest for~$[\psi]$ rel.~$\mathcal Z$, $[\psi_{\mathcal G}]$ the $[\psi]$-restriction to~$\mathcal G$, $(\mathcal Y_{\mathcal G}, \delta)$ a minimal $\mathcal G$-forest with trivial arc stabilizers such that $\mathcal L_{\mathcal Z}^+[\psi]$ and~$\delta$ are independent, $h_{\mathcal G} \colon (\mathcal Y_{\mathcal G}, \delta) \to (\mathcal Y_{\mathcal G}, \delta)$ a $\psi_{\mathcal G}$-equivariant $\lambda$-homothety, $\mathcal S$ a minimal simplicial $\mathcal F[\mathcal T_{n-1}]$-forest that is the skeleton for the graph of actions for~$\mathcal L_{\mathcal Z}^+[\psi]$ and~$\delta$, $\sigma \colon \mathcal S \to \mathcal S$ the corresponding simplicial automorphism, and $(\mathcal X, \delta)$ the limit forest for~$[\tau_i]_{i=1}^{n-1}$,~$\sigma$, and~$h_{\mathcal G}$.

If $(\mathcal Y', \delta')$ is a minimal $\mathcal F$-forest with trivial arc stabilizers, the characteristic subforest of $(\mathcal Y', \delta')$ for~$\mathcal G$ is equivariantly isometric to $(\mathcal Y_{\mathcal G}, \delta)$, and the lower-support~$\operatorname{\underline{supp}}[\psi; {\mathcal Z}]$ of $\mathcal L_{\mathcal Z}^+[\psi]$  is $\mathcal Y'$-elliptic, then the limit of $(\mathcal Y' \psi^{m}, \lambda^{-m} \delta')_{m \ge 0}$ is $(\mathcal X, \delta)$.  \qed
\end{lem}

\subsection{Sketch of Lemma~\ref{lem-convergence4}}\label{SubsecConverge4}

Fix an automorphism $\psi \colon \mathcal F \to \mathcal F$ with an invariant proper free factor system~$\mathcal Z$ and a minimal $\mathcal Z$-forest $(\mathcal Y_{\mathcal Z}, \delta)$ with trivial arc stabilizers.
Let $(\tau_i \colon \mathcal T_i \to \mathcal T_i)_{i=1}^n$ be a descending sequence of irreducible train tracks for~$[\psi]$ with $\mathcal F[\mathcal T_n] = \mathcal Z$, $d_n$ the eigenmetric on~$\mathcal T_n$ for~$[\tau_n]$, and $h_{\mathcal Z} \colon (\mathcal Y_{\mathcal Z}, \delta) \to (\mathcal Y_{\mathcal Z}, \delta)$ a $\psi_{\mathcal Z}$-equivariant $\lambda$-homothety, where $\lambda > \lambda[\tau_n]$ and $[\psi_{\mathcal Z}]$ is the $[\psi]$-restriction to~$\mathcal Z$.
Set $\mathcal F_1 \defeq \mathcal F$ and $\mathcal F_{i+1} \defeq \mathcal F[\mathcal T_i]$.

Choose an arbitrary equivariant iterated blow-up~$\mathcal T^*$ of~$(\mathcal T_i)_{i=1}^n$ and let $\tau^* \colon \mathcal T^* \to \mathcal T^*$ be the $\psi$-equivariant topological representative induced by $(\tau_i)_{i=1}^n$.
Extend the metric~$d_n$ on~$\mathcal T_n$ to an $\mathcal F$-invariant convex metric~$d^*$ on~$\mathcal T^*$ so that $\tau^* \colon (\mathcal T^*, d^*) \to (\mathcal T^*, d^*)$ is $\lambda[\tau_n]$-Lipschitz.
Finally, choose an arbitrary equivariant metric blow-up $(\mathcal T^\circ, d^* \oplus \delta)$ of $(\mathcal T^*, d^*)$ rel.~$(\mathcal Y_{\mathcal Z}, \delta)$.
For a parameter $c > 0$, the topological representative~$\tau^*$ induces a $\psi$-equivariant map~$\tau_c^\circ$ on~$\mathcal T^\circ$ that linearly extends the $\lambda$-homothety~$h_{\mathcal Z}$ with respect to the metric $d_c^\circ \defeq  (c \, d^*) \oplus \delta$.
As $\lambda > \lambda[\tau_n]$, the map~$\tau_c^\circ$ is $\lambda$-Lipschitz with respect to~$d_c^\circ $ for $c \gg 1$.
Let $(\mathcal Y, \delta)$ be the limit forest for~$[\tau_c^\circ]$ and $\pi_c^\circ \colon (\mathcal T^\circ, d_c^\circ) \to (\mathcal Y, \delta)$ the equivariant metric surjection constructed through $\tau^\circ$-iteration.

Suppose $(\mathcal Y', \delta')$ is a minimal $\mathcal F$-forest with trivial arc stabilizers and whose characteristic subforest for $\mathcal Z$ is equivariantly isometric to $(\mathcal Y_{\mathcal Z}, \delta)$.
Let $f_c \colon (\mathcal T^\circ, d_c^\circ) \to (\mathcal Y', \delta')$ be an equivariant map that linearly extends the identification of $(\mathcal Y_{\mathcal Z}, \delta)$.

\begin{claim} \label{claim-uniformconst4}
For any $p_1, p_2 \in \mathcal T^\circ$,
\[ \lim_{m \to \infty} \lambda^{-m} \delta'(f_c(\tau_c^{\circ m}(p_1)), f_c(\tau_c^{\circ m}(p_2))) = \delta(\pi_c^\circ(p_1), \pi_c^\circ(p_2)). \]
\end{claim}
\begin{proof}[Sketch of proof]
For $m' \ge 0$, the interval $[\tau_c^{\circ m'}(p_1), \tau_c^{\circ m'}(p_2)]$ is a concatenation of $\beta(m')$ segments that are in the orbit of~$\mathcal Y_{\mathcal Z}$ or edges from $\mathcal T_i~(i \ge 1)$, where $\beta(m')$ has exponential growth rate $\lambda[\tau_n] < \lambda$.
Proceed just as in the proof of Claim~\ref{claim-uniformconst3}.
%
\end{proof}

Altogether, we have proven the following:

\begin{lem}\label{lem-convergence4}
Let $\psi\colon \mathcal F \to \mathcal F$ be an automorphism,~$\mathcal Z$ a $[\psi]$-invariant proper free factor system, $(\mathcal Y_{\mathcal Z}, \delta)$ a minimal $\mathcal Z$-forest with trivial arc stabilizers, $\left(\tau_i \colon \mathcal T_i \to \mathcal T_i\right)_{i=1}^n$ a descending sequence of irreducible train tracks for~$[\psi]$ with $\mathcal F[\mathcal T_n] = \mathcal Z$, $h_{\mathcal Z} \colon (\mathcal Y_{\mathcal Z}, \delta) \to (\mathcal Y_{\mathcal Z}, \delta)$ a $\psi_{\mathcal Z}$-equivariant $\lambda$-homothety, and $(\mathcal Y, \delta)$  the limit forest for~$[\tau_i]_{i=1}^n$ and~$h_{\mathcal Z}$, where $\lambda > \lambda[\tau_n]$ and $[\psi_{\mathcal Z}]$ is the $[\psi]$-restriction to~$\mathcal Z$.

If $(\mathcal Y', \delta')$ is a minimal $\mathcal F$-forest with trivial arc stabilizers and the characteristic subforest of $(\mathcal Y', \delta')$ for~$\mathcal Z$ is equivariantly isometric to $(\mathcal Y_{\mathcal Z}, \delta)$, then the limit of $(\mathcal Y' \psi^{m}, \lambda^{-m} \delta')_{m \ge 0}$ is $(\mathcal Y, \delta)$. \qed
\end{lem}


\section{Expanding forests}\label{SecExpForests}
We finally characterize the \underline{expanding forests} for an automorphism $\psi \colon \mathcal F \to \mathcal F$, i.e.~minimal very small $\mathcal F$-forests that admit $\psi$-equivariant expanding homotheties.
By the last paragraph in the proof of Corollary~\ref{cor-unique2}, expanding forests have trivial arc stabilizers.
We start with a criterion of nonconvergence that complements Lemmas~\ref{lem-convergence3} and~\ref{lem-convergence4}.

\subsection{Nonconvergence criterion}\label{SubsecNonConverge}

Fix an automorphism $\psi \colon \mathcal F \to \mathcal F$ with an invariant proper free factor system~$\mathcal Z$ and a minimal $\mathcal Z$-forest $(\mathcal Y_{\mathcal Z}, \delta)$ with trivial arc stabilizers.
Let $\tau \colon \mathcal T \to \mathcal T$ be an expanding irreducible train track for~$[\psi]$ with $\mathcal F[\mathcal T] = \mathcal Z$, $d_\tau$ the eigenmetric on~$\mathcal T$ for~$[\tau]$,  $\mathcal L_{\mathcal Z}^+[\psi]$ the stable lamination for~$[\psi]$ in~$\mathbb R(\mathcal F, \mathcal Z)$, $(\mathcal Y_\tau, d_\infty)$ the limit forest for~$[\psi]$ rel.~$\mathcal Z$, and $h_{\mathcal Z} \colon (\mathcal Y_{\mathcal Z}, \delta) \to (\mathcal Y_{\mathcal Z}, \delta)$ a $\psi_{\mathcal Z}$-equivariant $\lambda$-homothety, where $1 < \lambda \le \lambda[\tau]$ and $[\psi_{\mathcal Z}]$ is the restriction of~$[\psi]$ to~$\mathcal Z$.

Set $\mathcal G \defeq \mathcal G[\mathcal Y_\tau]$, and denote the restriction of~$[\psi]$ to~$\mathcal G$ by  $[\psi_{\mathcal G}]$.
Since~$[\psi_{\mathcal G}]$ is polynomially growing rel.~$\mathcal Z$, we can equivariantly include $(\mathcal Y_{\mathcal Z}, \delta)$ in a minimal $\mathcal G$-forest $(\mathcal Y_{\mathcal G}, \delta)$ with trivial arc stabilizers and extend~$h_\mathcal Z$ to a $\psi_{\mathcal G}$-equivariant $\lambda$-homothety $h_{\mathcal G} \colon  (\mathcal Y_{\mathcal G}, \delta) \to  (\mathcal Y_{\mathcal G}, \delta)$. 
Construct the equivariant psuedoforest blow-up $(\mathcal Y^*, (d_\infty, \delta))$ of $(\mathcal Y_\tau, d_\infty)$ rel. $(\mathcal Y_{\mathcal G}, \delta)$ and the expanding homotheties representing $[\psi]$ and $[\psi_{\mathcal G}]$.
Finally, suppose $\mathcal L_{\mathcal Z}^+[\psi]$ and~$\delta$ are \emph{dependent}, i.e.~the pseudoleaf segments for~$\mathcal L_{\mathcal Z}^+[\psi]$ in~$\mathcal Y^*$ have some positive $\delta$-diameter intersections with~$\mathcal Y_{\mathcal G}$.
We are essentially in the case not covered by Lemmas~\ref{lem-convergence3} and~\ref{lem-convergence4}.

\smallskip
Choose an iterate $[\tau^{k}]$ that maps all $\mathcal F$-orbits of branches in~$\mathcal T$ to $[\tau^{k}]$-fixed orbits.
Pick a branch~$e^+$ in~$\mathcal T$;
suppose its basepoint $p \in \mathcal T$ is a vertex with a nontrivial stabilizer.
Without loss of generality, assume $\tau^{k}(e^+) = e^+$.
Use the contraction mapping theorem to decide how to equivariantly attach $\tau^{k}(e^+)$ to $\mathcal F \cdot \mathcal Y_{\mathcal Z}$;
then equivariantly attach~$e^+$ to the same point.
Now suppose the basepoint~$p$ has a trivial stabilizer but $\tau^{k}(p)$ has a nontrivial one.
Then there are finitely many directions~$e_1^+, \ldots, e_l^+$ at~$p$.
We have described how to attach their images~$\tau^{k}(e_1), \ldots, \tau^{k}(e_l)$ to the $\mathcal F$-orbit of~$\mathcal Y_{\mathcal Z}$;
let $C_p \subset \mathcal F \cdot \mathcal Y_{\mathcal Z}$ be the \emph{convex hull} of these attaching points.
Equivariantly replace~$p \in \mathcal T$ with a copy of~$(C_p, \lambda^{-k} \delta )$ and attach~$e_j^+$ to the copy of the attaching point for its $\tau^{k}$-image.
Finally, if~$\tau^{k}(p)$ has a trivial stabilizer, then there is nothing to do.
As~$[e^+]$ ranges over all $\mathcal F$-orbits of branches in~$\mathcal T$, this defines a {preferred} equivariant metric blow-up $(\mathcal T^\circ, d_\tau \oplus \delta)$ of $(\mathcal T, d_\tau)$ rel.~$(\mathcal Y_\mathcal Z, \delta)$.
The train track~$\tau$ induces a $\psi$-equivariant map $\tau^\circ \colon (\mathcal T^\circ, d_\tau \oplus \delta) \to (\mathcal T^\circ, d_\tau \oplus \delta)$ that linearly extends the homothety~$h_{\mathcal Z}$.
The preferred construction guarantees~$\tau^\circ$ is a train track in a sense: $\tau^{\circ m}$ is injective on the edges from~$\mathcal T$ for all $m \ge 1$.
Suppose $(\mathcal Y, \delta)$ is a minimal $\mathcal F$-forest with trivial arc stabilizers and whose characteristic subforest for $\mathcal Z$ is equivariantly isometric to $(\mathcal Y_{\mathcal Z}, \delta)$.

\begin{claim} \label{claim-blowup}
For some element~$x$ in~$\mathcal F$, $ \lambda^{-m} \| \psi^m(x) \|_{\delta} \to \infty $ as $m \to \infty$.
\end{claim}
\begin{proof}[Sketch of proof]
A long leaf segment in~$\mathcal T$ contains at least three (unoriented) translates $x_i \cdot e~(1\le i\le 3)$ of an edge~$e$.
So $x \defeq x_i^{-1} x_{i+1}$ is $\mathcal T$-loxodromic for some~$i \pmod 3$.
Choose a fundamental domain $[p,q]$ of~$x$ acting on its axis that is a leaf segment with endpoints at vertices.
Set $d^\circ \defeq d_\tau \oplus \delta$, and let $f \colon (\mathcal T^\circ, d^\circ) \to (\mathcal Y, \delta)$ an equivariant map that linearly extends the identification of $(\mathcal Y_{\mathcal Z}, \delta)$.

The assumption that~$\mathcal L_{\mathcal Z}^+[\psi]$ and~$\delta$ were dependent implies the $\tau^\circ$-image of some edge~$e$ from~$\mathcal T$ has a nondegenerate intersection with $\mathcal F \cdot \mathcal Y_{\mathcal Z}$.
Fix $m' \gg 1$ so that, for some $L > 0$, $\tau^{\circ m'}(e) \cap \mathcal F \cdot \mathcal Y_{\mathcal Z}$ has a component with $\delta$-length $\ge 2C[f] + L$ for all edges from~$\mathcal T$.
For $m \ge 0$, let~$\beta(m)$ be the number of edges from~$\mathcal T$ in $[\tau^{\circ m}(p), \tau^{\circ m}(q)]$;
note that~$\beta(m)$ grows exponentially in~$m$ with rate~$\lambda[\tau]$ --- the growth of~$[\psi]$ rel.~$\mathcal Z$.
By the train track property of~$\tau^\circ$ and bounded cancellation for~$f$, $\lambda^{-m}\|\psi^{m+m'}(x)\|_{\delta} \ge \sum_{i=0}^m \beta(i)\lambda^{-i} L$ tends to infinity as $m \to \infty$ since $\lambda \le \lambda[\tau]$.
\end{proof}

Thus there is no $\psi$-equivariant homothety of $(\mathcal Y, \delta)$:

\begin{lem}\label{lem-nonconvergence}
Let $\psi\colon \mathcal F \to \mathcal F$ be an automorphism, $\tau \colon \mathcal T \to \mathcal T$ an expanding irreducible train track for~$[\psi]$, $\mathcal Z \defeq \mathcal F[\mathcal T]$,  and $(\mathcal Y, \delta)$ an expanding forest for~$[\psi]$ with stretch factor~$\lambda$.
If $\mathcal L_{\mathcal Z}^+[\psi]$ and~$\delta$ are dependent, then $\lambda > \lambda[\tau]$. \qed
\end{lem}

\subsection{Expanding is dominating}\label{SubsecExpIsDom}

Fix an automorphism $\psi \colon \mathcal F \to \mathcal F$ and an expanding forest $(\mathcal Y, \delta)$ for~$[\psi]$. 
Our remaining goal is to generalize Corollary~\ref{cor-unique2}: 
$(\mathcal Y, \delta)$ must be some dominating forest for~$[\psi]$.

Let $\tau \colon \mathcal T \to \mathcal T$ be an expanding irreducible train track for~$[\psi]$, $\mathcal Z \defeq \mathcal F[\mathcal T]$, $\mathcal L_{\mathcal Z}^+[\psi]$ the stable lamination for~$[\psi]$ in~$\mathbb R(\mathcal F, \mathcal Z)$, $(\mathcal Y_\tau, d_\infty)$ the limit forest for~$[\psi]$ rel.~$\mathcal Z$, and $\mathcal G \defeq \mathcal G[\mathcal Y_\tau]$.

For induction, assume the characteristic subforest of $(\mathcal Y, \delta)$ for~$\mathcal G$ is equivariantly isometric to the dominating forest for the restriction~$[\psi_{\mathcal G}]$ (of~$[\psi]$ to~$\mathcal G$) with respect to some orbits~$\{ \mathcal A_i^{dom}[\psi_{\mathcal G}] \}_{i=1}^k$ with the same stretch factor $\lambda > 1$;
denote the subforest by $(\mathcal Y_{\mathcal G}, \delta)$.
Suppose~$\mathcal L_{\mathcal Z}^+[\psi]$ and~$\delta$ are dependent.
By Lemma~\ref{lem-nonconvergence}, $\lambda > \lambda[\tau]$ and each $\mathcal A_i^{dom}[\psi_{\mathcal G}] $ is actually a $\psi_*$-orbit~$\mathcal A_i^{dom}[\psi]$ of dominating attracting laminations for~$[\psi]$.
By Lemma~\ref{lem-convergence4}, $(\mathcal Y, \delta)$ is equivariantly isometric to the dominating forest for~$[\psi]$ with respect to~$\{ \mathcal A_i^{dom}[\psi] \}_{i=1}^k$.

\smallskip
We may now assume~$\mathcal L_{\mathcal Z}^+[\psi]$ and~$\delta$ are independent.
So $\mathcal A_i^{dom}[\psi_{\mathcal G}] $ is a $\psi_*$-orbit~$\mathcal A_i^{dom}[\psi]$ of dominating attracting laminations for~$[\psi]$.
Let $\mathcal A_0^{dom}[\psi]  \subset \mathbb R(\mathcal F)$ be the closure of~$\mathcal L_{\mathcal Z}^+[\psi]$ and $(\mathcal Y^*, d_\infty \oplus \delta)$ the unique equivariant metric blow-up of $(\mathcal Y_\tau, d_\infty)$ rel. $(\mathcal Y_{\mathcal G}, \delta)$ that admits a $\psi$-equivariant expanding dilation.
By construction, the blow-up is equivariantly isometric to the dominating forest for~$[\psi]$ with respect to $\{ \mathcal A_i^{dom}[\psi] \}_{i=0}^k$.
Recall that independence of~$\mathcal L_{\mathcal Z}^+[\psi]$ and~$\delta$ implies~$\mathcal Y^*$ is a graph of actions with vertex forests coming from $\widehat{\mathcal Y}_{\mathcal G}$ and overlapping classes for~$\mathcal L_{\mathcal Z}^+[\psi]$ --- these are $\mathcal G$- and $\operatorname{\underline{supp}}[\psi; \mathcal Z]$-forests respectively;
let~$\mathcal S$ be the skeleton for this graph of actions.

If the lower-support~$\operatorname{\underline{supp}}[\psi; {\mathcal Z}]$ is $\mathcal Y$-elliptic, then $(\mathcal Y, \delta)$ is equivariantly isometric to the associated $\mathcal F$-forest for~$\delta$ on~$\mathcal Y$ by Lemma~\ref{lem-convergence3};
in particular, $(\mathcal Y, \delta)$ is equivariantly isometric to the dominating forest for~$[\psi]$ with respect to $\{ \mathcal A_i^{dom}[\psi] \}_{i=1}^k$.
Otherwise,~$\operatorname{\underline{supp}}[\psi; {\mathcal Z}]$ is not $\mathcal Y$-elliptic.
Let $\mathcal T' \subset \mathcal T$ be the characteristic convex subset for the upper-support~$\operatorname{\overline{supp}} \mathcal L^+[\psi]$ of $\mathcal L^+[\psi]$ (defined at the end of Section~\ref{SubsecTopmost}) and~$[\psi']$ the restriction of~$[\psi]$ to the upper-support.
Independence of $\mathcal L_{\mathcal Z}^+[\psi]$ and~$\delta$ implies $\mathcal Z' \defeq \mathcal F[\mathcal T']$ is $\mathcal Y$-elliptic.
So the characteristic subforests of $(\mathcal Y, \delta)$ and $(\mathcal Y_\tau, d_\infty)$ for the upper-support~$\operatorname{\overline{supp}} \mathcal L^+[\psi]$ are expanding forests for~$[\psi']$ rel.~$\mathcal Z'$;
by Corollary~\ref{cor-unique2}, they are equivariantly homothetic and $\lambda[\tau] = \lambda$.
Thus the characteristic subforests of $(\mathcal Y, \delta)$ and $(\mathcal Y_\tau, c\,d_\infty)$ for $\operatorname{\underline{supp}}[\psi; {\mathcal Z}]$ are equivariantly isometric for some $c > 0$.
A minor modification of Lemma~\ref{lem-convergence3} implies $(\mathcal Y, \delta)$ is equivariantly isometric to $(\mathcal Y^*, c\, d_\infty \oplus \delta)$ --- details are left to the reader;
therefore, $(\mathcal Y, \delta)$ is equivariantly isometric to the dominating forest for~$[\psi]$ with respect to $\{ \mathcal A_i^{dom}[\psi] \}_{i=0}^k$.

\smallskip
Generally,~$[\psi]$ has a descending sequence of irreducible train tracks $(\tau_i \colon \mathcal T_i \to \mathcal T_i)_{i=1}^n$.
If $(\mathcal Y, \delta)$ is degenerate, then there is nothing to show.
Otherwise, the $\psi$-expanding homothety on~$(\mathcal Y, \delta)$ implies $\lambda[\tau_n] > 1$.
Set $\mathcal F_1 \defeq \mathcal F$ and $\mathcal F_{i+1} \defeq \mathcal F[\mathcal T_i]$.
The preceding discussion proves that the characteristic subforest of $(\mathcal Y, \delta)$ for~$\mathcal F_n$ is equivariantly isometric to some dominating forest for the restriction~$[\psi_n]$.
Lemma~\ref{lem-convergence4} implies $(\mathcal Y, \delta)$ is equivariantly isometric to some dominating forest for~$[\psi]$.
Conversely, it follows from Theorem~\ref{thm-dominatinglimitexist}(\ref{cond-dominating-expanding}) that the dominating forest for~$[\psi]$ with respect to a subset of $\psi_*$-orbits of dominating attracting laminations with the same stretch factor is an expanding forest for~$[\psi]$:

\begin{thm}\label{thm-unique}
An $\mathcal F$-forest $(\mathcal Y, \delta)$ is an expanding forest for an automorphism $\psi \colon \mathcal F \to \mathcal F$ if and only if it is equivariantly isometric to the dominating forest for~$[\psi]$ with respect to a subset of $\psi_*$-orbits of dominating attracting laminations with the same stretch factor. \qed
\end{thm}

\appendix
\section{Recognizing and centralizing atoroidal automorphisms}\label{SecAtoroidal}
For a given outer automorphism, restrict it to point stabilizers of a complete topmost tree and inductively construct the descending sequence of complete topmost forests.
The blow-up construction applied to this descending sequence produces the universal topmost pseudotree (whose underlying pretree is the limit pretree).
For an application of this universal construction, we prove a recognition theorem for atoroidal outer automorphisms.

\begin{cor} If~$[\phi]$ and~$[\psi]$ are  atoroidal outer automorphisms of~$F$ with the same universal topmost pseudotree, and the pseudotree admits a $\phi\psi^{-1}$-equivariant isometry fixing each orbit of branches, then $[\phi] = [\psi]$.
\end{cor}

\noindent The hypothesis is akin to assuming two pseudo-Anosov mapping classes have the same stable measured foliation, stretch factor, and action on singular leaves.

\begin{proof}
Let $(T, (\oplus_{j=1}^{k_i} \delta_{i,j})_{i=1}^n)$ be the universal topmost pseudotree for $[\phi], [\psi]$ and denoted by~$\iota$ the $\phi\psi^{-1}$-equivariant isometry on $(T, (\oplus_{j=1}^{k_i} \delta_{i,j})_{i=1}^n)$ that fixes each orbit of branches.
Choose $\phi' \in [\phi]$ such that the $\phi'\psi^{-1}$-equivariant isometry~$\iota'$ on $(T, (\oplus_{j=1}^{k_i} \delta_{i,j})_{i=1}^n)$ fixes a branch point.
The $F$-action on the limit pretree~$T$ is free since~$[\phi]$ is atoroidal.
Adapting Kapovich--Lustig's Proposition~4.1 in~\cite{KL11} to pseudotrees, we conclude~$\iota'$ fixes all points of~$T$, i.e.~$\iota'$ is the identity map on~$T$ and $\phi' = \psi$;
therefore, $[\phi] = [\psi]$.
\end{proof}

We call this a recognition theorem because it lists a set of dynamical invariants (universal topmost pseudotree, stretch factors of the factored pseudometrics, and  action on orbits of branches) that determine an atoroidal outer automorphism.
Feighn--Handel's recognition theorem \cite[Theorem~5.3]{FH11} gives related dynamical invariants (attracting laminations, their stretch factors, non-repelling fixed points at infinity, and twist coordinates) that determine a {forward rotationless} outer automorphisms; 
their theorem can also be extended to all atoroidal outer automorphisms as in our corollary.

A minor change introducing {twist coordinates} extends our corollary (or Feighn--Handel's recognition theorem) to outer automorphisms whose limit pretrees have cyclic point stabilizers.
With more care, the corollary should generalize to outer automorphisms whose restrictions to the point stabilizers of limit pretrees is linearly growing --- linearly growing automorphisms have canonical representatives~\cite{CL95}.
Generalizing to all outer automorphisms would require having canonical nondegenerate representatives for all polynomially growing automorphisms.

\begin{cor}\label{cor-central}If $\phi \colon F \to F$ is an atoroidal automorphism, then the centralizer of~$[\phi]$ in the outer automorphism group $\operatorname{Out}(F)$ is virtually a free abelian group with rank at most the number of $[\phi]$-orbits of attracting laminations for~$[\phi]$.
\end{cor}

\noindent
Feighn--Handel do not explicitly state this corollary, but it follows from \cite[Lemma 5.5]{FH09}.
Bestvina--Feighn--Handel previously proved that centralizers of {fully irreducible} outer automorphisms are virtually cyclic~\cite[Theorem~2.14]{BFH97}.
In the first version of this paper, we claimed Corollary~\ref{cor-central} as a new result, and
a referee told us the corollary follows from Feighn--Handel's work on CT maps.
Our new proof uses the universal topmost pseudotree.

\begin{proof}
Let $(T, (\oplus_{j=1}^{k_i} \delta_{i,j})_{i=1}^n)$ be the universal topmost pseudotree for~$[\phi]$,~$C[\phi]$ the centralizer for~$[\phi]$ in $\operatorname{Out}(F)$, and $k \defeq \sum_{i=1}^n k_i$.
Replace $C[\phi]$ with a finite index subgroup and assume  it acts trivially on the attracting laminations for~$[\phi]$.
If $[\phi'] \in C[\phi]$, then the universal pseudotree supports a $\phi'$-equivariant dilation by uniqueness of the pseudotree for~$[\phi]$.
Thus we can define a group homomorphism $\ell \colon C[\phi] \to \mathbb R_{>0}^k$ that maps~$[\phi']$ to $(\lambda_{i,j}'\,:\,1 \le i \le n, 1 \le j \le k_i)$.
The image of~$C[\phi]$ under each coordinate projection~$\ell_{i,j}$ of~$\ell$ is a cyclic subgroup of~$\mathbb R_{>0}$ by Corollary~\ref{cor-cyclicimage}.

By index theory, we can replace~$C[\phi]$ with a finite index subgroup again and assume it fixes the orbits of branches in~$T$.
As the $F$-action on~$T$ is free, the kernel~$\ker(\ell)$ is trivial --- see Proposition~4.2 in~\cite{KL11}.
So~$C[\phi]$ is free abelian with rank $\le k$.
\end{proof}

Again, the corollary can be adapted to work for outer automorphisms whose limit pretrees have cyclic point stabilizers.
Yassine Guerch recently gave another proof of this more general statement using different methods~\cite[Theorem~5.3]{Gue22}.
With more care, our work or Feighn--Handel's can combine with Andrew--Martino's paper~\cite[Theorem~1.5]{AM22} to characterize the centralizer of an outer automorphism whose restriction to point stabilizers of limit pretrees is linearly growing.

We think it is open whether arbitrary centralizers are finitely generated.
For a complete description of arbitrary centralizers, one needs canonical nondegenerate representatives for polynomially growing automorphisms.
Presumably, a polynomially growing automorphism of degree $d \ge 2$ has a canonical fixed free splitting whose loxodromics are exactly the elements that grow with degree~$d$.

\bibliography{zrefs}
\bibliographystyle{plain}

\end{document}